\documentclass[11pt]{article}

\usepackage{amsfonts}
\usepackage{amscd}
\usepackage{amssymb}
\usepackage{amsthm}
\usepackage{amsmath, xspace}
\usepackage{blkarray}
\usepackage{cancel}
\usepackage{color}
\usepackage{graphics}
\usepackage{graphicx}
\usepackage{enumitem}
\usepackage{fancyhdr}
\usepackage{mathdots}
\usepackage{mathrsfs}
\usepackage{multicol}
\usepackage{stmaryrd}
\usepackage{ytableau}
\usepackage[all]{xy}

\usepackage[plainpages,backref]{hyperref}

\usepackage{lscape}

\theoremstyle{plain}
\newtheorem{thm}{Theorem}[section]

\newtheorem{prop}[thm]{Proposition}

\theoremstyle{definition}

\theoremstyle{remark}

\numberwithin{equation}{section}

\providecommand{\keywords}[1]{\textbf{\textit{Key words---}} #1}

\newcommand{\btimes}{\mathbin{\rotatebox[origin=c]{90}{$\ltimes$}}}

\def\cH{\mathcal{H}}

\def\CC{\mathbb{C}}

\def\RR{\mathbb{R}}

\def\ZZ{\mathbb{Z}}

\def\fa{\mathfrak{a}}

\def\fh{\mathfrak{h}}

\def\Card{\mathrm{Card}}

\def\id{\mathrm{id}}

\def\wt{\mathrm{wt}}

\usepackage{etex} 
\usepackage{pictexwd}
\usepackage{tikz}
\usepgflibrary{shapes.geometric}
\usetikzlibrary{arrows, calc, positioning}

\tikzstyle{V}=[draw, fill =black, circle, inner sep=0pt, minimum size=1.5pt]
\tikzstyle{wV}=[draw, fill =white, circle, inner sep=0pt, minimum size=4.5pt]
\tikzstyle{bV}=[draw, fill =black, circle, inner sep=0pt, minimum size=4.5pt]
\tikzstyle{over}=[draw=white,double=black,line width=2pt, double distance=.5pt]

\def\Over[#1,#2][#3,#4]{ 
	\draw[style=over]   (#1,#2) .. controls ++(0,#4*.5-#2*.5) and ++(0,-#4*.5+#2*.5) .. (#3,#4);}
\def\Under[#1,#2][#3,#4]{ 
	\draw  (#1,#2) .. controls ++(0,#4*.5-#2*.5) and ++(0,-#4*.5+#2*.5) .. (#3,#4);}
\def\Cross[#1,#2][#3,#4]{
	\Under[#3,#2][#1,#4]\Over[#1,#2][#3,#4]}

\def\Tops[#1][#2][#3]{
	\foreach\x in {#1}{
		\draw (\x+.1,#2) -- (\x+.1,#2+.15) (\x-.1,#2) -- (\x-.1,#2+.15) ;
		\draw (\x+.1,#2+.15) arc (0:360:1mm and .5mm);}
	\foreach \x in {1,...,#3} {\draw (\x,#2)  to (\x,#2+.05) node[V]{};}
	}
\def\Bottoms[#1][#2][#3]{
	\foreach\x in {#1}{
		\draw (\x+.1,#2) -- (\x+.1,#2-.1) (\x-.1,#2) -- (\x-.1,#2-.1) ;
		\draw (\x+.1,#2-.1) arc (0:-180:1mm and .5mm);}
	\foreach \x in {1,...,#3} {\draw (\x,#2)  to (\x,#2-.05) node[V]{};}
	}
\def\Caps[#1][#2,#3][#4]{
	\Tops[#1][#3][#4]
	\Bottoms[#1][#2][#4]
	}
\def\Pole[#1][#2,#3]{
	\shade[left color=white,right color=white] (#1+.1,#2) rectangle (#1-.1,#3);
	\draw[over] (#1+.1,#2) to (#1+.1,#3) (#1-.1,#2) to (#1-.1,#3) ;}
\def\Label[#1,#2][#3][#4]{
	\node[above] at (#3,#2+.1) {#4};
	\node[below] at (#3,#1-.1) {#4};		}

\newcommand{\posleq}[1]{
	\hspace{0.1cm}
	\begin{tikzpicture}
	\draw (-0.8ex, -0.5ex) -- (0.8ex, -0.5ex);
	\draw (-0.8ex, 0.4ex) -- (0.7ex, -0.2ex);
	\draw (-0.8ex, 0.4ex) -- (0.7ex, 1ex);
	\draw (0.4ex,0.4ex) --(1.1ex, 0.4ex);
	\draw (0.75ex,0.75ex) --(0.75ex, 0.05ex);
	\end{tikzpicture}
	\hspace{0.1cm}
	}
\newcommand{\negleq}[1]{
	\hspace{0.1cm}
	\begin{tikzpicture}
	\draw (-0.8ex, -0.5ex) -- (0.8ex, -0.5ex);
	\draw (-0.8ex, 0.4ex) -- (0.7ex, -0.2ex);
	\draw (-0.8ex, 0.4ex) -- (0.7ex, 1ex);
	\draw (0.4ex,0.4ex) --(1.1ex, 0.4ex);
	\end{tikzpicture}
	\hspace{0.1cm}
	}
	
\newcommand{\zeroleq}[1]{
	\hspace{0.1cm}
	\begin{tikzpicture}
	\draw (-0.8ex, -0.5ex) -- (0.8ex, -0.5ex);
	\draw (-0.8ex, 0.4ex) -- (0.7ex, -0.2ex);
	\draw (-0.8ex, 0.4ex) -- (0.7ex, 1ex);
	\draw  (0.75ex,0.4ex) ellipse (0.2ex and 0.35ex);
	\end{tikzpicture}
	\hspace{0.1cm}
	}
	
\newcommand{\posgeq}[1]{
	\hspace{0.1cm}
	\begin{tikzpicture}
	\draw (-0.8ex, -0.5ex) -- (0.8ex, -0.5ex);
	\draw (0.8ex, 0.4ex) -- (-0.7ex, -0.2ex);
	\draw (0.8ex, 0.4ex) -- (-0.7ex, 1ex);
	\draw (-0.4ex,0.4ex) --(-1.1ex, 0.4ex);
	\draw (-0.75ex,0.75ex) --(-0.75ex, 0.05ex);
	\end{tikzpicture}
	\hspace{0.1cm}
	}
\newcommand{\neggeq}[1]{
	\hspace{0.1cm}
	\begin{tikzpicture}
	\draw (-0.8ex, -0.5ex) -- (0.8ex, -0.5ex);
	\draw (0.8ex, 0.4ex) -- (-0.7ex, -0.2ex);
	\draw (0.8ex, 0.4ex) -- (-0.7ex, 1ex);
	\draw (-0.4ex,0.4ex) --(-1.1ex, 0.4ex);
	\end{tikzpicture}
	\hspace{0.1cm}
	}
	
\newcommand{\zerogeq}[1]{
	\hspace{0.1cm}
	\begin{tikzpicture}
	\draw (-0.8ex, -0.5ex) -- (0.8ex, -0.5ex);
	\draw (0.8ex, 0.4ex) -- (-0.7ex, -0.2ex);
	\draw (0.8ex, 0.4ex) -- (-0.7ex, 1ex);
	\draw  (-0.75ex,0.4ex) ellipse (0.2ex and 0.35ex);
	\end{tikzpicture}
	\hspace{0.1cm}
	}

\newcommand{\posl}[1]{
	\hspace{0.1cm}
	\begin{tikzpicture}
	\draw (-0.8ex, 0.4ex) -- (0.7ex, -0.2ex);
	\draw (-0.8ex, 0.4ex) -- (0.7ex, 1ex);
	\draw (0.4ex,0.4ex) --(1.1ex, 0.4ex);
	\draw (0.75ex,0.75ex) --(0.75ex, 0.05ex);
	\end{tikzpicture}
	\hspace{0.1cm}
	}
\newcommand{\negl}[1]{
	\hspace{0.1cm}
	\begin{tikzpicture}
	\draw (-0.8ex, 0.4ex) -- (0.7ex, -0.2ex);
	\draw (-0.8ex, 0.4ex) -- (0.7ex, 1ex);
	\draw (0.4ex,0.4ex) --(1.1ex, 0.4ex);
	\end{tikzpicture}
	\hspace{0.1cm}
	}
	
\newcommand{\zerol}[1]{
	\hspace{0.1cm}
	\begin{tikzpicture}
	\draw (-0.8ex, 0.4ex) -- (0.7ex, -0.2ex);
	\draw (-0.8ex, 0.4ex) -- (0.7ex, 1ex);
	\draw  (0.75ex,0.4ex) ellipse (0.2ex and 0.35ex);
	\end{tikzpicture}
	\hspace{0.1cm}
	}
	
\newcommand{\posg}[1]{
	\hspace{0.1cm}
	\begin{tikzpicture}
	\draw (0.8ex, 0.4ex) -- (-0.7ex, 1ex);
	\draw (0.8ex, 0.4ex) -- (-0.7ex, -0.2ex);
	\draw (-0.4ex,0.4ex) --(-1.1ex, 0.4ex);
	\draw (-0.75ex,0.75ex) --(-0.75ex, 0.05ex);
	\end{tikzpicture}
	\hspace{0.1cm}
	}
\newcommand{\negg}[1]{
	\hspace{0.1cm}
	\begin{tikzpicture}
	\draw (0.8ex, 0.4ex) -- (-0.7ex, -0.2ex);
	\draw (0.8ex, 0.4ex) -- (-0.7ex, 1ex);
	\draw (-0.4ex,0.4ex) --(-1.1ex, 0.4ex);
	\end{tikzpicture}
	\hspace{0.1cm}
	}
	
\newcommand{\zerog}[1]{
	\hspace{0.1cm}
	\begin{tikzpicture}
	\draw (0.8ex, 0.4ex) -- (-0.7ex, -0.2ex);
	\draw (0.8ex, 0.4ex) -- (-0.7ex, 1ex);
	\draw  (-0.75ex,0.4ex) ellipse (0.2ex and 0.35ex);
	\end{tikzpicture}
	\hspace{0.1cm}
	}

\makeatletter
\renewcommand{\@makefnmark}{\mbox{\textsuperscript{}}}
\makeatother

\title{Comparing formulas for type $GL_n$ Macdonald polynomials \\
Supplement}
\author{
Weiying Guo\quad\ \ email:\ guwg@student.unimelb.edu.au \\
Arun Ram\quad\ \ email:\ aram@unimelb.edu.au \\
\\
}
\date{}

\usetikzlibrary{arrows.meta}

\begin{document}

\maketitle
\vspace{-2em}
\begin{center}
{\sl Dedicated to H\'el\`ene Barcelo}
\end{center}

\begin{abstract}
\noindent
This paper is a supplement to \cite{GR21}, containing examples, remarks and
additional material that could be useful to researchers working with Type 
$GL_n$ Macdonald polynomials.  In the course of our comparison of the alcove walk
formula and the nonattacking fillings formulas for type $GL_n$ Macdonald polynomials
we did many examples and significant analysis of the literature.  In the preparation
of \cite{GR21} it seemed sensible to produce a document with focus and this material was
removed.  This is paper resurrects and organizes that material, in hopes that others
may also find it useful.
\end{abstract}


\setcounter{section}{-1}
\tableofcontents

\section{Introduction}

This paper is a supplement to \cite{GR21}, containing examples, remarks and
additional material that could be useful to researchers working with Type 
$GL_n$ Macdonald polynomials.  In the course of our comparison of the alcove walk
formula and the nonattacking fillings formulas for type $GL_n$ Macdonald polynomials
we did many examples and significant analysis of the literature.  In the preparation
of \cite{GR21} it seemed sensible to produce a document with focus and this material was
removed.  This is paper resurrects and organizes that material, in hopes that others
may also find it useful.
\begin{enumerate}
\item[1.] The material in Section 1:
Several colleagues have asked us questions about permuted basement 
Macdonald polynomials and KZ-families (the permuted basement Macdonald polynomials
are called relative Macdonald polynomials in this paper).  These questions
are helpfully considered in the context of the results of the two paragraphs
following equation (6.6) in Macdonald's S\'eminaire Bourbaki article \cite{Mac95} and
Sections 5.4 and 5.5 of Macdonald's followup book \cite{Mac03} treating the
fully general case.  In hopes of making these results more accessible, in Section 1 we have 
recast these completely in the type $GL_n$ and included their proofs (which are not difficult).
These results are
the $H$-decomposition in Section \ref{Hdecomp}, 
symmetrization statement in Proposition \ref{PinEzmu}, 
and the KZ-family characterization in 
Proposition \ref{KZfamilyczn}.  We hope that these type $GL_n$ specific expositions of these
results can be helpful to the community.

\item[2.] The material in Section 2:  This section has a focus on counting the number of alcove walks
and the number of nonattacking fillings, in order to compare the number of terms
that appear in alcove walks formula and the nonattacking fillings formula for Macdonald polynomials.
Some explicit formulas for these counts, which may not have been widely noticed before, are included.

\item[3.] The material in Section 3:  This section explains how to recast the alcove walks
and nonattacking fillings into path form and pipe dream form.  Pictures are provided.

\item[4,5,6.] The material in Sections 4, 5 and 6:  These sections provide explicit examples of the
main results of \cite{GR21}:  the inversions and the box-greedy reduced word for $u_\mu$
proved in \cite[Proposition 2.2]{GR21}, the step-by-step and box-by-box recursions for computing
Macdonald polynomials in \cite[Proposition 4.1 and 4.3]{GR21} and some specific examples
to help support the exposition of the type $GL_n$ double affine Hecke algebra (DAHA) given in 
\cite[Section 5]{GR21}.
\item[7.] The material in Section 7:  In this final section we provide additional explicit expansions of
Macdonald polynomials for special cases: $n=2$, $n=3$, a single column, partitions with 3 boxes,
and explicit nonattacking fillings and their weights for $E_\mu$ where $\mu$ has less than 3 boxes.
\item[8.] Section 8 contains some brief remarks about the queue tableaux and multiline queues
which appear in \cite[Section 1.2 and Definition A.2]{CMW18}.
\end{enumerate}

A small warning: Even though they all have a Type A root system,
type $SL_n$ Macdonald polynomials, type $PGL_n$ Macdonald polynomals
and   type $GL_n$ Macdonald
polynomials are all \emph{different} 
(though the relationship is well known and not difficult).  We should
stress that this paper is specific to the $GL_n$-case and some results of this paper
do not hold for Type $SL_n$ or type $PGL_n$ unless properly modified.

We thank L.\ Williams and M.\ Wheeler for bringing our attention to \cite{CMW18}
and \cite{BW19}, both of which were important stimuli during our work.  We are also very grateful 
for the encouragement, questions, and discussions from A.\ Hicks, S.\ Mason, O.\ Mandelshtam,
Z.\ Daugherty,  Y.\ Naqvi, S.\ Assaf, and especially A.\ Garsia and S.\ Corteel,
which helped so much in getting going and keeping up the energy.  We thank
S.\ Billey, Z.\ Daugherty, C.\ Lenart and J.\ Saied 
for very useful specific comments for improving 
the exposition.  A.\ Ram extends a very special and heartfelt
thank you to P.\ Diaconis who has provided 
unfailing support and advice and honesty and encouragement.

\section{Symmetrization, $H$ decomposition of $\CC[X]$ and KZ-families}

Let $q, t^{\frac12}\in \CC^\times$.  Following the notation of \cite[Ch.\ VI (3.1)]{Mac},
let $T_{q^{-1},x_1}$ be the operator on $\CC[x_1^{\pm1}, \ldots, x_n^{\pm1}]$ given by
$$T_{q^{-1},x_n} h(x_1, \ldots, x_n) = h(x_1, \ldots, x_{n-1}, q^{-1}x_n).$$
The symmetric group $S_n$ acts on $\CC[x_1^{\pm1}, \ldots, x_n^{\pm1}]$ by permuting
the the variables $x_1, \ldots, x_n$.
Define operators $T_1, \ldots, T_{n-1}$, $g$ and $g^\vee$ on 
$\CC[x_1^{\pm1}, \ldots, x_n^{\pm1}]$ by 
\begin{equation}
T_i =t^{-\frac12} \big( t -\frac{tx_i - x_{i+1}}{x_i-x_{i+1}}(1-s_i) \big),
\qquad
g = s_1s_2\cdots s_{n-1} T_{q^{-1}, x_n}, \qquad
g^\vee = x_1T_1\cdots T_{n-1},
\label{DAHAonCX}
\end{equation}
where $s_1, \ldots, s_{n-1}$ are the simple transpositions in $S_n$.
The \emph{Cherednik-Dunkl operators} are 
\begin{equation}
Y_1 = g T_{n-1}\cdots T_1,
\quad Y_2 = T_1^{-1}Y_1T_1^{-1}, \quad Y_3 = T_2^{-1}Y_2T_2^{-1}, \quad \ldots,\quad
Y_n= T_{n-1}^{-1}Y_{n-1}T_n^{-1}.
\label{GLnCDopsdefn}
\end{equation}
For $\mu\in \ZZ^n$ the \emph{nonsymmetric Macdonald polynomial $E_\mu$} is the
(unique) element $E_\mu\in \CC[x_1^{\pm1}, \ldots, x_n^{\pm1}]$ such that
\begin{equation}Y_i E_\mu = q^{-\mu_i}t^{-(v_\mu(i)-1)+\frac12(n-1)}E_\mu,
\qquad\hbox{and the coefficient of $x_1^{\mu_1}\cdots x_n^{\mu_n}$ in $E_\mu$ is $1$,}
\label{Emueigenvalue}
\end{equation}
where $v_\mu\in S_n$ is the minimal length permutation such that $v_\mu\mu$ is weakly increasing.
\hfil\break
Let $\mu = (\mu_1, \ldots, \mu_n)$ and let $z\in S_n$. 
\begin{equation}
\hbox{The \emph{relative Macdonald polynomial} $E_\mu^z$  is}\qquad
E^z_\mu = t^{-\frac12(\ell(zv_\mu^{-1})-\ell(v_\mu^{-1}))}T_zE_\mu.
\label{Ezmudefn}
\end{equation}
Let $\lambda = (\lambda_1\ge \cdots \ge \lambda_n)\in \ZZ^n$.  
\begin{equation}
\hbox{The \emph{symmetric Macdonald polynomial} $P_\lambda$ is}\qquad
P_\lambda = \sum_{\nu\in S_n\lambda} t^{\frac12\ell(z_\nu)}T_{z_\nu} E_\lambda,
\label{Pdefn}
\end{equation}
where the sum is over rearrangements $\nu$ of $\lambda$ and
$z_\nu\in S_n$ is minimal length such that $\nu = z_\nu\lambda$.

\subsection{The $H$-modules $\CC[X]^\lambda$}\label{Hdecomp}

Let $H$ be the algebra generated by the operators $T_1, \ldots, T_{n-1}$
and $Y_1, \ldots, Y_n$ (so that $H$ is an affine Hecke algebra) and let
$$\tau_i^\vee = T_i + \frac{t^{-\frac12}(1-t)}{1-Y^{-1}_i Y_{i+1}}\quad\hbox{for $i\in \{1, \ldots, n-1\}$.}
$$
As $H$-modules
$$\CC[x_1^{\pm1}, \ldots, x_n^{\pm1}] = \bigoplus_{\lambda} \CC[X]^\lambda
\qquad\hbox{where}\qquad
\CC[X]^\lambda 
= \hbox{span}\{ E_\mu\ |\ \mu\in S_n\lambda\},
$$
and the direct sum is over decreasing $\lambda = (\lambda_1\ge \cdots \ge \lambda_n)\in \ZZ^n$.
A description of the action of $H$ on $\CC[X]^\lambda$ is given by the following.
Let $\mu\in \ZZ^n$ and, with notations as in \eqref{Emueigenvalue},  let
$$
\begin{array}{l}
a_\mu =q^{\mu_i-\mu_{i+1}}t^{v_\mu(i)-v_\mu(i+1)} , \\
a_{s_i\mu} =q^{\mu_{i+1}-\mu_i} t^{v_\mu(i+1)-v_\mu(i)} ,
\end{array}
\qquad\hbox{and}\qquad
D_\mu = \frac{(1-ta_\mu)(1-ta_{s_i\mu})}{(1-a_\mu)(1-a_{s_i\mu})}.
$$
Assume that $\mu_i>\mu_{i+1}$.
By using the identity $E_{s_i\nu} = t^{\frac12}\tau^\vee_iE_\nu$ if $\nu_{i+1}>\nu_i$ from
\cite[(3.5)]{GR21}, the eigenvalue  from \eqref{Emueigenvalue}
and \cite[Proposition 5.5 (5.23)]{GR21}, it is straightforward to compute that
\begin{equation}
\begin{array}{l}
Y_i^{-1}Y_{i+1} E_\mu = a_\mu E_\mu, \\
Y_i^{-1}Y_{i+1} E_{s_i\mu} = a_{s_i\mu}E_{s_i\mu},
\end{array}
\quad
\begin{array}{l}
t^{\frac12}\tau^\vee_i E_\mu = E_{s_i\mu}, \\
t^{\frac12}\tau^\vee_i E_{s_i\mu} = D_\mu E_\mu,
\end{array}
\quad\hbox{and}
\quad
\begin{array}{l}
t^{\frac12}T_i E_\mu = -\frac{1-t}{1-a_\mu}E_\mu + E_{s_i\mu}, \\
t^{\frac12}T_i E_{s_i\mu} = D_\mu E_\mu + \frac{1-t}{1-a_{s_i\mu}} E_{s_i\mu}.
\end{array}
\label{CXlambdaaction}
\end{equation}
Now assume that $\mu_i=\mu_{i+1}$.  Then
$v_\mu(i+1) = v_\mu(i)+1$ and $a_\mu = t^{-1}$ so that
\begin{equation}
Y_i^{-1}Y_{i+1} E_\mu = t^{-1} E_\mu, 
\qquad (t^{\frac12}\tau^\vee_i) E_\mu = 0,
\qquad\hbox{and}\qquad (t^{\frac12}T_i) E_\mu = t E_\mu.
\label{Tigivest}
\end{equation}
These formulas make explicit the action of $H$ on $\CC[X]^\lambda$
in the basis $\{ E_\mu\ |\ \mu\in S_n \lambda\}$.  The formulas in \eqref{CXlambdaaction}
are the type $GL_n$ special cases of \cite[(5.4.3),(5.6.6)]{Mac03}.

\subsection{Symmetrization of $E_\mu$ for $\mu\in \ZZ^n$}

If $z\in S_n$ and 
$$\hbox{$z = s_{i_1}\cdots s_{i_\ell}$ is a reduced word,
\qquad let}\qquad T_z = T_{i_1}\cdots T_{i_\ell}.$$
Let $w_0$ be the longest element of $S_n$ so that
$$w_0(i) = n-i+1,\ \hbox{for $i\in \{1, \ldots, n\}$,}\qquad
\hbox{and}\qquad \ell(w_0) =\frac{n(n-1)}{2} = \binom{n}{2}.
$$
Following \cite[(5.5.7), (5.5.16), (5.5.17)]{Mac03}, let
\begin{equation}
\mathbf{1_0} = t^{-\frac12\ell(w_0)} \sum_{z\in S_n} t^{\frac12\ell(z)} T_z,
\qquad\hbox{so that}\qquad
T_i\mathbf{1}_0 = \mathbf{1}_0T_i = t^{\frac12}\mathbf{1}_0
\quad\hbox{for $i\in \{1, \ldots, n-1\}$,}
\label{symmetrizer}
\end{equation}
and 
\begin{equation}
\mathbf{1}_0^2 = W_0(t)\mathbf{1}_0,
\qquad\hbox{where}\quad
W_0(t) = \sum_{z\in S_n} t^{\ell(z)}
\label{Poincarepoly}
\end{equation}
is the \emph{Poincar\'e polynomial for $S_n$.}

For $\mu\in \ZZ^n$, the \emph{symmetrization of $E_\mu$} is (see \cite[(5.7.1)]{Mac03} and \cite[Remarks after (6.8)]{Mac95})
\begin{equation}
F_\mu = \mathbf{1}_0E_\mu = t^{-\frac12\ell(w_0)}
\sum_{z\in S_n} t^{\frac12(\ell(z)-\ell(zv^{-1}_\mu)+\ell(v^{-1}_\mu)}
E^z_\mu,
\label{Fmudefn}
\end{equation}
so that $F_\mu$ is a (weighted) sum of the relative Macdonald polynomials 
$E^z_\mu$ defined in \eqref{Ezmudefn}).
%
%
%
The following Proposition
shows that $F_\mu$ is always, up to an explicit constant factor,
equal to the symmetric Macdonald polynomial $P_\lambda$ (defined in \eqref{Pdefn}).
Proposition \ref{PinEzmu} is the specialization of 
\cite[remarks after (6.8)]{Mac95} and \cite[(5.7.2)]{Mac03} to our setting.

\begin{prop} \label{PinEzmu}
Let $\mu=(\mu_1, \ldots, \mu_n)\in \ZZ^n$.
Let $\lambda = (\lambda_1, \ldots, \lambda_n)$ be the weakly decreasing rearrangement of $\mu$
and let $z_\mu\in S_n$ be minimal length such that $\mu = z_\mu \lambda$.
Let
$$S_\lambda = \{ y\in S_n\ |\ y\lambda = \lambda\}
\qquad\hbox{and}\qquad
W_\lambda(t) = \sum_{y\in S_\lambda} t^{\ell(y)}.
$$
Then
$$P_\lambda = \frac{t^{\frac12\ell(w_0)} }{W_\lambda(t) }
\Big(
\prod_{(i,j)\in \mathrm{Inv}(z_\mu)} \frac{1- q^{\lambda_i-\lambda_j}t^{j-i} }
{1- q^{\lambda_i-\lambda_j}t^{j-i+1} } \Big) F_\mu.$$
\end{prop}
\begin{proof}  The proof is by induction on $\ell(z_\mu)$.
The base case $z_\mu=1$ has $\mu= \lambda$ and $v_\lambda = w_0z_\lambda$ so that
\begin{align*}
F_\lambda 
&= \mathbf{1}_0E_\lambda
= t^{-\frac12\ell(w_0)}
\Big(\sum_{u\in S_n/S_\lambda} \sum_{v\in S_\lambda} t^{\frac12\ell(x)+\ell(y)} T_x T_y\Big)
E_\lambda \\
&= t^{-\frac12\ell(w_0)}\Big(\sum_{u\in S_n/S_\lambda} t^{\frac12\ell(x)} T_x\Big)W_\lambda(t) E_\lambda 
= t^{-\frac12\ell(w_0)} W_\lambda(t) P_\lambda,
\end{align*}
where $T_yE_\lambda = t^{\frac12\ell(y)}E_y$ is a consequence of \eqref{Tigivest}
and the last equality is \eqref{Pdefn}.
For the induction step, assume that
$\mu$ is not weakly decreasing and let $i\in \{1, \ldots, n-1\}$ be such that $\mu_i<\mu_{i+1}$.
Then $z_{s_i\mu} = s_iz_\mu$ and $\ell(z_{s_i\mu})= \ell(z_\mu)-1$.
Using $E_\mu = t^{\frac12}\tau_i^\vee E_{s_i\mu}$ 
and $\mathbf{1}_0 T_i = \mathbf{1}_0 t^{\frac12}$ from \eqref{CXlambdaaction} and
\eqref{Tigivest} gives
\begin{align*}
F_{\mu} 
&= \mathbf{1}_0  E_\mu 
= \mathbf{1}_0 t^{\frac12}\tau_{i_1}E_{s_i\mu}
= \mathbf{1}_0 \Big(t^{\frac12}T_i + \frac{1-t}{1-Y_i^{-1}Y_{i+1} } \Big) E_{s_i\mu}
= \mathbf{1}_0 \Big(t + \frac{1-t}{1-Y_i^{-1}Y_{i+1} } \Big) E_{s_i\mu} \\
&= \mathbf{1}_0 \frac{ 1-tY_i^{-1}Y_{i+1} }{1- Y_i^{-1}Y_{i+1} }  E_{s_i\mu}
= \mathbf{1}_0 \frac{1-t q^{\mu_{i+1}-\mu_i} t^{v_\mu(i+1)-v_\mu(i) }  } 
{1- q^{\mu_{i+1} - \mu_i}  t^{v_\mu(i+1)-v_\mu(i)} } E_{s_i\mu} 
= \frac{1- q^{\mu_{i+1}-\mu_i} t^{v_\mu(i+1)-v_\mu(i)+1 }  } 
{1- q^{\mu_{i+1} - \mu_i}  t^{v_\mu(i+1)-v_\mu(i)} } F_{s_i\mu} 
\end{align*}
and the result follows by induction (see Section \ref{Symmexample} for an example).
\end{proof}

\subsection{The KZ-family basis of $\CC[X]^\lambda$}

For $\mu\in \ZZ^n$, let $\lambda=(\lambda_1\ge \cdots \ge\lambda_n)$ be the decreasing
rearrangement of $\mu$ and let $z_\mu\in S_n$ be minimal length such that $\mu = z_\mu\lambda$.
Define
\begin{equation}
f_\mu = E^{z_\mu}_\lambda = t^{\frac12\ell(z_\mu)}T_{z_\mu} E_\lambda.
\end{equation}
It follows from the identities in the last column of \eqref{CXlambdaaction} that
$$\hbox{$\{ f_\mu\ |\ \mu\in S_n\lambda\}$\quad is another basis of 
$\CC[X]^\lambda$.}
$$
The following Proposition says that the $\{f_\mu\ |\ \mu\in \ZZ^n\}$ form a KZ-family, 
in the terminology of \cite[Def.\ 3.3]{KT06} (see also \cite[Def.\ 1.13]{CMW18}, 
\cite[(17), (18), (19)]{CdGW15}, \cite[Def.\ 2]{CdGW16}).

\begin{prop}  \label{KZfamilyczn}
Let $\mu = (\mu_1, \ldots, \mu_n)\in \ZZ_{\ge 0}^n$.
Let $i\in \{1, \ldots, n-1\}$ and let $T_i$ and $g$ be as defined in \eqref{DAHAonCX}.
Then
$$
t^{\frac12}T_if_\mu = \begin{cases}
f_{s_i\mu}, &\hbox{if $\mu_i>\mu_{i+1}$,} \\
tf_\mu, &\hbox{if $\mu_i = \mu_{i+1}$,}
\end{cases}
\qquad\hbox{and}\qquad
g f_\mu = q^{-\mu_n } f_{(\mu_n,\mu_1, \ldots, \mu_{n-1})}.
$$
\end{prop}
\begin{proof}  Assume $\mu_i>\mu_{i+1}$. Then $z_{s_i\mu} = s_iz_\mu$
and $\ell(z_{s_i\mu}) = \ell(z_\mu)+1$ so that
$$t^{\frac12} T_i f_\mu 
= t^{\frac12} T_i t^{\frac12\ell(z_\mu)} T_{z_\mu} E_\lambda
= t^{\frac12\ell(z_{s_i\mu})} T_{z_{s_i\mu}} E_\lambda = f_{s_i\mu}.
$$
Assume $\mu_i=\mu_{i+1}$. Then there exists $j\in \{1, \ldots, n-1\}$ 
such that $s_j\lambda = \lambda$ and $s_iz_\mu = z_\mu s_j$ 
(so that $s_i\mu = s_iz_\mu\lambda = z_\mu s_j \lambda$).  Then
$$t^{\frac12} T_i f_\mu 
= t^{\frac12} T_i t^{\frac12\ell(z_\mu)} T_{z_\mu} E_\lambda
=  t^{\frac12\ell(z_\mu)} T_{z_\mu} t^{\frac12} T_j E_\lambda
=  t^{\frac12\ell(z_\mu)} T_{z_\mu} t E_\lambda
=t f_\mu.
$$
(c) Let $\mu=(\mu_1, \ldots, \mu_n)$ and let $i$ and $j$ be such that $\lambda_i$ is the first part of $\lambda$ equal to $\mu_n$ and $\lambda_j$ is the last part of $\lambda$ equal to $\mu_n$.
Thus $\mu_n = \lambda_i=\lambda_{i+1}=\cdots = \lambda_j$.  
Write $z_\mu = z s_{n-1}\cdots s_j$ with $z\in S_{n-1}$ and let
$c_n = s_1\cdots s_{n-1}$.  Then,
using
$v_\lambda(j) = 1+ (j-i) + n-j = n-i+1$ from \cite[Proposition 2.1(a)]{GR21}, 
\begin{align*}
gf_\mu 
&= g t^{\frac12\ell(z_\mu)}T_{z_\mu} E_\lambda 
= g t^{\frac12\ell(z)}T_z t^{\frac12(n-j)} T_{n-1}\cdots T_j E_\lambda
= t^{\frac12(n-j)}  gt^{\frac12\ell(z)}T_zg^{-1}gT_{n-1}\cdots T_jE_\lambda \\
&=t^{\frac12(n-j)} (gt^{\frac12\ell(z)}T_zg^{-1} )
T_1\cdots T_{j-1} (T_{j-1}^{-1}\cdots T_1^{-1}gT_{n-1}\cdots T_j)E_\lambda  \\
&= t^{\frac12(n-j)}
(t^{\frac12\ell(z)} T_{c_nzc_n^{-1}}) T_1\cdots T_{j-1} Y_j E_\lambda  \\
&= t^{\frac12(n-j)}
(t^{\frac12\ell(z)} T_{c_nzc_n^{-1}}) T_1\cdots T_{j-1}  
q^{-\lambda_j}t^{-(v_\lambda(j)-1)+\frac12(n-1)}E_\lambda \\
&=q^{-\lambda_j } t^{\frac12(n-j) - (n-i+1-1) +\frac12(n-1)}  (t^{\frac12\ell(z)} T_{c_nzc_n^{-1}})
T_1\cdots T_{i-1} T_i\cdots T_{j-1} E_\lambda \\
&= 
q^{-\mu_n }t^{-\frac12 j + i - \frac12} 
(t^{\frac12\ell(z)}T_{c_nzc_n^{-1}}) T_1\cdots T_{i-1}  t^{\frac12(j-i)}E_\lambda \\
&= 
q^{-\mu_n } (t^{\frac12\ell(z)}T_{c_nzc_n^{-1}}) t^{\frac12(i-1)}T_1\cdots T_{i-1} E_\lambda 
= 
q^{-\mu_n } f_{(\lambda_i,\mu_1, \ldots, \mu_{n-1})} 
=
q^{-\mu_n } f_{(\mu_n,\mu_1, \ldots, \mu_{n-1})},
\end{align*}
where the next to last equality follows from
$s_1\cdots s_{i-1}(\lambda_1, \ldots, \lambda_n) 
= (\lambda_i, \lambda_1, \ldots, \lambda_{i-1}, \lambda_{i+1},\ldots, \lambda_n)$ and 
$c_nzc_n^{-1}(\lambda_i, \lambda_1, \ldots, \lambda_{i-1}, \lambda_{i+1},\ldots, \lambda_n)
= (\lambda_i, \mu_1,\ldots, \mu_{n-1})$.
\end{proof}

\subsubsection{Examples of the elements $E_\mu$ and $f_\mu$ in $\CC[X]^{(2,1,0)}$.}\label{CX210}

\begin{align*}
E_{(2,1,0)} &= x_1^2x_2+\Big(\frac{1-t}{1-qt^2}\Big)qx_1x_2x_3, \\
E_{(2,0,1)} &= x_1^2x_3
+\Big(\frac{1-t}{1-qt}\Big) x_1^2x_2
+ \Big(\frac{1-t}{1-qt}\Big)qx_1x_2x_3,
\\
E_{(1,2,0)} &= x_1x_2^2 + \Big(\frac{1-t}{1-qt}\Big)x_1^2x_2 
+ \Big(\frac{1-t}{1-qt}\Big)qx_1x_2x_3,
\\
E_{(0,2,1)} &= x_2^2x_3 + \Big(\frac{1-t}{1-qt}\Big) x_1x_2^2
+ \Big(\frac{1-t}{1-q^2t^2}\Big) x_1^2x_3
+ \Big(\frac{1-t}{1-q^2t^2}\Big)\Big(\frac{1-t}{1-qt}\Big) x_1^2x_2 \\
&\qquad
+ \Big( \Big(\frac{1-t}{1-qt}\Big) 
+\Big(\frac{1-t}{1-q^2t^2}\Big)\Big(\frac{1-t}{1-qt}\Big)q\Big) x_1x_2x_3,
\\
E_{(1,0,2)} &= x_1x_3^2 + \Big(\frac{1-t}{1-qt}\Big) x_1^2x_3
+ \Big(\frac{1-t}{1-q^2t^2}\Big) x_1x_2^2
+ \Big(\frac{1-t}{1-q^2t^2}\Big)\Big(\frac{1-t}{1-qt}\Big) x_1^2x_2 \\
&\qquad
+ \Big( \Big(\frac{1-t}{1-qt}\Big) 
+\Big(\frac{1-t}{1-q^2t^2}\Big)\Big(\frac{1-t}{1-qt}\Big)q\Big) x_1x_2x_3,
\\
E_{(0,1,2)} &= x_2x_3^2 
+ \Big(\frac{1-t}{1-qt}\Big) x_2^2x_3
+ \Big(\frac{1-t}{1-qt}\Big) x_1x_3^2
+ \Big(\frac{1-t}{1-q^2t^2}\Big) \Big(\frac{1-t}{1-qt}\Big) x_1^2x_3 
\\
&\qquad
+ \Big(\frac{1-t}{1-q^2t^2}\Big) tx_1^2x_2
+ \Big(\frac{1-t}{1-q^2t^2}\Big)\Big(\frac{1-t}{1-qt}\Big)^2 x_1^2x_2
\\
&\qquad
+ \Big(\frac{1-t}{1-q^2t^2}\Big) \Big(\frac{1-t}{1-qt}\Big) qt x_1x_2^2 
+ \Big(\frac{1-t}{1-q^2t^2}\Big) \Big(\frac{1-t}{1-qt}\Big) x_1x_2^2 
\\
&\qquad
+ \Big(\frac{1-t}{1-qt}\Big)^2 x_1x_2x_3
+ \Big(\frac{1-t}{1-q^2t^2}\Big) \Big(\frac{1-t}{1-qt}\Big) qt x_1x_2x_3
\\
&\qquad
+ \Big(\frac{1-t}{1-q^2t^2}\Big) \Big(\frac{1-t}{1-qt}\Big)^2 q x_1x_2x_3
+ \Big(\frac{1-t}{1-qt}\Big) x_1x_2x_3
, 
\end{align*}
\begin{align*}
f_{(2,1,0)} &= E_{(2,1,0)} 
= x_1^2x_2 + q\frac{(1-t)}{(1-qt^2)} x_1x_2x_3, \\
f_{(1,2,0)} &= t^{\frac12} T_{s_1} E_{(2,1,0)} 
= x_1x_2^2 + t^{-1}\frac{(1-t)qt^2}{(1-qt^2)} x_1x_2x_3, \\
f_{(2,0,1)} &= t^{\frac12} T_{s_2} E_{(2,1,0)} 
= x_1^2x_3  + t^{-1}\frac{(1-t)qt^2}{(1-qt^2)} x_1x_2x_3, \\
f_{(1,0,2)} &= t^{\frac22} T_{s_2} T_{s_1} E_{(2,1,0)} 
= x_1x_3^2 + \frac{(1-t)}{(1-qt^2)} x_1x_2x_3, \\
f_{(0,2,1)} &= t^{\frac22} T_{s_1} T_{s_2} E_{(2,1,0)} 
= x_2^2x_3 + \frac{(1-t)}{(1-qt^2)} x_1x_2x_3, \\
f_{(0,1,2)} &= t^{\frac32} T_{s_1} T_{s_2} T_{s_1} E_{(2,1,0)} 
= x_2x_3^2 + t\frac{(1-t)}{(1-qt^2)} x_1x_2x_3,
\end{align*}

\subsubsection{$P_{(2,1,0)}$ as a symmetrization of $E_{(2,1,0)}$} 


When $n=3$ then
\begin{align*}
W_0(t) &= \sum_{w\in S_3} t^{\ell(w)} = 1+t+t+t^2+t^2+t^3 
= (1+t)(1+t+t^2) = \frac{(1-t^2)(1-t^3)}{(1-t)(1-t)}, \quad\hbox{and} \\
\mathbf{1}_0 &= t^{-\frac32}+t^{-\frac22}T_1+t^{-\frac22}T_2
+t^{-\frac12}T_1T_2+t^{-\frac12}T_2T_1 +T_1T_2T_1.
\end{align*}
Since $W_{(2,1,0)}=\{1\}$ then $W_{(2,1,0)}(t) = 1$ and 
$$
P_{(2,1,0)} 
= \frac{t^{\frac32}}{W_{(2,1,0)}(t)} \mathbf{1}_0E_{(2,1,0)} 
= t^{\frac32} \mathbf{1}_0 t^{-\frac32} \tau_\pi^\vee \tau_\pi^\vee \tau_1^\vee \tau_\pi^\vee \mathbf{1},
$$
and, with $f_{(2,1,0)}, f_{(1,2,0)}, \ldots, f_{(0,1,2)}$ as in Section \ref{CX210},
\begin{align*}
P_{(2,1,0)}
&= (1+t^{\frac12}T_1+t^{\frac12}T_2
+t^{\frac22}T_1T_2 + t^{\frac22}T_2T_1+t^{\frac32}T_1T_2T_1)E_{(2,1,0)} \\
&= f_{(2,1,0)}+f_{(1,2,0)}+f_{(2,0,1)}+f_{(1,0,2)}+f_{(0,2,1)}+f_{(0,1,2)} \\
&= (x_1^2x_2+q \frac{(1-t)}{1-qt^2)} x_1x_2x_3)
+ (x_1x_2^2 + qt\frac{(1-t)}{(1-qt^2)}x_1x_2x_3) 
+ (x_1^2x_3 + qt\frac{(1-t)}{(1-qt^2)}x_1x_2x_3) \\
&\qquad + (x_1x_3^2 + \frac{(1-t)}{(1-qt^2)}x_1x_2x_3) 
+ (x_2^2x_3+\frac{(1-t)}{(1-qt^2)}x_1x_2x_3)
 + (x_2x_3^2+t\frac{(1-t)}{(1-qt^2)}x_1x_2x_3) \\
&=x_1^2x_2+x_1x_2^2+x_1^2x_3+x_1x_3^2+x_2^2x_3+x_2x_3^2 
+ \Big( \frac{(1-t^2)}{(1-qt)}\frac{(1-q^2t)}{(1-qt^2)} + \frac{(1-t)}{(1-q)} \frac{(1-q^2)}{(1-qt)} \Big)
x_1x_2x_3.
\end{align*}

\subsubsection{Symmetrizations for $\mu$ with distinct parts when $n=3$. }\label{Symmexample}

For example, if $n=3$ and $\lambda_1>\lambda_2>\lambda_3$ then $W_\lambda = \{1\}$ 
and $W_\lambda(t) = 1$ and $w_0=s_1s_2s_1$ and $\ell(w_0)=3$.  So
\begin{align*}
F_{(\lambda_1, \lambda_2, \lambda_3) }
&= t^{\frac32} \mathbf{1}_0E_{(\lambda_1, \lambda_2, \lambda_3)} 
= P_{(\lambda_1, \lambda_2, \lambda_3)}, \\
F_{(\lambda_2, \lambda_1, \lambda_3)}
&= t^{\frac32} 
\Big(\frac{1-tq^{\lambda_1-\lambda_2}t^{2-1}}{1-q^{\lambda_1-\lambda_2}t^{2-1} }\Big)
P_{(\lambda_1, \lambda_2, \lambda_3)} \\
F_{(\lambda_1, \lambda_3, \lambda_2)}
&= t^{\frac32} 
\Big(\frac{1-tq^{\lambda_2-\lambda_3}t^{3-2}}{1-q^{\lambda_2-\lambda_3}t^{3-2} }\Big)
P_{(\lambda_1, \lambda_2, \lambda_3)} \\
F_{(\lambda_2, \lambda_3, \lambda_1)}
&= t^{\frac32} 
\Big( \frac{1-tq^{\lambda_1-\lambda_3}t^{3-1}}{1-q^{\lambda_1-\lambda_3}t^{3-1} } \Big) 
\Big(\frac{1-tq^{\lambda_1-\lambda_2}t^{2-1}}{1-q^{\lambda_1-\lambda_2}t^{2-1} }\Big)
P_{(\lambda_1, \lambda_2, \lambda_3)} \\
F_{(\lambda_3, \lambda_1, \lambda_2)}
&= t^{\frac32} 
\Big( \frac{1-tq^{\lambda_1-\lambda_3}t^{3-1}}{1-q^{\lambda_1-\lambda_3}t^{3-1} } \Big) 
\Big(\frac{1-tq^{\lambda_2-\lambda_3}t^{3-2}}{1-q^{\lambda_2-\lambda_3}t^{3-2} }\Big)
P_{(\lambda_1, \lambda_2, \lambda_3)} \\
F_{(\lambda_3, \lambda_2, \lambda_1)}
&= t^{\frac32} 
\Big( \frac{1-tq^{\lambda_1-\lambda_2}t^{2-1}}{1-q^{\lambda_1-\lambda_2}t^{2-1} } \Big) 
\Big( \frac{1-tq^{\lambda_1-\lambda_3}t^{3-1}}{1-q^{\lambda_1-\lambda_3}t^{3-1} } \Big) 
\Big(\frac{1-tq^{\lambda_2-\lambda_3}t^{3-2}}{1-q^{\lambda_2-\lambda_3}t^{3-2} }\Big)
P_{(\lambda_1, \lambda_2, \lambda_3)} 
\end{align*}
since, for example, using
$$v_\lambda(1) = 3,\ v_\lambda(2) = 2, \ v_\lambda(3) =1,
\quad\hbox{and}\quad Y_i^{-1}Y_j E_{(\lambda_1,\lambda_2,\lambda_3)}
= q^{\lambda_i-\lambda_j}t^{v_\lambda(i)-v_\lambda(j)} E_{(\lambda_1,\lambda_2,\lambda_3)}
$$
and $v_\lambda(i)-v_\lambda(j) = (n-i+1)-(n-j+1) = i-j$, 
\begin{align*}
F_{(\lambda_2, \lambda_1, \lambda_3)}
&= \mathbf{1}_0 t^{\frac12} \tau^\vee_1 E_{(\lambda_1, \lambda_2, \lambda_3)} 
=\mathbf{1}_0  \Big(t^{\frac12}T_1+\frac{(1-t)}{1-Y_1^{-1}Y_2 }\Big)
E_{(\lambda_1, \lambda_2, \lambda_3)} \\
&=\mathbf{1}_0  \Big(t+\frac{(1-t)}{1-Y_1^{-1}Y_2 }\Big)
E_{(\lambda_1, \lambda_2, \lambda_3)} 
=\mathbf{1}_0  \Big(\frac{1-tY_1^{-1}Y_2 }{1-Y_1^{-1}Y_2 }\Big)
E_{(\lambda_1, \lambda_2, \lambda_3)} \\
&=\mathbf{1}_0 
\Big(\frac{1-tq^{\lambda_1-\lambda_2}t^{2-1}}{1-q^{\lambda_1-\lambda_2}t^{2-1} }\Big)
E_{(\lambda_1, \lambda_2, \lambda_3)} 
= \Big(\frac{1-tq^{\lambda_1-\lambda_2}t^{2-1}}{1-q^{\lambda_1-\lambda_2}t^{2-1} }\Big)
P_{(\lambda_1, \lambda_2, \lambda_3)} \\
F_{(\lambda_2, \lambda_3, \lambda_1)}
&= \mathbf{1}_0 t^{\frac12} \tau^\vee_2 t^{\frac12} \tau^\vee_1
E_{(\lambda_1, \lambda_2, \lambda_3)} 
= \mathbf{1}_0 
\Big( \frac{1-tY_2^{-1}Y_3}{1-Y_2^{-1}Y_3} \Big) t^{\frac12} \tau^\vee_1
E_{(\lambda_1, \lambda_2, \lambda_3)} \\
&= \mathbf{1}_0 
t^{\frac12} \tau^\vee_1 \Big( \frac{1-tY_1^{-1}Y_3}{1-Y_1^{-1}Y_3} \Big) 
E_{(\lambda_1, \lambda_2, \lambda_3)} 
= \mathbf{1}_0 
t^{\frac12} \tau^\vee_1 
\Big( \frac{1-tq^{\lambda_1-\lambda_3}t^{3-1}}{1-q^{\lambda_1-\lambda_3}t^{3-1} }\Big) 
E_{(\lambda_1, \lambda_2, \lambda_3)} \\
&= 
\Big( \frac{1-tq^{\lambda_1-\lambda_3}t^{3-1}}{1-q^{\lambda_1-\lambda_3}t^{3-1} } \Big) 
\Big(\frac{1-tq^{\lambda_1-\lambda_2}t^{2-1}}{1-q^{\lambda_1-\lambda_2}t^{2-1} }\Big)
P_{(\lambda_1, \lambda_2, \lambda_3)} \\
\end{align*}

\subsubsection{Examples of the $gf_\mu$ condition for a KZ-family.}

Let $n=3$ and $\lambda = (2,1,0)$.  Then $v_\lambda(1)=3$,
$v_\lambda(2)=2$ and $v_\lambda(3)=1$ and
$$Y_i E_{(2,1,0)} = q^{-\lambda_i}t^{-(v_\lambda(i)-1)+\frac12(n-1)}E_{(2,1,0)}.$$
Then
$$Y_1 = gT_2T_1, \qquad Y_2 = T_1^{-1}gT_2, \quad Y_3 = T_2^{-1}T_1^{-2}g,$$
Since
$$
\begin{array}{lll}
f_{(2,1,0)} = E_{(2,1,0)}, \quad
&f_{(1,2,0)} = t^{\frac12}T_1E_{(2,1,0)}, \quad
&f_{(2,0,1)} = t^{\frac12}T_2E_{(2,1,0)}, \\
f_{(0, 2,1)} = t^{\frac22}T_1T_2E_{(2,1,0)}, \quad
&f_{(1,0, 2)} = t^{\frac22}T_2T_1E_{(2,1,0)}, \quad
&f_{(0,1, 2)} = t^{\frac32}T_1T_2T_1E_{(2,1,0)}, 
\end{array}
$$
then
\begin{align*}
gf_{(2,1,0)} &= gE_{(2,1,0)} = T_1T_2(T_2^{-1}T_1^{-1}g)E_{(2,1,0)} = T_1T_2Y_3E_{(2,1,0)}
=q^{-0}t^1T_1T_2E_{(2,1,0)} =  f_{(0,2,1)}, \\
gf_{(1,2,0)} &= g t^{\frac12} T_1E_{(2,1,0)} 
= t^{\frac12} T_2gE_{(2,1,0)} = t^{\frac12} T_2 t T_1T_2E_{(2,1,0)} =  f_{(0,1,2)}, \\
gf_{(2,0,1)} &= g t^{\frac12} T_2E_{(2,1,0)} = t^{\frac12}T_1T_1^{-1}g T_2E_{(2,1,0)} 
= t^{\frac12}T_1Y_2E_{(2,1,0)}
=t^{\frac12}T_1 q^{-1}t^{-1+1}E_{(2,1,0)} = q^{-1} f_{(1,2,0)}, \\
gf_{(0, 2,1)} &= gt^{\frac22}T_1T_2E_{(2,1,0)} = t^{\frac22}T_2gT_2E_{(2,1,0)} 
= t^{\frac22}T_2T_1q^{-1}t^0 E_{(2,1,0)}
= q^{-1} f_{(1,0,2)}, \\
gf_{(1,0, 2)} &= t^{\frac22}gT_2T_1E_{(2,1,0)}=t^{\frac22}Y_1E_{(2,1,0)} 
= t^{\frac22}q^{-2}t^{-2+1}E_{(2,1,0)} 
=q^{-2}f_{(2,1,0)}, \\
gf_{(0,1, 2)} &= t^{\frac32}gT_1T_2T_1E_{(2,1,0)} = t^{\frac32}T_1gT_2T_1E_{(2,1,0)}
=t^{\frac32}T_1q^{-2}t^{-1}E_{(2,1,0)} = q^{-2}f_{(1,2,0)}.
\end{align*}

%

\section{Boxes, arms, legs and counting terms}

\subsubsection{Common terminology.}

\begin{equation*}
\begin{array}{ll}
\ZZ_{\ge 0}^n = \{ \mu= (\mu_1, \ldots, \mu_n)\ |\ \mu_i\in \ZZ_{\ge 0}\}
\qquad&\hbox{is the set of \emph{weak compositions},} \\
\ZZ_{>0}^n = \{ \mu= (\mu_1, \ldots, \mu_n)\ |\ \mu_i\in \ZZ_{>0}\}
\qquad&\hbox{is the set of \emph{strong compositions},} \\
\ZZ^n = \{ \mu= (\mu_1, \ldots, \mu_n)\ |\ \mu_i\in \ZZ\}
\qquad&\hbox{is the \emph{lattice of integral weights,}  } \\
(\ZZ^n)^+ = \{  (\mu_1, \ldots, \mu_n)\in \ZZ^n\ |\ 
\mu_1\ge \mu_2\ge \cdots \ge \mu_n \}
\qquad&\hbox{is the set of \emph{dominant integral weights,} }  \\
(\ZZ_{\ge 0}^n)^+ = \{ (\mu_1, \ldots, \mu_n)\in \ZZ_{\ge 0}^n\ |\ 
\mu_1\ge \mu_2\ge \cdots \ge \mu_n \}
\quad&\hbox{is the set of \emph{partititions of length $\le n$.} } 
\end{array}
\end{equation*}

\subsubsection{Examples of box diagrams.}

If  $\lambda = (5,4,4,1,0)$ and $\mu = (0,4,5,1,4)$ then
$$dg(\lambda) 
=\begin{array}{|ccccc}
\phantom{ \boxed{ \begin{matrix} \phantom{T} \\   \end{matrix} }  }
\\
\boxed{ \begin{matrix} \phantom{\pi^4}  \end{matrix} } 
&\boxed{ \begin{matrix} \phantom{\pi^4} \end{matrix} }
&\boxed{ \begin{matrix} \phantom{\pi^4} \end{matrix} }
&\boxed{ \begin{matrix} \phantom{\pi^4} \end{matrix} }
&\boxed{ \begin{matrix} \phantom{\pi^4} \end{matrix} }
\\
\boxed{ \begin{matrix} \phantom{\pi^4}  \end{matrix} } 
&\boxed{ \begin{matrix}   \phantom{\pi^4} \end{matrix} }
&\boxed{ \begin{matrix} \phantom{\pi^4} \end{matrix} }
&\boxed{ \begin{matrix} \phantom{\pi^4} \end{matrix} }
\\
\boxed{ \begin{matrix} \phantom{\pi^4}  \end{matrix} } 
&\boxed{ \begin{matrix} \phantom{\pi^4} \end{matrix} }
&\boxed{ \begin{matrix} \phantom{\pi^4} \end{matrix} }
&\boxed{ \begin{matrix} \phantom{\pi^4} \end{matrix} }
\\
\boxed{ \begin{matrix} \phantom{\pi^4}  \end{matrix} } 
\end{array}
\qquad\hbox{and}\qquad
dg(\mu) 
=\begin{array}{|ccccc}
\phantom{ \boxed{ \begin{matrix} \phantom{T} \\   \end{matrix} }  }
\\
\boxed{ \begin{matrix} \phantom{\pi^4}  \end{matrix} } 
&\boxed{ \begin{matrix}   \phantom{\pi^4} \end{matrix} }
&\boxed{ \begin{matrix} \phantom{\pi^4} \end{matrix} }
&\boxed{ \begin{matrix} \phantom{\pi^4} \end{matrix} }
\\
\boxed{ \begin{matrix} \phantom{\pi^4}  \end{matrix} } 
&\boxed{ \begin{matrix} \phantom{\pi^4} \end{matrix} }
&\boxed{ \begin{matrix} \phantom{\pi^4} \end{matrix} }
&\boxed{ \begin{matrix} \phantom{\pi^4} \end{matrix} }
&\boxed{ \begin{matrix} \phantom{\pi^4} \end{matrix} }
\\
\boxed{ \begin{matrix} \phantom{\pi^4}  \end{matrix} } 
\\
\boxed{ \begin{matrix} \phantom{\pi^4}  \end{matrix} } 
&\boxed{ \begin{matrix} \phantom{\pi^4} \end{matrix} }
&\boxed{ \begin{matrix} \phantom{\pi^4} \end{matrix} }
&\boxed{ \begin{matrix} \phantom{\pi^4} \end{matrix} }
\end{array}
$$
To conform to \cite[p.2]{Mac}, we draw the box $(i,j)$ 
as a square in row $i$ and column $j$ using the same coordinates
as are usually used for matrices.
$$\hbox{The \emph{cylindrical coordinate} of the box $(i,j)$ is the number $i+nj$.}
$$


\subsubsection{Formulas for $\#\mathrm{Nleg}_\mu(i,j)$ and $\#\mathrm{Narm}_\mu(i,j)$}

Using cylindrical coordinates for boxes define, for a box $b\in dg(\mu)$,
\begin{align}
\mathrm{attack}_\mu(b) &= \{b-1, \ldots, b-n+1\}\cap \widehat{dg}(\mu), 
\label{attackdefn}
\\
\mathrm{Nleg}_\mu(b) &= (b+n\ZZ_{>0})\cap dg(\mu)
\quad \hbox{and} \\
\mathrm{Narm}_\mu(b) &= \{ a\in \mathrm{attack}_\mu(b)
 \ |\ \#\mathrm{Nleg}_\mu(a)\le \#\mathrm{Nleg}_\mu(b)\}.
\label{Harmdef}
\end{align}
As  in \cite[(15)]{HHL06}, the number of elements of $\mathrm{Nleg}_\mu(i,j)$ and 
$\mathrm{Narm}_\mu(i,j)$ are
\begin{align*}
\#\mathrm{Nleg}_\mu(i,j) 
&= \#\{ (i, j')\in dg(\mu)\ |\ j'>j\} = \mu_i-j, \\
\#\mathrm{Narm}_\mu(i,j)
&= \#\{ (i',j)\in dg(\mu)\ |\ \hbox{$i'<i$ and $\mu_{i'}\le \mu_i$} \}  
+\#\{ (i',j-1)\in \widehat{dg}(\mu)\ |\ \hbox{$i'>i$ and $\mu_{i'}<\mu_i$} \},
\end{align*}
where $\widehat{dg}(\mu) = dg(\mu) \cup \{ (1,0), \ldots, (n,0)\}$.

\subsubsection{Relating HHL arms and legs to Macdonald arms and legs.}

If $\mu$ is decreasing so that $\mu_1\ge \mu_2 \ge \dots \ge \mu_n$ then $\mu$ is a partition and
$$
\#\mathrm{Narm}_\mu(i,j) = \mu_{j-1}' - i = \mathrm{leg}_\mu(i,j-1)
\quad\hbox{and}\quad
\#\mathrm{Nleg}_\mu(i,j) = \mu_i-j = \mathrm{arm_\mu}(i,j).
$$
If $\mu$ is increasing so that $\mu_1\le \mu_2\le \cdots \le \mu_n$ then 
$w_0\mu = (\mu_n, \ldots, \mu_1)$ is a partition and
$$
\begin{array}{rl}
\#\mathrm{Narm}_\mu(i,j) &= (w_0\mu)_j'-(n-i) \\
&= \mathrm{leg}_{w_0\mu}(n-i,j)
\end{array}
\qquad\hbox{and}\qquad
\begin{array}{rl}
\#\mathrm{Nleg}_\mu(i,j) &= \mu_i-j = (w_0\mu)_{n-i}-j \\
&= \mathrm{arm}_{w_0\mu}(n-i,j)
\end{array}
$$
(see \cite[remarks before (17)]{HHL06} and  \cite[p. 136, remarks before Figure 6]{Hgl06}).

\subsubsection{Formulas for the number of alcove walks $\#\mathrm{AW}^z_\mu$ and 
nonattacking fillings $\#\mathrm{NAF}^z_\mu$}

The motivation for computing $\#\mathrm{AW}^z_\mu$
and $\#\mathrm{NAF}^z_\mu$ is that
the alcove walks formula and the nonattacking fillings formulas for the relative
Macdonald polynomial $E^z_\mu$ are, respectively,
$$
E^z_\mu = \sum_{p\in AW^z_\mu}  \wt(p)
\qquad\hbox{and}\qquad
E^z_\mu  = \sum_{T\in \mathrm{NAF}^z_\mu} \wt(T).
$$
(see \cite[Theorem 1.1]{GR21}).
The number of terms in the first formula is $\#\mathrm{AW}^z_\mu$
and the number of terms in the second formula is $\#\mathrm{NAF}^z_\mu$.

For a box $(i,j)\in dg(\mu)$ define $u_\mu(i,j)$ by the equation
$$u_\mu(i,j) +1 =  n- \#\mathrm{attack}_\mu(i,j).$$
Since
$\# \mathrm{attack}_\mu(i,j) = \#\{ i'\in \{1, \ldots, i-1\} \ |\ \mu_{i'}\ge j \}
+
\#\{ i'\in \{i+1, \ldots, n\}\ |\ \mu_{i'} \ge j-1\}
$
then
$$u_\mu(i,j) = \#\{ i'\in \{1, \ldots, i-1\}\ |\ \mu_{i'}<j\le \mu_i \} 
+\#\{ i'\in \{i+1, \ldots, n\}\ |\ \mu_{i'} < j-1 < \mu_i \} ).
$$


Let $\mu = (\mu_1, \ldots, \mu_n)\in \ZZ_{\ge 0}^n$ and $z\in S_n$.  By 
\cite[Prpoposition (2.2)]{GR21} and the definition of alcove walks and nonattacking fillings
in \cite[(1.11) and (1.7)]{GR21},
\begin{equation}
\#\mathrm{AW}^z_\mu = 2^{\ell(u_\mu)} = \prod_{(i,j)\in \mu} 2^{u_\mu(i,j)}
\qquad\hbox{and}\qquad
\#\mathrm{NAF}^z_\mu = \prod_{(i,j)\in \mu} (u_\mu(i,j)+1).
\label{AWNAFformulas}
\end{equation}
(The right hand side does not depend on the choice of $z$.)
For example (as in \cite[Table 1]{CMW18}),
$$
\#\mathrm{NAF}^z_{(4,3,3,3,2,2,1,1,0,0)} = \left( \begin{array}{l}
\phantom{\cdot} 1 \cdot 3\cdot 5\cdot 7 \\
\cdot 1 \cdot 3\cdot 5 \\
\cdot 1 \cdot 3\cdot 5 \\
\cdot 1 \cdot 3\cdot 5 \\
\cdot 1 \cdot 3 \\
\cdot 1 \cdot 3 \\
\cdot 1 \\
\cdot 1
\end{array}\right)
= 3189375,
\qquad\hbox{for $z\in S_{10}$.}
$$


%
%

\subsection{The column strict tableaux formula for $P_\lambda$}

Let $\lambda$ and $\mu$ be partitions such that $\lambda\supseteq \mu$ and 
$\lambda/\mu$ is a horizontal strip.  Following \cite[Ch.\ VI \S7 Ex.\ 2(b)]{Mac}, define
$$
\psi_{\lambda/\mu} = 
\prod_{1\le i < j\le \ell(\mu)} 
\frac{ (q^{\mu_i-\mu_j}t^{j-i+1};q)_\infty   (q^{\lambda_i-\lambda_{j+1}}t^{j-i+1};q)_\infty 
(q^{\lambda_i-\mu_j+1}t^{j-i};q)_\infty (q^{\mu_i-\lambda_{j+1}+1}t^{j-i};q)_\infty }
{ (q^{\mu_i-\mu_j+1}t^{j-i};q)_\infty (q^{\lambda_i-\lambda_{j+1}+1}t^{j-i};q)_\infty
(q^{\lambda_i-\mu_j}t^{j-i+1};q)_\infty (q^{\mu_i-\lambda_{j+1}}t^{j-i+1};q)_\infty }.
$$
where the infinite product $(x;q)_\infty = (1-x)(1-xq)(1-xq^2)\cdots$.
A \emph{column strict tableau} of shape $\lambda$
is a filling $T\colon dg(\lambda) \to \{1, \ldots, n\}$ such that
$$T(i,j)\le T(i,j+1)
\qquad\hbox{and}\qquad
T(i,j)< T(i+1,j).$$
For a column strict tableau $T$ define 
$$\psi_T 
= \prod_{i=1}^r \psi_{\lambda^{(i)}/\lambda^{(i-1)}}
\qquad\hbox{where}\quad
\lambda^{(i)} = \{ u\in dg(\lambda) \ |\ T(u)\le i\}.
$$
Then \cite[Ch.\ VI (7.13${}'$)]{Mac} gives
\begin{equation}
P_\lambda = \sum_{T} \psi_T x^T,
\qquad\hbox{where}\quad
x^T = x_1^{\#\hbox{\scriptsize{(1s in $T$)}}}\cdots x_n^{\#\hbox{\scriptsize{($n$s in $T$)}}}.
\label{colstricttableauxformula}
\end{equation}
By \cite[Ch.\ 1 \S 3 Ex.\ 4]{Mac}, this formula for $P_\lambda$ has
$$\prod_{b\in \lambda} \frac{n+c(b)}{h(b)}
\quad\hbox{terms},\quad \hbox{where}\quad
\begin{array}{l}
\hbox{$c(b)$ is the content of the box $b$,} \\
\hbox{$h(b)$ is the hook length at the box $b$.}
\end{array}
$$

\subsubsection{Comparing numbers of terms in formulas for $P_\lambda$.}

Let $\lambda = (\lambda_1, \ldots, \lambda_n)$ be a partition 
and write $\lambda = (0^{m_0}1^{m_1}2^{m_2}\cdots)$ so that
$m_i$ is the number of rows of $\lambda$ of length $i$.  Then number of elements
of the orbit $S_n\lambda$ (the number of rearrangements of $\lambda$) is
$$\Card(S_n\lambda) = \frac{n!}{m_\lambda!},
\qquad\hbox{where}\qquad m_\lambda! = m_0!m_1!m_2!\cdots.$$
By \eqref{Pdefn}, the symmetric Macdonald polynomial is given by
$P_\lambda = \sum_{\nu\in S_n\lambda} E^z_\lambda,$
and using the alcove walks formula for $E^z_\lambda$ and the
nonattacking fillings formulas for $E^z_\lambda$ provide formulas for
$P_\lambda$ with
$$
\frac{n!}{m_\lambda!}\cdot \#\mathrm{AW}^z_\lambda\ \hbox{terms,}
\qquad\hbox{and}\qquad
\frac{n!}{m_\lambda!}\cdot \#\mathrm{NAF}^z_\lambda\ \hbox{terms},
\quad\hbox{respectively.}
$$
Alternatively, by Proposition \ref{PinEzmu}, there is a constant $(const)$ such that
$$P_\lambda = (const)\sum_{\nu\in S_n\lambda} E^z_{rev(\lambda)},
\qquad\hbox{where}\quad
\begin{array}{l}
\hbox{if $\lambda = (\lambda_1, \lambda_2, \ldots, \lambda_k, 0, \ldots, 0)$ with $\lambda_k\ne 0$} \\
\hbox{then $rev(\lambda) = (\lambda_k, \ldots, \lambda_2,\lambda_1, 0, \ldots, 0)$.}
\end{array}$$
Then using the alcove walks formula for $E^z_{rev(\lambda)}$ and the
nonattacking fillings formulas for $E^z_{rev(\lambda)}$ provide formulas for
$P_\lambda$ with
$$
\frac{n!}{m_\lambda!}\cdot \#\mathrm{AW}^z_{rev(\lambda)}\ \hbox{terms,}
\qquad\hbox{and}\qquad
\frac{n!}{m_\lambda!}\cdot \#\mathrm{NAF}^z_{rev(\lambda)}\ \hbox{terms},
\quad\hbox{respectively.}
$$

Let $\lambda$ be a partition. 
Let $\lambda' = (\lambda'_1, \ldots, \lambda'_k)$ be the conjugate partition to $\lambda$
so that $\lambda_j'$ is the length of the $j$th column of $\lambda$.  
For $\lambda = (\lambda_1, \lambda_2, \ldots, \lambda_k, 0, \ldots, 0)$ with $\lambda_k\ne 0$
let $rev(\lambda) = (\lambda_k, \ldots, \lambda_2,\lambda_1, 0, \ldots, 0)$.
Then $u_\lambda(i,1)=u_{rev(\lambda)}(i,1)=0$ and 
if $j>1$ then $u_\lambda(i,j) = n-\lambda'_{j-1}$ and $u_{rev(\lambda)}(i,j) = n-\lambda'_j$.
Thus
$$\#\mathrm{AW}_\lambda = \prod_{(i,j)\in \lambda\atop j>1} 2^{n-\lambda'_{j-1}},
\quad
\#\mathrm{NAF}_\lambda = \prod_{(i,j)\in \lambda\atop j>1} (n-\lambda'_{j-1}+1),
\quad
\#\mathrm{NAF}_{rev(\lambda)} = \prod_{(i,j)\in \lambda \atop j>1} (n-\lambda'_j+1),
$$
and
$$t(\lambda) = n!\cdot \prod_{(i,j)\in \lambda\atop j>1} (n-\lambda'_{j-1}+1),
\quad
c(\lambda) = 
\prod_{(i,j)\in\lambda\atop j>1} \frac{2^{n-\lambda'_{j-1}}}{n-\lambda'_{j-1}+1},
\quad
r(\lambda) = \prod_{(i,j)\in\lambda\atop j>1} \frac{n-\lambda'_j+1}{n-\lambda'_{j-1}+1}
$$
are formulas for the values provided in the table in  \cite[end of \S3]{Len08}
(Lenart assumes that the parts of $\lambda$ are distinct so that $m_\lambda! = 1$).
For example, if $\lambda = (5,4,2,1,0)$ as in the last row of Lenart's table then
$$
t(\lambda) = 5!\cdot \left(
\begin{array}{l} 
1\cdot 2\cdot 3\cdot 4\cdot 4 \\
1\cdot 2\cdot 3\cdot 4 \\
1\cdot 2 \\
1
\end{array}\right),
\quad
c(\lambda) = 
\frac{
\left(
\begin{array}{l} 
2^0\cdot 2^1\cdot 2^2\cdot 2^3\cdot 2^3 \\
2^0\cdot 2^1\cdot 2^2 \cdot 2^3 \\
2^0 \cdot 2^1 \\
2^0
\end{array}\right)
}{
\left(
\begin{array}{l} 
1\cdot 2\cdot 3\cdot 4\cdot 4 \\
1\cdot 2\cdot 3\cdot 4 \\
1\cdot 2 \\
1
\end{array}\right)
},
\quad
r(\lambda) = 
\frac{
\left(
\begin{array}{l} 
1 \\
1 \cdot 3 \\
1\cdot 3\cdot 4\cdot 4 \\
1\cdot 3\cdot 4\cdot 4 \cdot 5 \\
\end{array}\right)
}{
\left(
\begin{array}{l} 
1\cdot 2\cdot 3\cdot 4\cdot 4 \\
1\cdot 2\cdot 3\cdot 4 \\
1\cdot 2 \\
1
\end{array}\right)
},
$$
so that $t(\lambda) = 552960$, $c(\lambda) = \frac{128}{9} \approx 14.222$ and $r(\lambda) = \frac{15}{2} = 7.5$.  To compare this with the number of column strict tableaux of shape
$\lambda = (5,4,2,1,0)$ (the number of terms in the formula for $P_\lambda$ in \eqref{colstricttableauxformula}),
$$\prod_{b\in \lambda} \frac{n+c(b)}{h(b)}
=
\frac{
\left(
\begin{array}{l} 
5\cdot 6\cdot 7\cdot 8\cdot 9 \\
4\cdot 5\cdot 6 \cdot 7 \\
3 \cdot 4 \\
2
\end{array}\right)
}{
\left(
\begin{array}{l} 
8\cdot 6\cdot 4\cdot 3\cdot 1 \\
6\cdot 4\cdot 2\cdot 1 \\
3\cdot 1 \\
1
\end{array}\right)
}
=5\cdot 7\cdot 3\cdot 5\cdot 7 = 3675,
\qquad\hbox{and}\qquad
\frac{552960}{3675} = 150.465.
$$


\section{Converting fillings and alcove walks to paths and pipe dreams}

\subsubsection{Hyperplanes and alcoves}

Let
$\RR^n = \fa_\RR^* = \RR\varepsilon_1+\cdots+\RR\varepsilon_n.$
For $i,j, k\in \{1, \ldots, n\}$ with $i<j$ and $\ell\in \ZZ$ define
\begin{align}
\fa^{\varepsilon_i^\vee-\varepsilon_j^\vee+\ell K}
&= \{ (\mu_1, \ldots, \mu_n)\in \RR^n\ |\ \mu_i-\mu_j = -\ell \},
\quad\hbox{and} \nonumber \\
\fa^{\epsilon_k^\vee + \ell K} &= \{ (\mu_1, \ldots, \mu_n)\in \RR^n\ |\ \mu_k = -\ell \}.
\label{hyperplanes}
\end{align}
The union of these hyperplanes is 
$$\cH =\{ (\mu_1,\ldots, \mu_n)\in \RR^n\ |\ 
\hbox{if $i,j\in \{1, \ldots, n\}$ and $i\ne j$ then $\mu_i\not\in \ZZ$ and $\mu_i-\mu_j\not\in \ZZ$}\}.$$
An \emph{alcove} is a connected component of
$$\RR^n -\cH,\quad\hbox{the complement of the hyperplanes listed in \eqref{hyperplanes}}.$$
The \emph{fundamental alcove} is
$$A_1 = \{ \mu = (\mu_1, \ldots, \mu_n)\in \RR^n\ |\ 
\hbox{$\mu_1-\mu_n\in \RR_{>0}$ and if $i\in \{1, \ldots, n\}$ then $\mu_i\in \RR_{(-1,0)}$}\}.
$$
For $n=2$, some pictures of these hyperplanes and paths in $\fa_\RR^*\cong \RR^2$ are in 
section \ref{GL2pathpictures}.
%
%

\subsubsection{Bijection $W\leftrightarrow W \cdot \frac{1}{n}\rho \leftrightarrow \{\hbox{alcoves}\}$}

Let $W$ be the group of $n$-periodic permutations and 
define an action of $W_{GL_n}$ on $\RR^n$ by
\begin{align}
\pi(\mu_1, \ldots, \mu_n) &= (\mu_n+1, \mu_1, \ldots, \mu_n), \label{affineWGonRn} \\
\hbox{and}\qquad
s_i(\mu_1, \ldots, \mu_n) 
&= (\mu_1, \ldots, \mu_{i-1}, \mu_{i+1}, \mu_i, \mu_{i+1}, \ldots, \mu_n),
\quad\hbox{for $i\in \{1, \ldots, n-1\}$.} \nonumber
\end{align}
Let
\begin{equation}
\rho = \hbox{$(\frac{n-1}{2},\frac{n-3}{2}, \ldots, \frac{-(n-1)}{2})$} = 
(n-1, n-2, \ldots, 1, 0) - \hbox{$\frac{n-1}{2}$}(1,1,\ldots, 1).
\label{GLnrhodefn}
\end{equation}
Then the maps
\begin{equation}
\begin{matrix}
W &\longleftrightarrow &W \cdot \hbox{$\frac{1}{n}$}\rho &\longleftrightarrow &\{\hbox{alcoves}\} \\
w &\longmapsto &\hbox{$\frac{1}{n}$}w\rho &\longmapsto &wA_1
\end{matrix}
\qquad\qquad\hbox{are bijections,}
\label{Wtoalcoves}
\end{equation}
and so we can identify $W$ with the set of alcoves and with the orbit 
$W \cdot \hbox{$\frac{1}{n}$}\rho$.
The statement in \eqref{Wtoalcoves} holds because the stabilizer of 
$\frac{1}{n}\rho$ under the action of $W$ on $\RR^n$ is $\{1\}$.

\subsubsection{Reflections in $W$.}

For any pair $(j,k)\in \ZZ \times \ZZ$ with $j\ne k$ define
$$s_{jk}(j) = k, \quad s_{jk}(k)=(j), \quad s_{jk}(i) = i\ \hbox{if $i\ne j \bmod n$ and $i\ne k\bmod n$.}.
$$

If $i\in \{1, \ldots, n-1\}$ and $t_\mu v = ((\mu_1)_{v(1)}, (\mu_2)_{v(2)}, \ldots, (\mu_n)_{v(n)})$
then
$$s_it_\mu v = ((\mu_1)_{v(1)}, \ldots, (\mu_{i-1})_{v(i-1)},
(\mu_{i+1})_{v(i+1)}, (\mu_{i})_{v(i)}, (\mu_{i+2})_{v(i+2)}, \ldots, (\mu_n)_{v(n)}),
$$
so that, in extended one-line notation, $s_i$ acts by switching the $i$th and $(i+1)$st components.
The
$$\hbox{hyperplane $\fa^{\beta^\vee}$ between $t_\mu vA_1$ and $s_it_\mu vA_1$ has root}\qquad
\beta^\vee = \varepsilon^\vee_{v(i+1)}-\varepsilon^\vee_{v(i)} + (\mu_i - \mu_{i+1}) K.$$

\subsubsection{Paths.}

A \emph{path} is a piecewise linear function $\gamma\colon \RR_{[0,a]}\to \RR^n$, where
$a\in \RR_{>0}$ and $\RR_{[0,a]} = \{ t\in \RR\ |\ 0\le t\le a\}$.
The \emph{concatenation of paths $\gamma_1\colon \RR_{[0,a]}\to \fh_\RR^*$ 
and $\gamma_2\colon \RR_{[0,b]}\to \fh_\RR^*$} is the path
$$\gamma_1 \gamma_2\colon \RR_{[0,a+b]} \to \fh_\RR^*
\qquad\hbox{given by}\qquad
(\gamma_1 \gamma_2)(t) = \begin{cases}
\gamma_1(t), &\hbox{if $i\in \RR_{[0,a]}$,} \\
\gamma_1(a)+\gamma_2(t-a), &\hbox{if $t\in \RR_{[a,a+b]}$.}
\end{cases}
$$

\subsubsection{Paths corresponding to nonattacking fillings.}

The \emph{straight line path to $0\to \varepsilon_i$} is
$$\begin{matrix}
x_i\colon &\RR_{[0,1]} &\to &\RR^n \\
&t &\mapsto &t\varepsilon_i
\end{matrix}
$$
If $T$ is a nonattacking filling of type $(z, \mu)$ then the \emph{word, or path, of $T$} is 
$$\vec x_T = \prod_{u\in \mu} x_{T(u)}
\qquad\hbox{taken in increasing order of cylindrical coordinate.}
$$
The path, or word,
$$\vec x_T = x_{i_1} x_{i_2} \cdots x_{i_\ell}
\qquad\hbox{is}\qquad
0\to \varepsilon_{i_1}\to (\varepsilon_{i_1}+\varepsilon_{i_2}) \to \cdots\to
\varepsilon_{i_1}+\cdots+\varepsilon_{i_\ell}
$$
as a sequence of straight line segments.

\subsubsection{Paths corresponding to alcove walks.}

Define paths $\omega\colon \RR_{[0,1]}\to \RR^n$ and $c_\alpha\colon \RR_{[0,1]} \to \RR^n$ and $f_\alpha\colon \RR_{[0,1]}\to \RR^n$ by
$$\omega(t) = \frac{t}{n}(1,1, \ldots,1),
\qquad
c_\alpha(t) = t\alpha
\quad\hbox{and}\quad
f_\alpha(t) = \begin{cases}
t\alpha, &\hbox{if $0\le t\le \frac12$,} \\
(1-t)\alpha, &\hbox{if $\frac12\le t \le 1$.}
\end{cases}
$$
Let $\mu = (\mu_1, \ldots, \mu_n)\in \ZZ^n_{\ge 0}$ and $z\in S_n$.
Let $s_\pi = \pi$ and let $\vec u_\mu=s_{i_1}\cdots s_{i_r}$ be a reduced word for $u_\mu$.
An \emph{alcove walk} of type $(z,\vec u_\mu)$ is 
\begin{equation}
\hbox{a sequence\quad $p=(p_0, p_1, \ldots, p_r)$\quad of elements of $W$\quad such that}
\label{alcwalkdefn}
\end{equation}
$p_0 = z$;
if $s_{i_k} = \pi$
then $p_k = p_{k-1}\pi$; and 
if $s_{i_k}\ne \pi$ then
$p_k\in \{p_{k-1}, p_{k-1}s_{i_k}\}$.  
The path corresponding to 
$p$ is
\begin{equation}
\gamma_{\beta_1}\cdots \gamma_{\beta_\ell},
\qquad\hbox{where}\qquad
\gamma_{\beta_j} = \begin{cases}
f_{p_{k-1}\alpha_{i_k}}, &\hbox{if $p_k=p_{k-1}$,} \\
c_{p_{k-1}\alpha_{i_k}}, &\hbox{if $p_k = p_{k-1}s_{i_k}$,}  \\
\omega, &\hbox{if $p_k = p_{k-1}\pi$,}
\end{cases}
\label{pathsforAW}
\end{equation}
See \S \ref{GL2examples} for pictures in $\RR^2$, for $n=2$.  The pictures of 
paths for $n=3$ in sections \ref{GL3pathpicturesA} and \ref{GL3pathpicturesA}
are projections from $\RR^3$ to the plane 
$\{ (\gamma_1, \gamma_2, \gamma_3)\in \RR^3\ |\ \gamma_1+\gamma_2+\gamma_3=0\}$.

\subsubsection{Pipe dreams corresponding to nonattacking fillings}

Let $\mu\in \ZZ_{\ge 0}^n$.  
A \emph{filling of $dg(\mu)$} is a function $T\colon dg(\mu) \to \{1, \ldots, n\}$.  
If the filling is nonattacking, then it satisfies the \emph{column distinct condition},
\begin{equation}
\hbox{if $j\in \ZZ_{\ge 0}$ and $(i,j), (i',j)\in D$ then $T(i,j)\ne T(i',j)$,}
\tag{CD}
\end{equation}
and so the filling $T$ can be converted into a \emph{pipe dream} 
$P\colon \{1,\ldots, n\}\times \ZZ_{\ge 0} \to \{1, \ldots, n\}$ by setting
\begin{equation}
P(k,j) = i \qquad\hbox{if and only if}\qquad T(i,j) = k,
\label{fillingtopipedream}
\end{equation}
and putting $P(k,j)=0$ if there does not exist $i\in \{1, \ldots, n\}$ such that $T(i,j)=k$.
(This bijection is given in \cite[(5.10)]{BW19} and \cite[Definition A.6]{CMW18}.  
In \cite[Definition A.6]{CMW18}
the pipe dreams are the \emph{multiline queues} and the fillings are the  Queue Tableaux and 
in \cite[(5.10)]{BW19} the pipe dreams are the $\mu$-legal configurations.)
The column distinct condition on $T$ is exactly the condition that $P$ obtained in this
way is a function.

For example,
$$\begin{array}{c|cc}
1 &1 &1 \\ 2 &2 &2 \\ 3 \end{array}
\qquad
\begin{array}{c|cc}
1 &1 &1 \\ 2 &2 &3 \\ 3 \end{array}
\qquad
\begin{array}{c|cc}
1 &1 &3 \\ 2 &2 &1 \\ 3 \end{array}
\qquad
\begin{array}{c|cc}
1 &1 &3 \\ 2 &2 &2 \\ 3 \end{array}
$$
are the 4 nonattacking fillings of $\mu = (2,2,0)$.
Converting these to pipe dreams gives
$$\left(\begin{array}{c|cc}
1 &1 &1 \\
2 &2  &2 \\
3 &0 &0
\end{array}\right)
\qquad
\left(\begin{array}{c|cc}
1 &1 &1 \\
2 &2  &0 \\
3 &0 &2
\end{array}\right)
\qquad
\left(\begin{array}{c|cc}
1 &1 &2 \\
2 &2  &1 \\
3 &0 &0
\end{array}\right)
\qquad
\left(\begin{array}{c|cc}
1 &1 &0 \\
2 &2  &2 \\
3 &0 &1
\end{array}\right)
$$
The example in \cite[Figure 5]{BW19} has
$$\hbox{filling}\quad\begin{array}{c|ccccc}
1 \\ 2 &1 &1 &1 &2 \\
3 &3 \\ 4 &4 &4 &5 &4 &4 \\ 
5 &5 &2 &3 &3 
\end{array}
\quad\hbox{with corresponding pipe dream}\quad
\left(\begin{array}{c|ccccc}
1 &2 &2 &2 &0 &0 \\
2 &0 &5 &0 &2 &0 \\
3 &3 &0 &5 &5 &0 \\
4 &4 &4 &0 &4 &4 \\
5 &5 &0 &4 &0 &0
\end{array}\right)
$$
and the picture of this pipe dream from \cite[Figure 5]{BW19} is
\begin{equation}
\vcenter{\hbox{\includegraphics[scale=0.3]{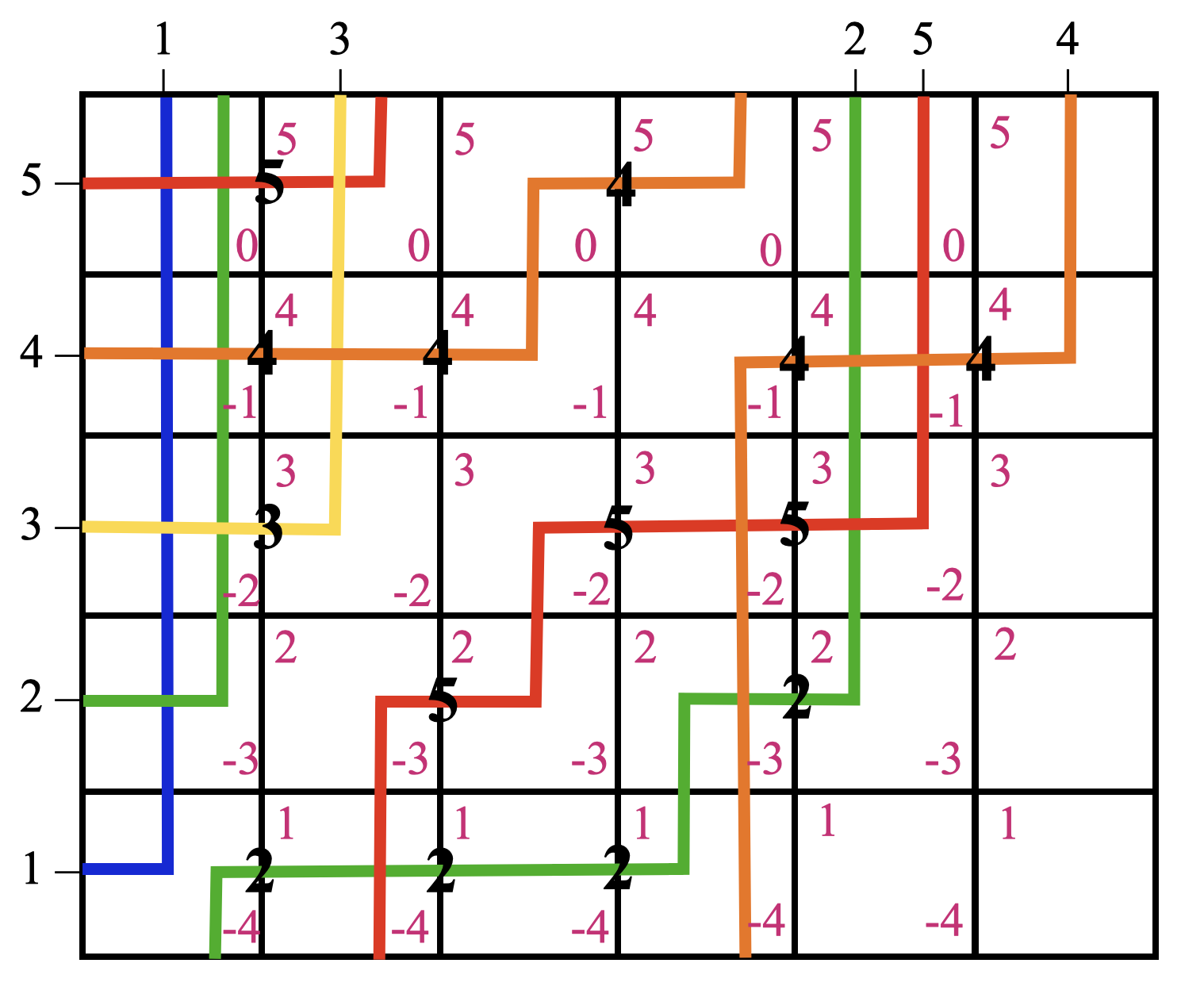}}} 
\label{BWpipedream}
\end{equation}
(\cite{BW19} index rows bottom
to top instead of top to bottom).
The example in \cite[Figures 3 and 12]{CMW18} has
\begin{equation*}
\hbox{nonattacking filling}\quad T=
\begin{array}{c|cccc}
6 &6 &5 &3 \\
1 &1 &6 \\
2 &2 &2 \\
7 &7 &4 \\
8 &8 \\
3 \\
4 \\
5
\end{array}
\quad\hbox{and pipe dream}
\quad
P= \left(\begin{array}{c|ccc}
2 &2 &0 &0 \\
3 &3 &3 &0 \\
6 &0 &0 &1 \\
7 &0 &4 &0 \\
8 &0 &1 &0 \\
1 &1 &2 &0 \\
4 &4 &0 &0 \\
5 &5 &0 &0
\end{array}
\right)
\end{equation*}
and the picture of this pipe dream (multiline queue in the terminology of \cite{CMW18}) 
from \cite[Fig.\ 3]{CMW18} is
\begin{equation*}
\hbox{the multiline queue}\qquad
\vcenter{\hbox{\includegraphics[scale=0.4, angle=270]{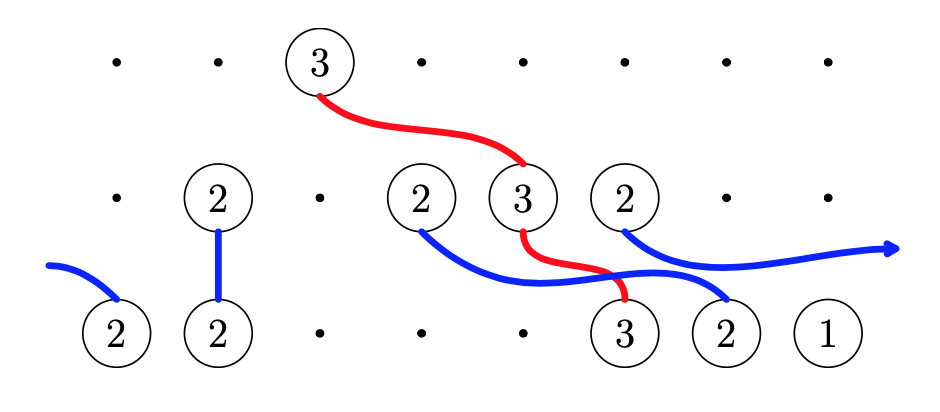}}} 
\label{CMWmultilinequeue}
\end{equation*}


\subsubsection{Alcove walks, nonattacking fillings and paths for $E_{(3,0)}$}\label{GL2pathpictures}

The explicit expansion of $E_{(3,0)}$ is
$$
E_{(3,0)} = x_1^3 + \Big(\frac{1-t}{1-q^2t}\Big)q^2 x_1x_2^2
+ \Big( \Big(\frac{1-t}{1-qt}\Big)q
+ \Big(\frac{1-t}{1-q^2 t}\Big)\Big(\frac{1-t}{1-qt}\Big)q^2 \Big) x_1^2x_2.
$$
The nonattacking fillings, words, paths, alcove walks and 
corresponding weights for $E_{(3,0)}$ are
$$
\begin{matrix}
\begin{array}{c|ccc}
1 &1 &1 &1 \\
2
\end{array}
&\begin{array}{c|ccc}
1 &1 &1 &2 \\
2
\end{array}
&\begin{array}{c|ccc}
1 &1 &2 &2 \\
2
\end{array}
&\begin{array}{c|ccc}
1 &1 &2 &1 \\
2
\end{array}
\\
x_1x_1x_1
&x_1 x_1 x_2
&x_1x_2x_2
&x_1x_2x_1
\\
\vcenter{\hbox{\includegraphics[scale=0.15]{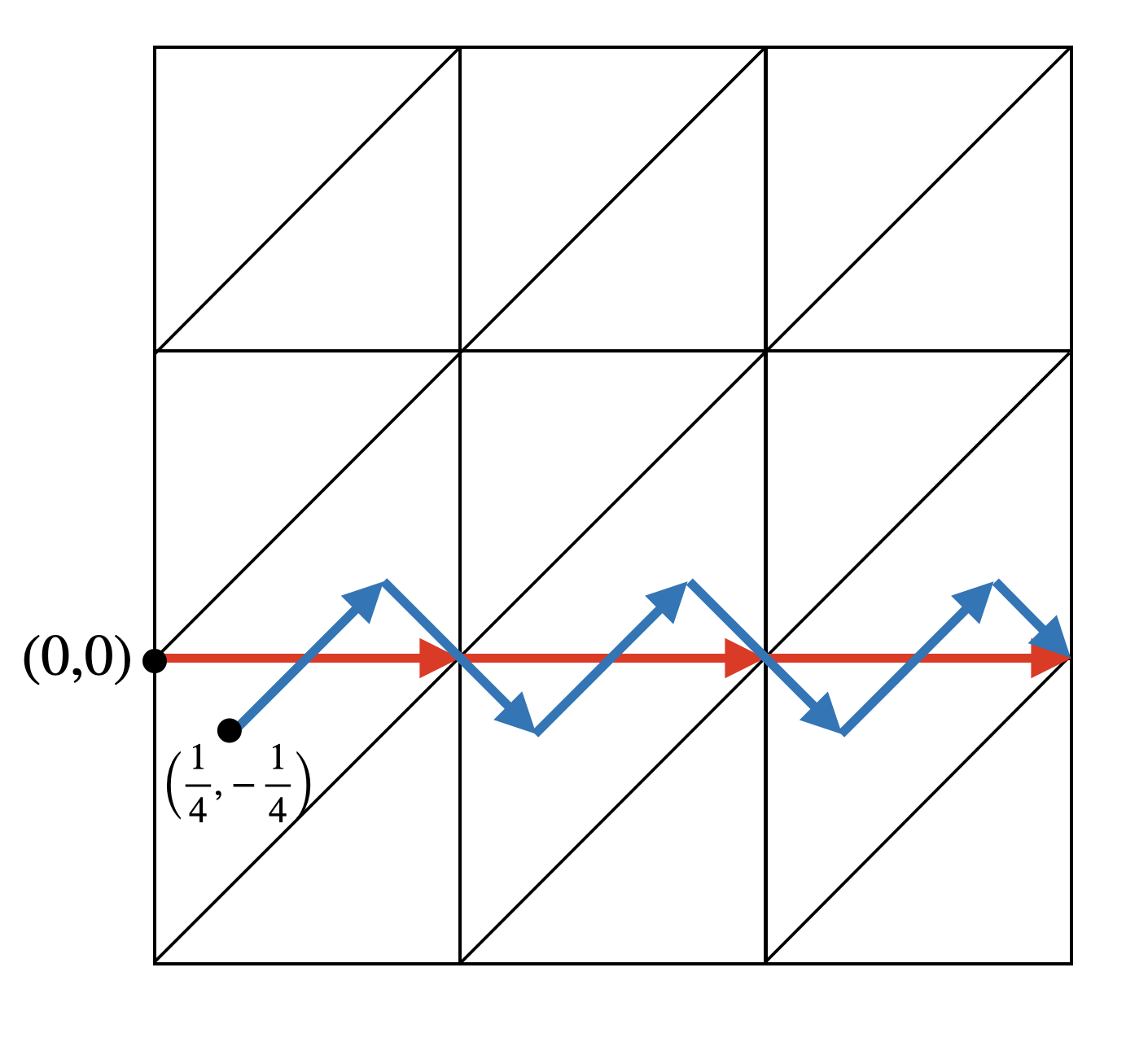}}} 
&\vcenter{\hbox{\includegraphics[scale=0.15]{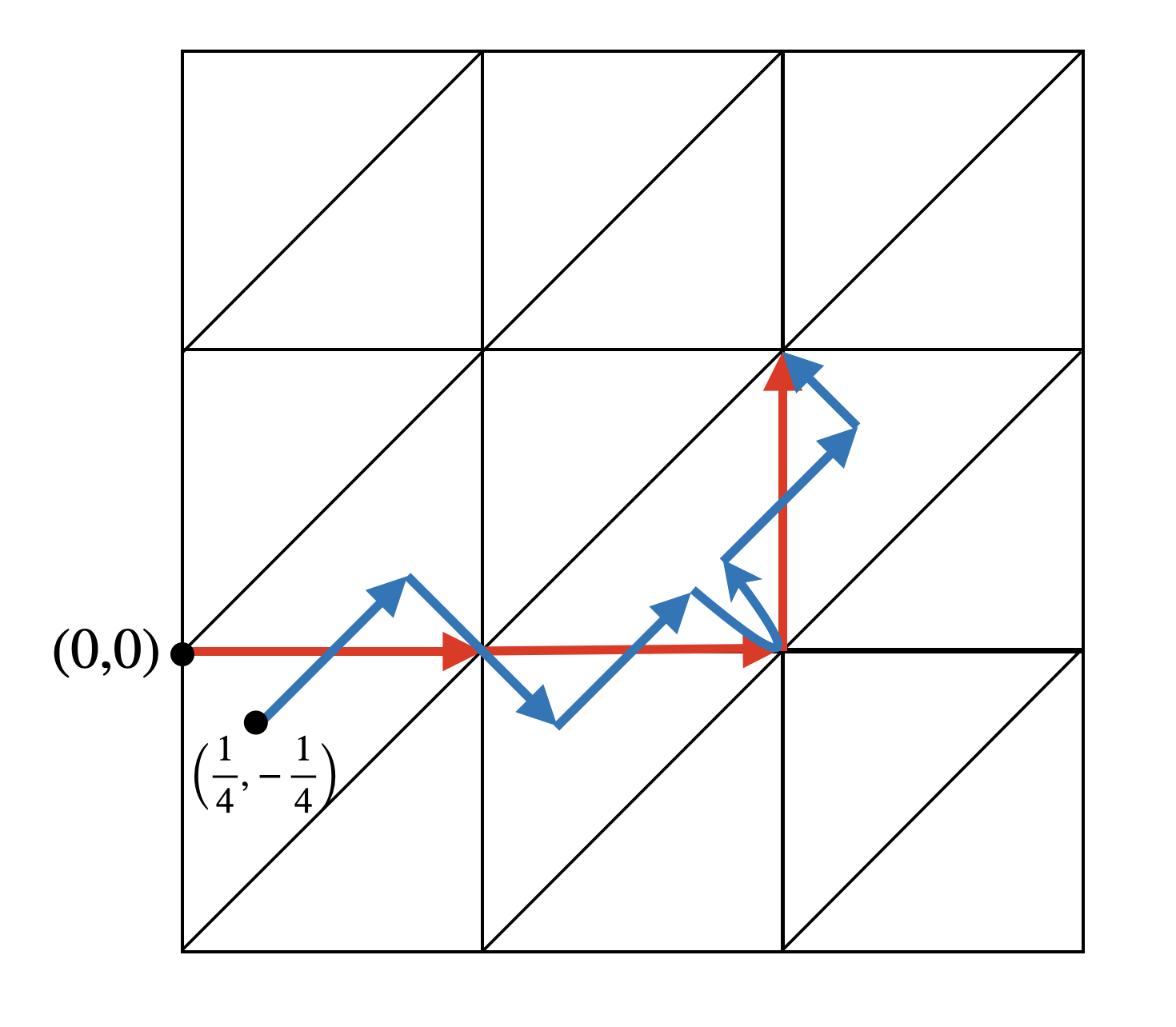}}} 
&\vcenter{\hbox{\includegraphics[scale=0.15]{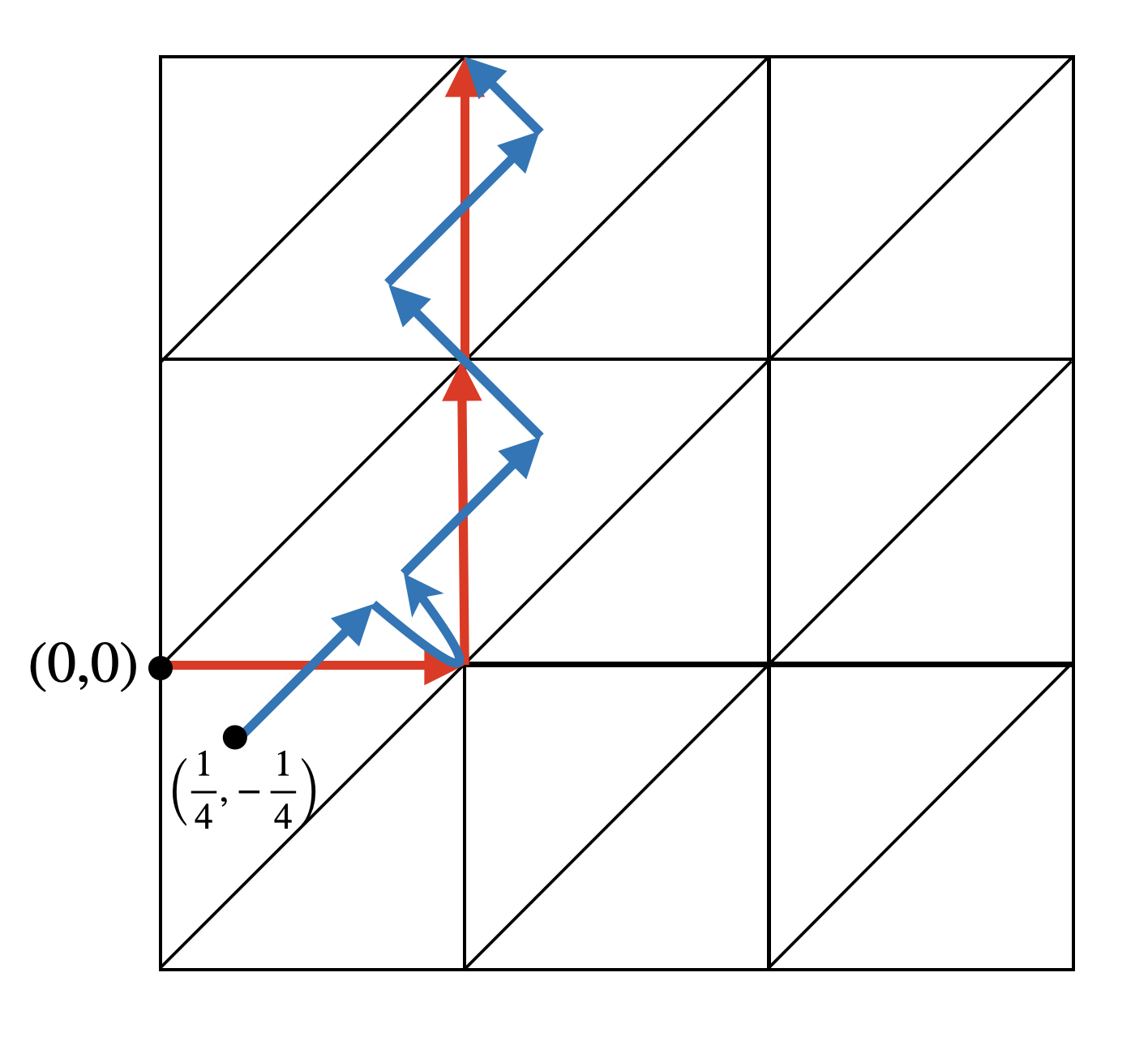}}} 
&\vcenter{\hbox{\includegraphics[scale=0.15]{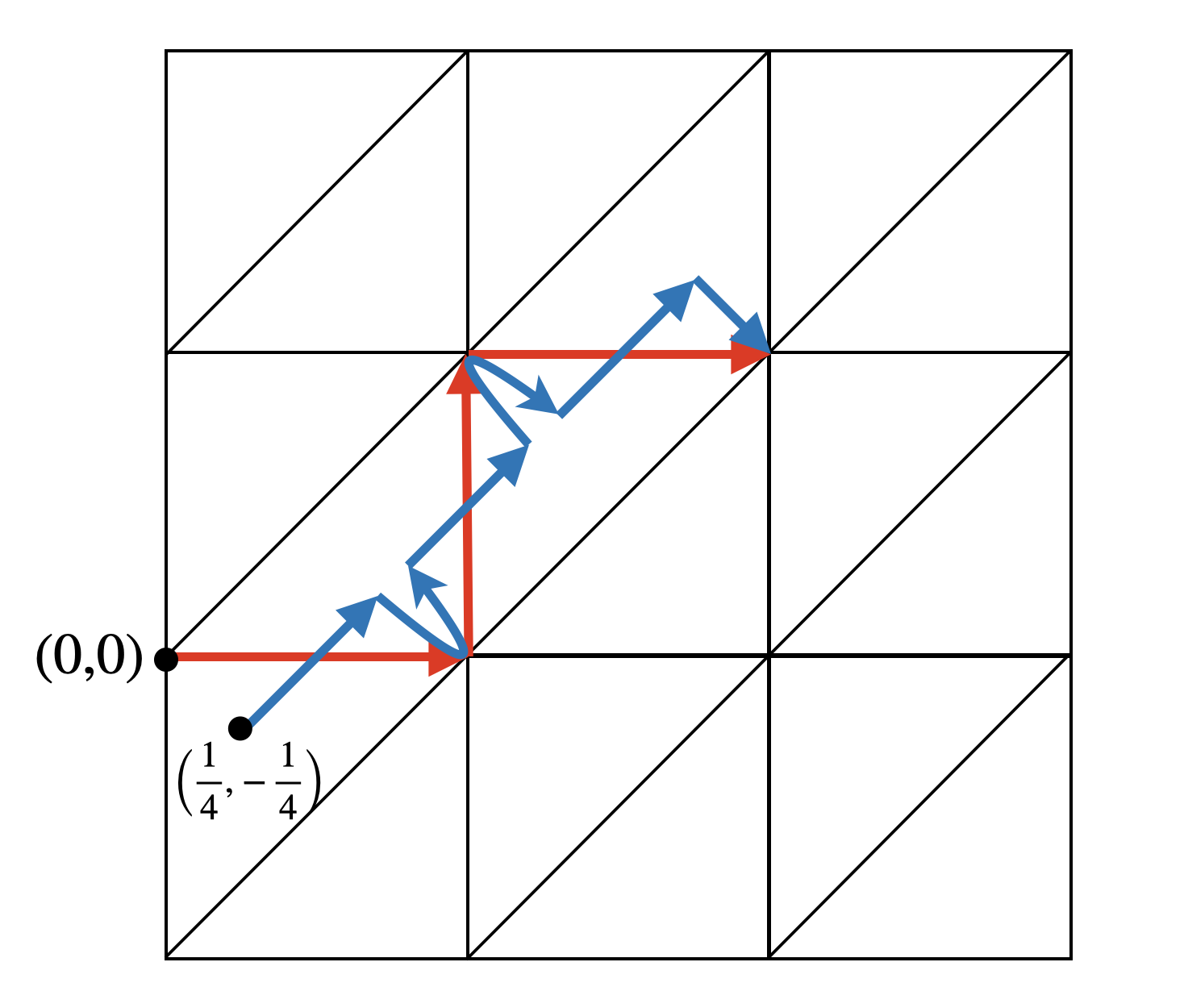}}} 
\\
\pi s_1 \pi s_1
&\pi s_1 \pi 1 \pi
&\pi 1 \pi s_1\pi
&\pi 1 \pi 1 \pi
\\
1
&q\Big( \frac{1-t}{1-qt}\Big)
&q^2\Big(\frac{1-t}{1-q^2t}\Big)
&q^2\Big(\frac{1-t}{1-qt}\Big)\Big(\frac{1-t}{1-qt}\Big)
\end{matrix}
$$
The first row contains the nonattacking fillings.
The second row contains the words of the nonattacking fillings.
The red paths are the paths corresponding to the words of the nonattacking fillings,
and the blue paths are the paths corresponding to the alcove walks.  We used a shortened notation
for the alcove walks so that
$$\begin{array}{lll}
\pi s_1 \pi s_1 \pi &\hbox{represents the alcove walk} &(1,\pi,\pi s_1, \pi s_1\pi, \pi s_1\pi s_1,
\pi s_1\pi s_1\pi), \\
\pi s_1 \pi 1 \pi &\hbox{represents the alcove walk} &(1,\pi,\pi s_1, \pi s_1\pi, \pi s_1\pi,
\pi s_1\pi^2), \\
\pi 1 \pi s_1 \pi &\hbox{represents the alcove walk} &(1,\pi,\pi, \pi^2, \pi^2 s_1,
\pi^2 s_1\pi), \\
\pi 1 \pi 1 \pi &\hbox{represents the alcove walk} &(1,\pi,\pi, \pi^2, \pi^2,
\pi^3). 
\end{array}
$$
The last row contains the weights of the alcove walks (which are the same as the weights
of the nonattacking fillings to illustrate that the factors of the form $\Big(\frac{1-t}{1-q^at^b}\Big)$ are
in bijection with the folds of the blue path.

\subsubsection{Alcove walks, nonattacking fillings and pipe dreams for $E_{(2,0.1)}$} \label{GL3pathpicturesA}

In the orthogonal projection from $\RR^3$ to the plane 
$$\{ (\gamma_1, \gamma_2, \gamma_3)\in \RR^3\ |\ \gamma_1+\gamma_2+\gamma_3=0\}$$
(so that we can draw 2-dimensional pictures), the straight line paths $x_1$, $x_2$, $x_3$ to $\varepsilon_1$, $\varepsilon_2$, $\varepsilon_3$, respectively, are pictured as
$$
\vcenter{\hbox{\includegraphics[scale=0.15]{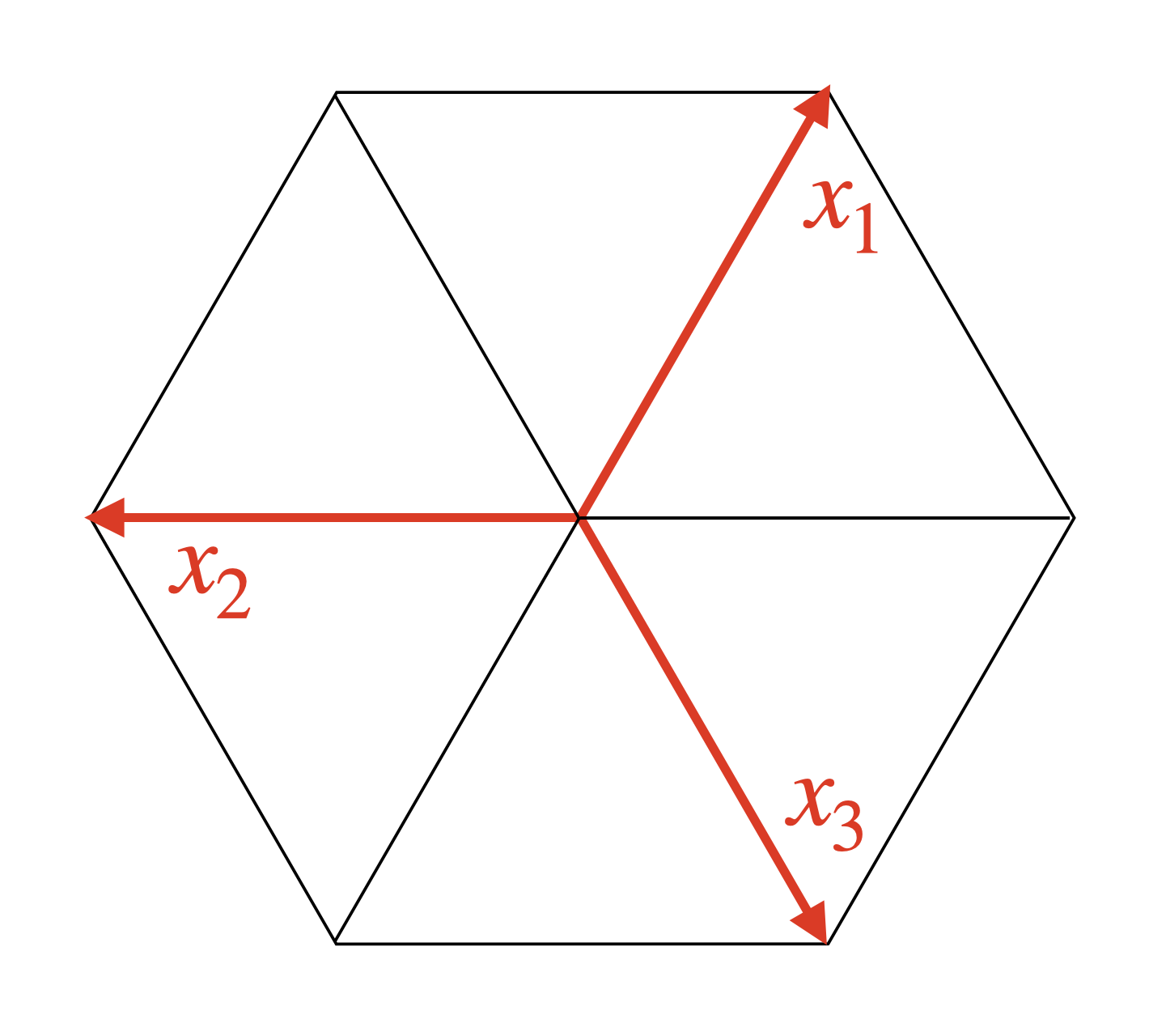}}} 
$$
The explicit expansion of $E_{(2,0,1)}$ is
$$
E_{(2,0,1)} 
=x_1x_3x_1 + \frac{1-t}{1-qt} x_1x_2x_1 
+ qt \frac{1-t}{1-qt^2} x_1x_3x_2 
+ q \frac{1-t}{1-qt} \frac{1-t}{1-qt^2} x_1x_2x_3
$$
The nonattacking fillings, words, paths, alcove walks and 
corresponding weights for $E_{(2,0,1)}$ are
$$\begin{array}{rcccc}
&\begin{array}{c|cc}
1 &1 &1 \\
2 \\
3 &3
\end{array}
&\begin{array}{c|cc}
1 &1 &1 \\
2 \\ 
3 &2
\end{array}
&\begin{array}{c|cc}
1 &1 &2 \\
2 \\
3 &3 \\
\end{array}
&\begin{array}{c|cc}
1 &1 &3 \\
2 \\
3 &2
\end{array}
\\
&\vcenter{\hbox{\includegraphics[scale=0.2]{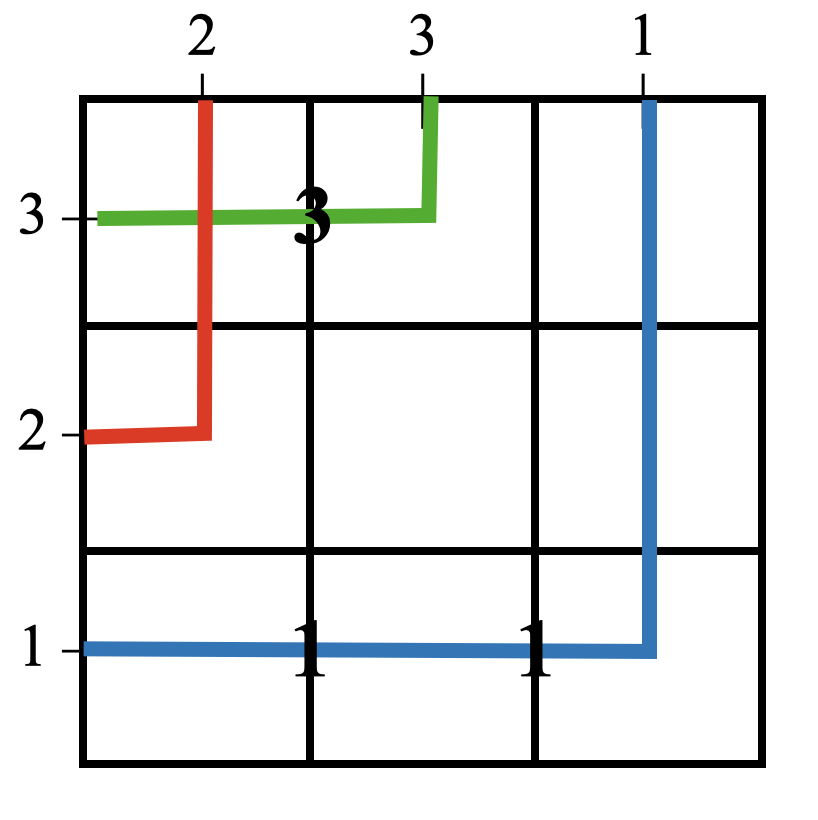}}} 
&\vcenter{\hbox{\includegraphics[scale=0.2]{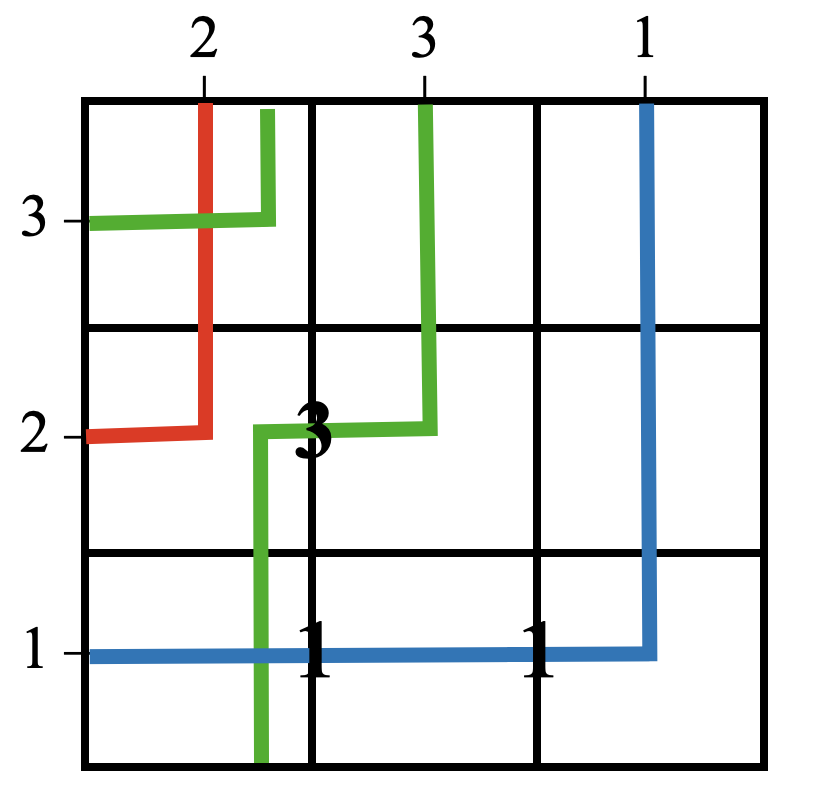}}} 
&\vcenter{\hbox{\includegraphics[scale=0.2]{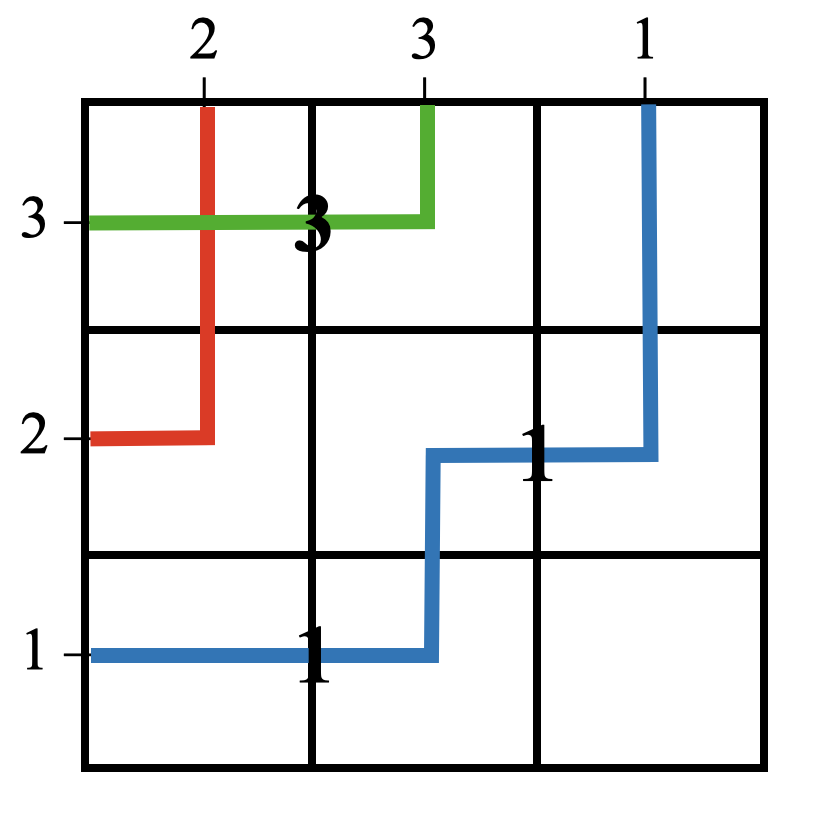}}} 
&\vcenter{\hbox{\includegraphics[scale=0.2]{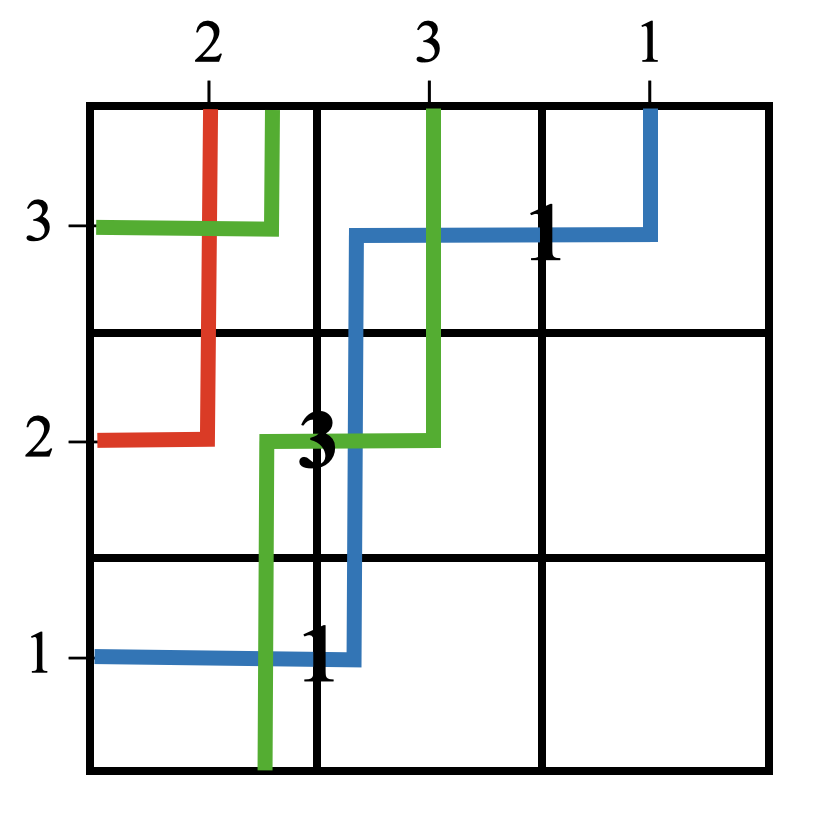}}} 
\\
&\pi s_1\pi s_1\pi
&\pi 1 \pi s_1\pi
&\pi s_1\pi 1 \pi
&\pi 1\pi 1\pi
\\
&\vcenter{\hbox{\includegraphics[scale=0.08]{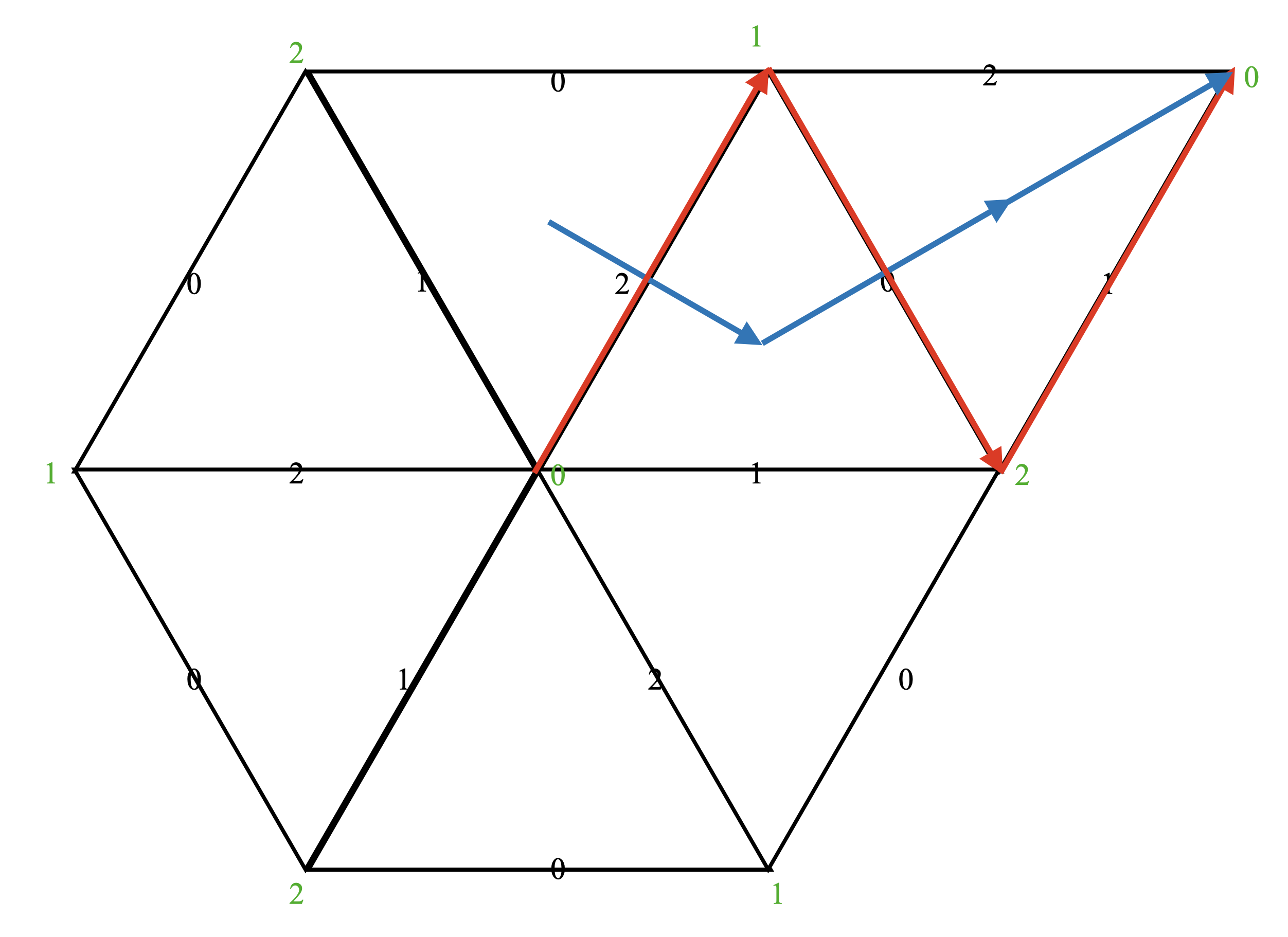}}} 
&\vcenter{\hbox{\includegraphics[scale=0.08]{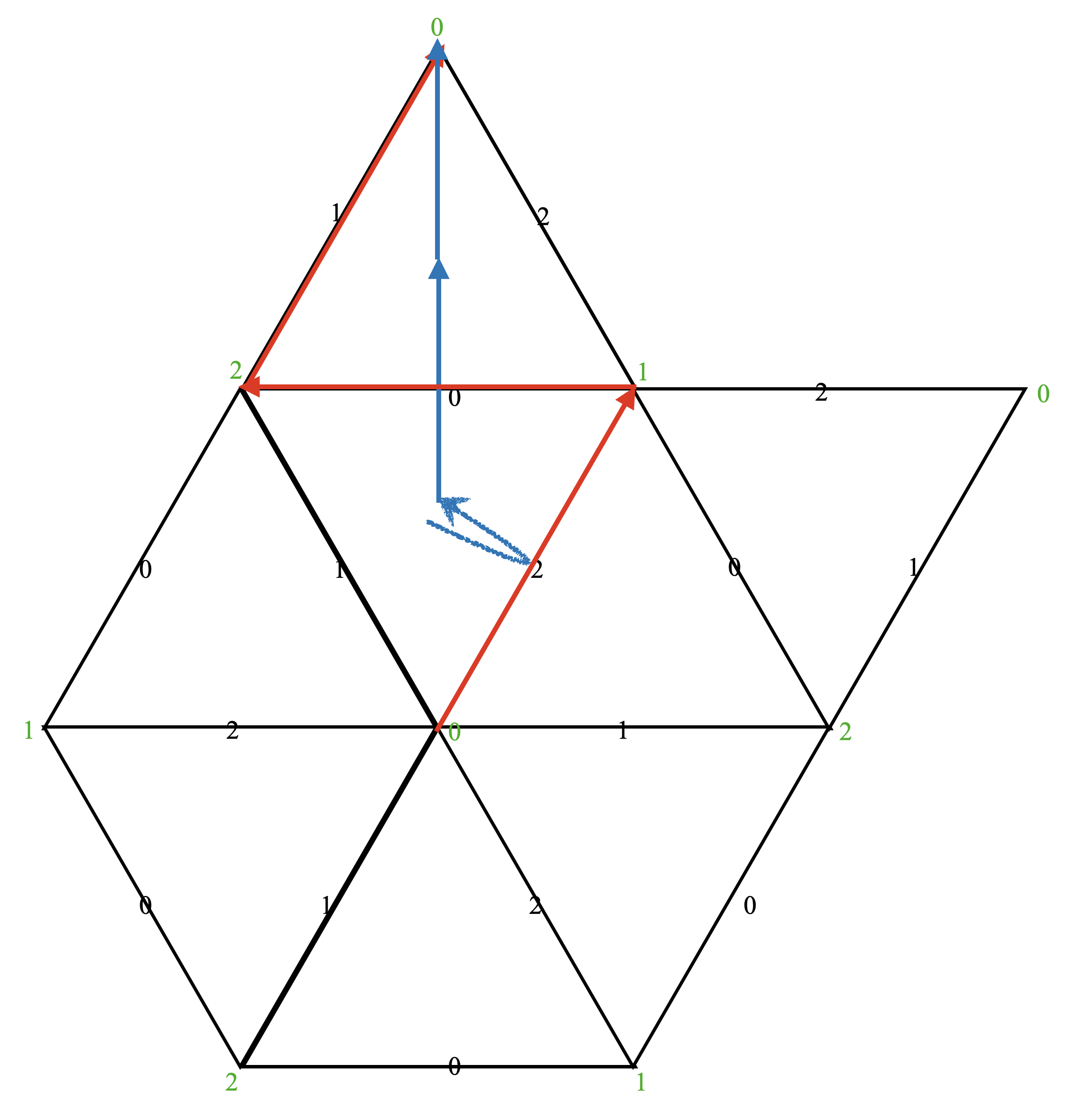}}} 
&\vcenter{\hbox{\includegraphics[scale=0.08]{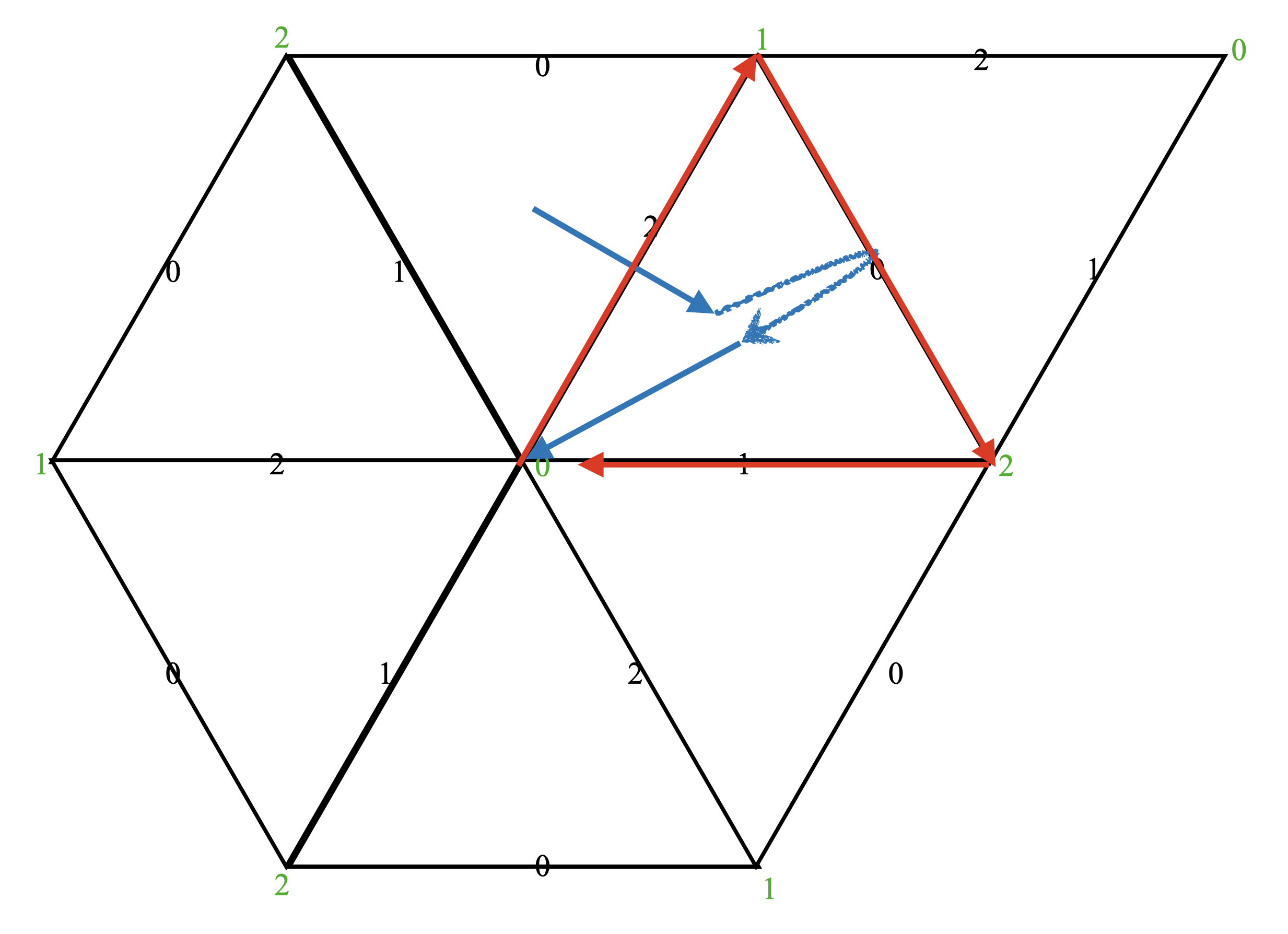}}} 
&\vcenter{\hbox{\includegraphics[scale=0.08]{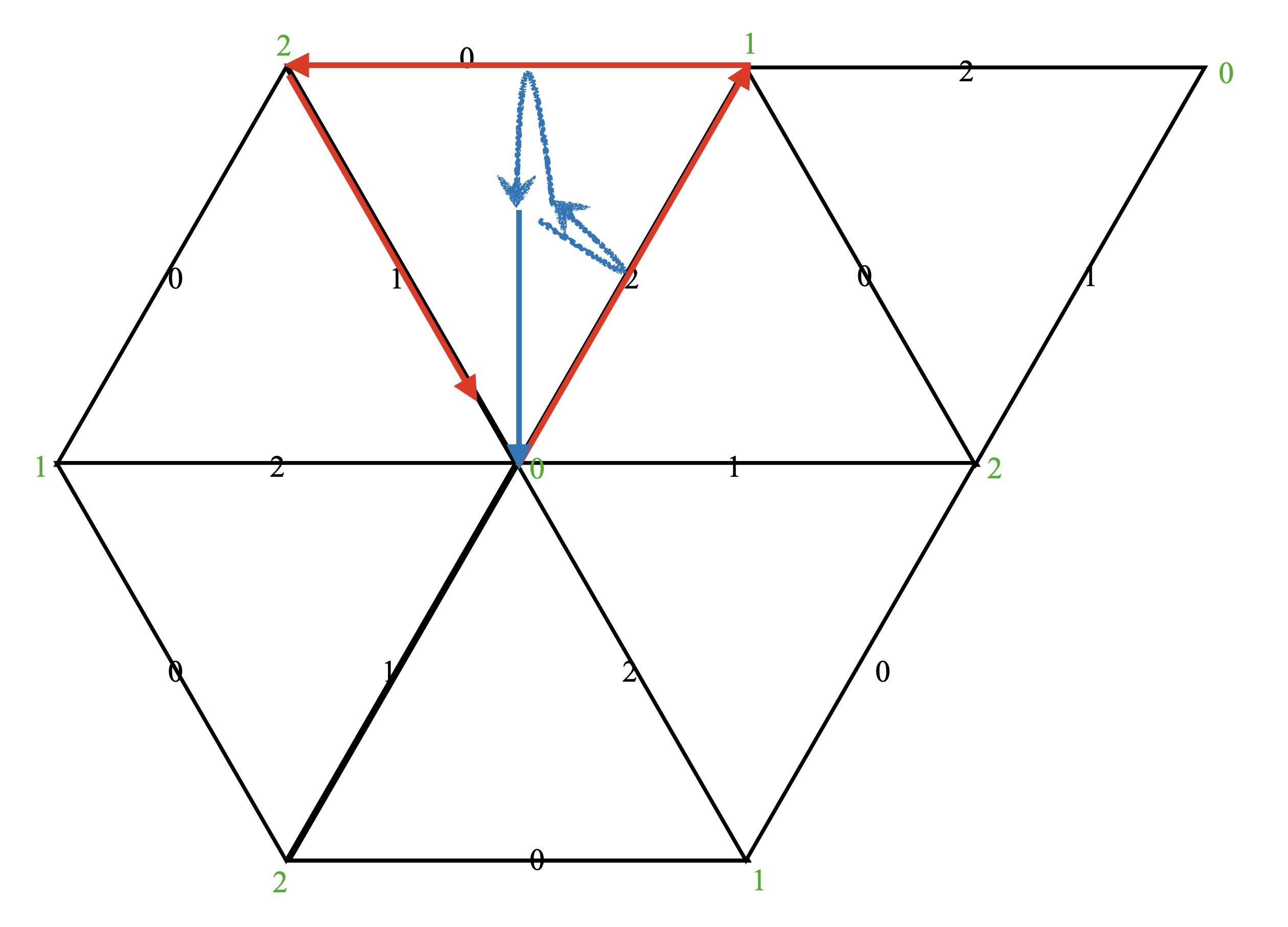}}} 
\end{array}
$$
where we have used the same shortened notation for alcove walks
as in the table in Section \ref{GL2pathpictures}.   The sections of type
$\omega$ in the paths corresponding to the alcove walks (see \eqref{pathsforAW})
are not visible in these pictures since the
pictures are in a projection orthogonal to the direction of $\omega$.



\subsubsection{Alcove walks, nonattacking fillings and pipe dreams for $E_{(1,2,0)}$}
\label{GL3pathpicturesB}

The explicit expansion of $E_{(1,2,0)}$ is
$$
E_{(1,2,0)} 
=x_1x_2x_2 + \frac{1-t}{1-qt} x_1x_2x_1 
+ q \frac{(1-qt^2)}{(1-qt)} \frac{(1-t)}{(1-qt^2)} x_1x_2x_3
$$
The nonattacking fillings, words, paths, alcove walks and 
corresponding weights for $E_{(1,2,0)}$ are
$$
\begin{array}{rcccc}
&\begin{array}{c|cc}
1 &1 \\
2 &2 &2 \\
3
\end{array}
&\begin{array}{c|cc}
1 &1 \\
2 &2 &1 \\
3
\end{array}
&\begin{array}{c|cc}
1 &1 \\
2 &2 &3 \\
3
\end{array}
\\
&\vcenter{\hbox{\includegraphics[scale=0.2]{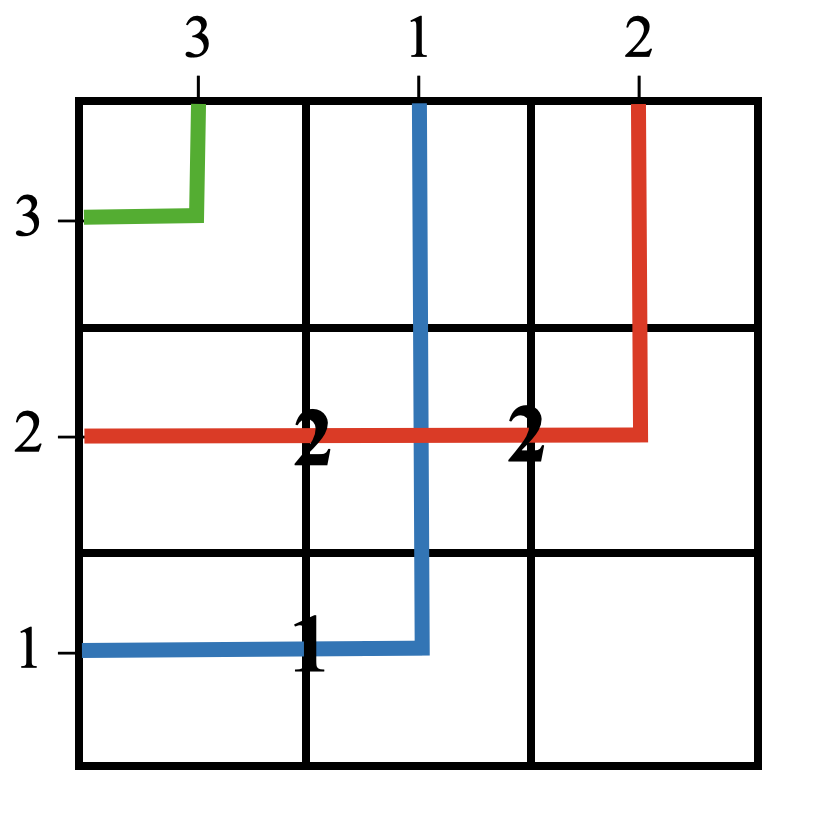}}} 
&\vcenter{\hbox{\includegraphics[scale=0.2]{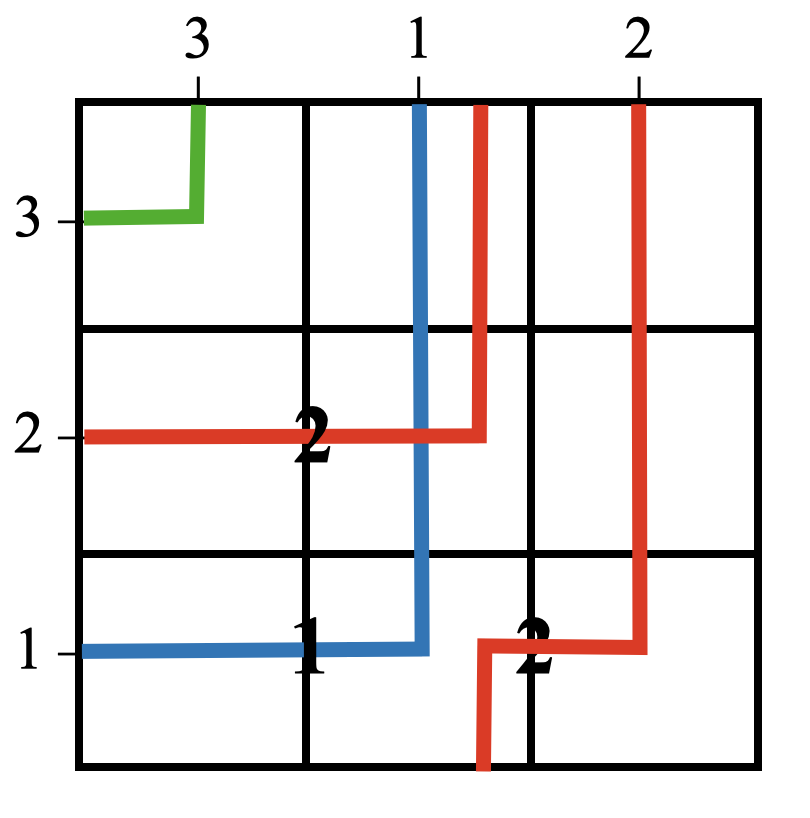}}} 
&\vcenter{\hbox{\includegraphics[scale=0.2]{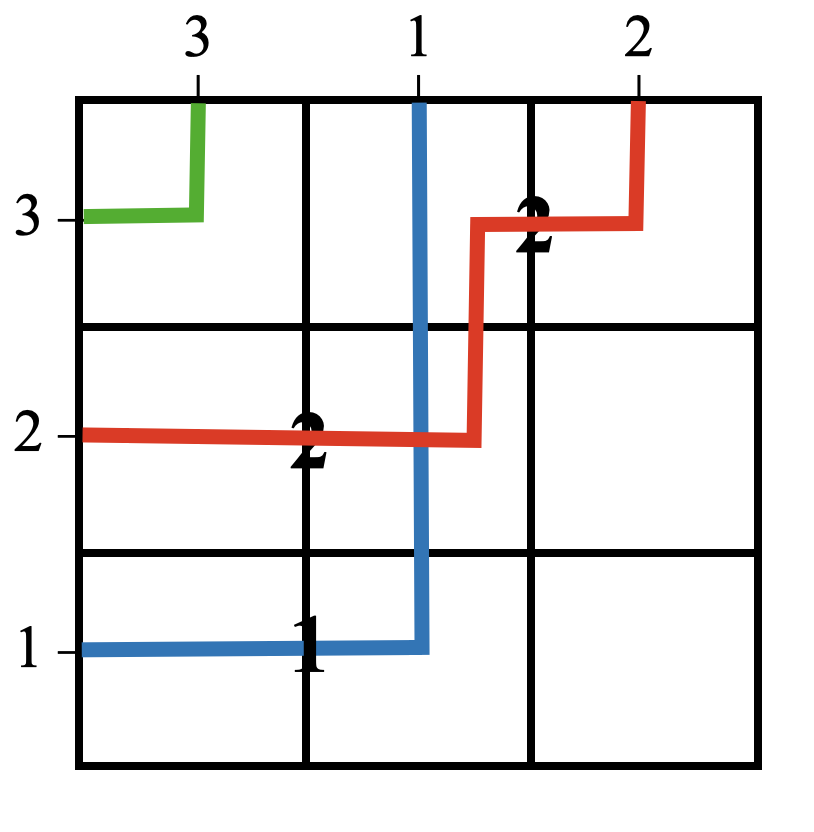}}} 
\\
&\pi\pi s_2s_1 \pi
&\pi\pi 1 s_1\pi
&\pi \pi 1 1 \pi
&\pi \pi s_2 1\pi
\\
&\vcenter{\hbox{\includegraphics[scale=0.08]{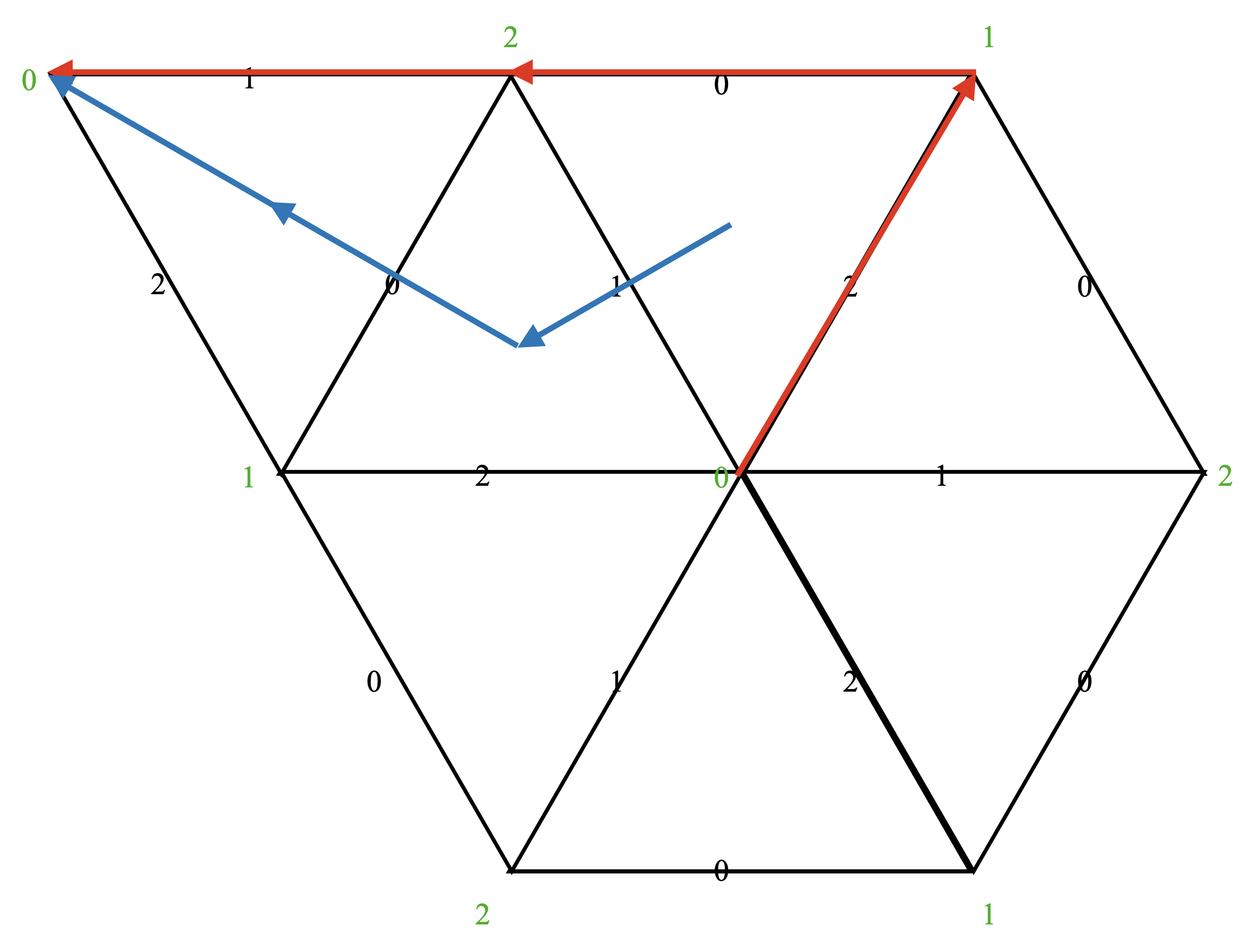}}} 
&\vcenter{\hbox{\includegraphics[scale=0.08]{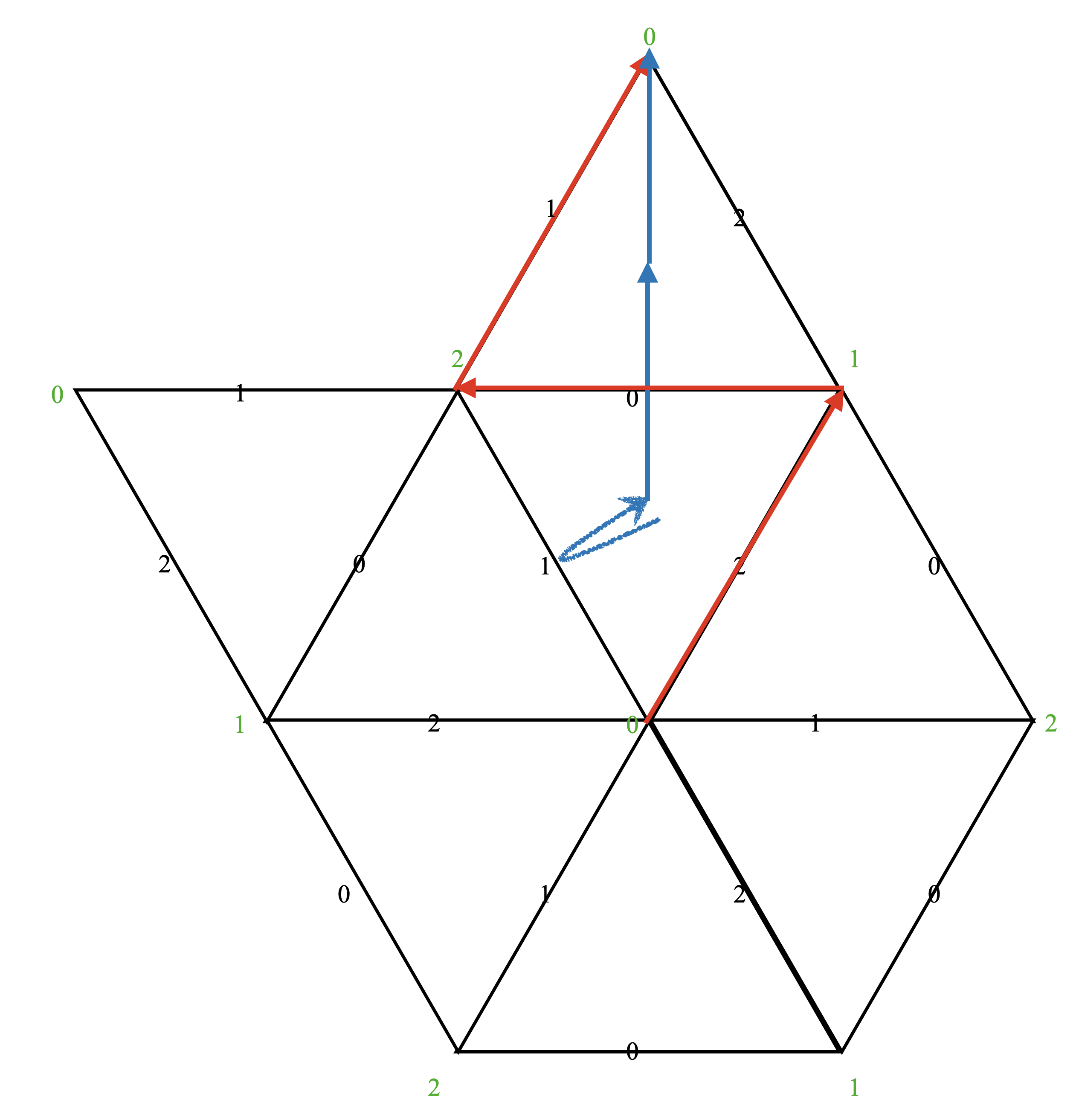}}} 
&\vcenter{\hbox{\includegraphics[scale=0.08]{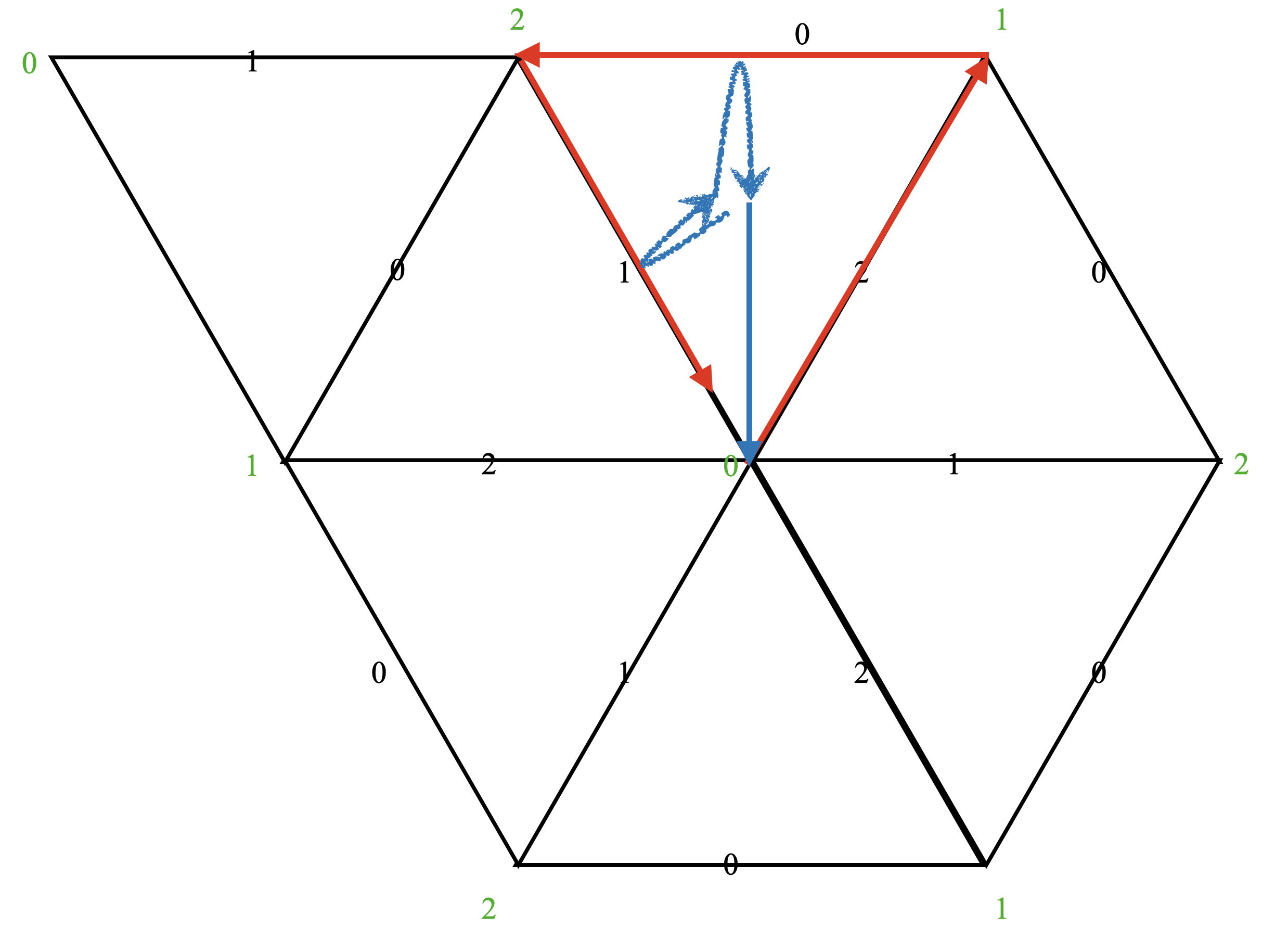}}} 
&\vcenter{\hbox{\includegraphics[scale=0.08]{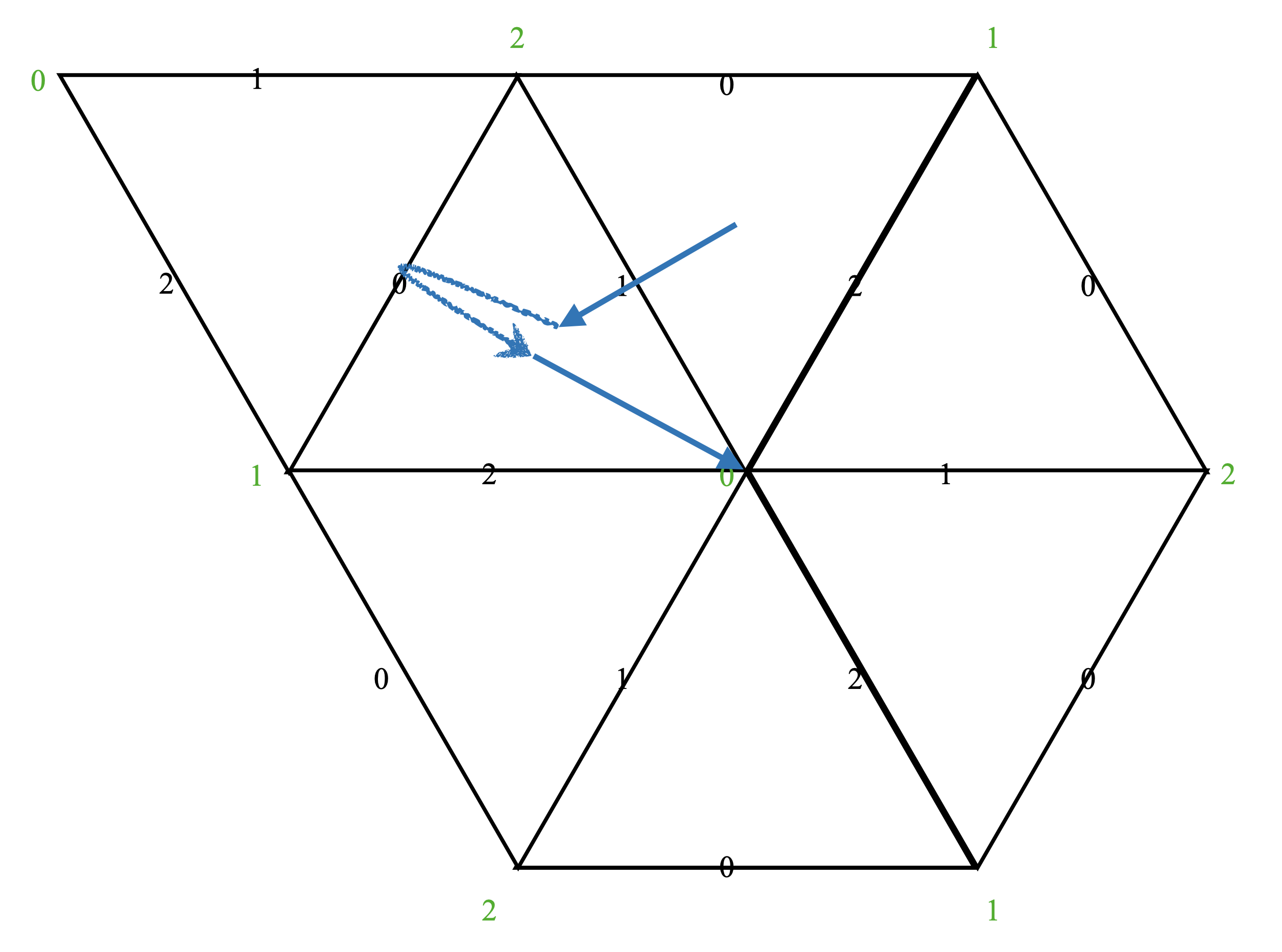}}} 
\end{array}
$$
where we have used the same shortened notation for alcove walks
as in the table in Section \ref{GL2pathpictures}.  The sections of type
$\omega$ in the paths corresponding to the alcove walks (see \eqref{pathsforAW})
are not visible in these pictures since the
pictures are in a projection orthogonal to the direction of $\omega$.
For this example, there are 4 alcove walks and 3 nonattacking fillings.


\section{Reduced words and inversions}

\subsubsection{Examples of the inversion set $\mathrm{Inv}(w)$.}

Define $n$-periodic permutations 
$\pi$ and  $s_0, s_1, \ldots, s_{n-1}\in W$ by
\begin{equation}
\pi(i) = i+1, \quad\hbox{for $i\in \ZZ$,}
\label{pidefn}
\end{equation}
\begin{equation}
\begin{array}{l}
s_i(i) = i+1, \\
s_i(i+1) = i,
\end{array}\qquad\hbox{and}\quad s_i(j) = j\ \ \hbox{for
$j\in \{0, 1, \ldots, i-1, i+2, \ldots, n-1\}.$}
\label{sidefn}
\end{equation}
An \emph{inversion} of a bijection $w\colon \ZZ\to \ZZ$ is 
$$\hbox{$(j,k)\in \ZZ\times \ZZ$\quad with\quad $j<k$ and $w(j)>w(k)$.}
$$
and the affine root corresponding to an inversion
\begin{equation}
(i,k) = (i, j+\ell n)\quad \hbox{with $i,j\in \{1, \ldots, n\}$ and $\ell\in \ZZ$,\qquad is}\quad
\beta^\vee = \varepsilon_i^\vee - \varepsilon_j^\vee+\ell K.
\label{invtoaffineroots}
\end{equation}
Let $n=3$.  The element
$$
w = s_1s_2 \quad\hbox{has}\quad w(1) = 2,\ w(2) = 3, \ w(3) = 1,
$$
and $w(1)>w(3)$ and $w(2)>w(3)$ and
$$\mathrm{Inv}(w) = \{ \alpha^\vee_2, s_2\alpha^\vee_1\}
=\{ \varepsilon^\vee_2-\varepsilon^\vee_3, \varepsilon^\vee_1-\varepsilon^\vee_3\}.$$
The element
$$
w = s_2s_1 \quad\hbox{has}\quad w(1) = 3,\ w(2) = 1, \ w(3) = 2,
$$
and $w(1)>w(2)$ and $w(1)>w(3)$ and
$$\mathrm{Inv}(w) = \{ \alpha^\vee_1, s_1\alpha^\vee_2\}
=\{ \varepsilon^\vee_1-\varepsilon^\vee_2, \varepsilon^\vee_1-\varepsilon^\vee_3\}.$$
These are examples of \cite[(2.11)]{GR21}.

\subsubsection{Relations in the affine Weyl group $W$}

The following relations are useful when working with $n$-periodic permutations.

\begin{prop}
Then
\begin{equation}
s_0 = t_{\varepsilon_1^\vee-\varepsilon_n^\vee}s_{n-1}\cdots s_2s_1s_2\cdots s_{n-1},
\qquad
t_{\varepsilon^\vee_1} = \pi s_{n-1}\cdots s_2s_1,
\label{s0andpi}
\end{equation}
\begin{equation}
\hbox{and}\qquad
t_{\varepsilon^\vee_{i+1}} = s_it_{\varepsilon^\vee_i} s_i, \qquad
\pi s_i \pi^{-1} = s_{i+1},
\label{Wsemidirect}
\end{equation}
for $i\in \{1, \ldots, n-1\}$.
\end{prop}
\begin{proof}
Proof of \eqref{s0andpi}:  If $i\not\in\{1,n\}$
$$
t_{\varepsilon_1^\vee-\varepsilon_n^\vee}s_{n-1}\cdots s_2s_1s_2\cdots s_{n-1}(i) 
t_{\varepsilon_1^\vee-\varepsilon_n^\vee}(i) = i = s_0(i).
$$
If $i=1$ then
$$t_{\varepsilon_1^\vee-\varepsilon_n^\vee}s_{n-1}\cdots s_2s_1s_2\cdots s_{n-1}(1) 
=t_{\varepsilon_1^\vee-\varepsilon_n^\vee}(n) = n-n=0 = s_0(1),
$$
and, if $i=n$ then
$$t_{\varepsilon_1^\vee-\varepsilon_n^\vee}s_{n-1}\cdots s_2s_1s_2\cdots s_{n-1}(n) 
=t_{\varepsilon_1^\vee-\varepsilon_n^\vee}(1) = 1+n = s_0(n),
$$
For $i\in \{2, \ldots, n\}$
$$\pi s_{n-1}\cdots s_1(i) = \pi (i-1) = i = t_{\varepsilon_1}(i),
\qquad\hbox{and}\qquad
\pi s_{n-1}\cdots s_1(1) = \pi (n) = n+1= t_{\varepsilon_1}(1).
$$
Proof of \eqref{Wsemidirect}:
\begin{align*}
s_it_{\varepsilon_i^\vee}s_i (i) &= s_it_{\varepsilon_i^\vee} (i+1) = s_i (i+1) 
= i = t_{\varepsilon_{i+1}^\vee}(i), \\
s_it_{\varepsilon_i^\vee}s_i (i+1) &= s_it_{\varepsilon_i^\vee} (i) = s_i (i+n) = i+1+n,
= t_{\varepsilon_{i+1}^\vee}(i+1), \\
s_it_{\varepsilon_i^\vee}s_i (j) &= s_it_{\varepsilon_i^\vee} (j) = s_i (j) = j
= t_{\varepsilon_{i+1}^\vee}(j),
\quad\hbox{if $j\in \{1, \ldots, n\}$ and $j\not\in \{i,i+1\}$.}
\end{align*}
Finally,
\begin{align*}
\pi s_i \pi^{-1}(i) &= \pi s_i (i-1) = \pi(i) = i+1 = s_{i+1}(i), \quad\hbox{and} \\
\pi s_i \pi^{-1}(i+1) &= \pi s_i (i) = \pi(i+1) = i+2 = s_{i+1}(i+1).
\end{align*}
\end{proof}

\subsubsection{The ``affine Weyl group'' and the ``extended affine Weyl group''}

The type $GL_n$ affine Weyl group $W$ is generated by $s_1, \ldots, s_n$ and $\pi$.  The
group $W$ contains also $s_0$ and all the elements $t_\mu$ for $\mu\in \ZZ^n$.
The \emph{projection homomorphism} is the group homomorphism
$\overline{\phantom{T}}\colon W \to S_n$ given by
\begin{equation}
\overline{t_\mu v}  = v, \qquad\hbox{for $\mu\in \ZZ^n$ and $v\in S_n$.}
\end{equation}

The \emph{type $PGL_n$-affine Weyl group} is the subgroup $W_{PGL_n}$ 
generated by $s_0, s_1, \ldots, s_{n-1}$.
\begin{align*}
W_{PGL_n} 
&= \{ t_\mu v\ |\ \hbox{$\mu=(\mu_1, \ldots, \mu_n)\in \ZZ^n$ 
with $\mu_1+\cdots+\mu_n=0$ and $v\in S_n$} \}, \qquad\hbox{and} \\
W_{GL_n} &= W = \{ t_\mu v\ |\ \mu\in \ZZ^n, v\in S_n\}
= \{ \pi^h w\ |\ h\in \ZZ, w\in W_{PGL_n} \}.
\end{align*}
Then
$$W_{GL_n} = \ZZ^n\rtimes S_n = \Omega \ltimes W_{PGL_n},
\qquad\hbox{where}\quad
\Omega = \{ \pi^h\ |\ h\in \ZZ\}
\quad\hbox{with}\quad \Omega \cong \ZZ.$$
The symbols $\ltimes$ and $\rtimes$ are brief notations whose purpose is to 
indicate that the  relations in \eqref{Wsemidirect} hold.   

The group $W_{PGL_n}$ is also
a quotient of $W_{GL_n}$, by the relation $\pi = 1$.
The \emph{type $SL_n$ affine Weyl group} is the quotient of $W_{GL_n}$ by the relation
$\pi^n=1$.  This is equivalent to putting a relation requiring
$$t_\mu = t_\nu
\qquad\hbox{if $\mu_i = \nu_i \bmod n$ for $i\in \{1, \ldots, n\}$.}
$$
As explained in \cite[Ch.\ 3, Exercise after Corollary 5]{St67}, there is a Chevalley group
$G_d$ for each positive integer $d$ dividing $n$.  The group $G_d$ is a central extension of
$PGL_n$ by $\ZZ/d\ZZ$ (so that $G_1 = PGL_n$ and $G_n= SL_n$).  Each of these groups
$G_d$ has an affine Weyl group $W_{G_d}$.  The group $W_{G_d}$ is the quotient of 
$W_{GL_n}$ by the relation $\pi ^d=1$, and is an extension of $W_{PGL_n}$ by $\ZZ/d\ZZ$.
The group $W_{PGL_n}$ is sometimes called the ``affine Weyl group of type $A$''
and the groups $W_{GL_n}$ and $W_{G_d}$ for $d\ne 1$ are sometimes called 
the ``extended affine Weyl groups of type $A$''.  We prefer the more specific terminologies
``affine Weyl group of type $PGL_n$'' for $W_{PGL_n}$,
``affine Weyl group of type $SL_n$'' for $W_{SL_n}$,
``affine Weyl group of type $GL_n$'' for $W_{GL_n}$,
and
``affine Weyl group of type $PGL_n \btimes (\ZZ/d\ZZ)$'' for $W_{G_d}$
(the symbol $\btimes$ indicates a central extension).

\subsubsection{The elements $u_\mu$, $v_\mu$, $z_\mu$ and $t_\mu$.}\label{uvzmudefns}

Let $\mu = (\mu_1, \ldots, \mu_n)\in \ZZ_{\ge 0}^n$ and 
let $u_\mu$ be the minimal length $n$-periodic permutation such that
$$u_\mu(0, 0, \ldots, 0) = (\mu_1, \ldots, \mu_n).$$
Let $\lambda=(\lambda, \ldots, \lambda_n)$
be the weakly decreasing rearrangement of $\mu$ and let
\begin{align*}
&\hbox{$z_\mu\in S_n$ \quad be minimal length such that}\quad 
z_\mu\lambda = \mu, \quad\hbox{and let}
\\
&\hbox{$v_\mu\in S_n$\quad be minimal length such that $v_\mu\mu$ is weakly increasing.
}
\end{align*}
Let $t_\mu\colon \ZZ\to \ZZ$ be the $n$-periodic permutation determined by
\begin{equation}
t_\mu(1) = 1+n\mu_1,\quad t_\mu(2) = 2+n\mu_2,\quad \ldots, \quad t_\mu(n) = n+n\mu_n.
\label{tmudefn}
\end{equation}

\subsubsection{Relating $u_\mu$, $v_\mu, z_\mu$ to $u_\lambda, v_\lambda, z_\lambda$.}

Let $\lambda = (\lambda_1, \ldots, \lambda_n)\in \ZZ^n$ with $\lambda_1\ge \cdots \ge \lambda_n$.
Let $S_\lambda = \{ w\in S_n\ |\ w\lambda = \lambda\}$ be the stabilizer of $\lambda$ in $S_n$.
Let
$$
\begin{array}{l}
\hbox{$w_0$ be the longest element in $S_n$,} \\
\hbox{$w_\lambda$ the longest length element in $S_\lambda$, and} \\
\hbox{$w^\lambda$ the minimal length element in the coset $w_0S_\lambda$,}
\end{array}
\qquad\hbox{so that}\qquad
\begin{array}{l}
w_0 = w^\lambda w_\lambda\ \ \hbox{and} \\
\\
\binom{n}{2} = \ell(w_0) = \ell(w^\lambda) + \ell(w_\lambda).
\end{array}
$$
Let $\mu\in \ZZ^n$ and let $\lambda$ be the decreasing rearrangement of $\lambda$.
Let $z_\mu\in S_n$ be minimal length such that $\mu = z_\mu\lambda$.
Then $z_\lambda = 1$, 
$$t_\mu = u_\mu v_\mu = (z_\mu u_\lambda) v_\mu\quad
\quad\hbox{and}\quad
t_\lambda = u_\lambda v_\lambda = u_\lambda (w^\lambda)^{-1},
\quad\hbox{with}
$$
$$\ell(t_\mu) = \ell(u_\mu)+\ell(v_\mu) = \ell(z_\mu)+ \ell(u_\lambda)+\ell(v_\mu)
\quad\hbox{and}\quad \ell(t_\lambda) = \ell(u_\lambda)+\ell((w^\lambda)^{-1}).
$$
Using that $z_\mu t_\lambda z_\mu^{-1} = t_{z_\mu \lambda} = t_\mu$ gives that
the elements $u_\mu$ and $v_\mu$
are given in terms of $z_\mu$, $u_\lambda$ and $w^\lambda$ by
$$u_\mu = z_\mu u_\lambda \quad\hbox{and}\quad v_\mu = v_\lambda z_\mu^{-1}
= (w^\lambda)^{-1}z_\mu^{-1} = (z_\mu w^\lambda)^{-1}
=(z_\mu w_0w_\lambda)^{-1} = w_\lambda w_0 z_\mu^{-1},$$
since $v_\lambda = (w^\lambda)^{-1}$ and 
$v_\lambda = v_\mu z_\mu$ with
$\ell((w_\lambda)^{-1}) = \ell(v_\lambda) = \ell(v_\mu)+\ell(z_\mu)$.

\subsubsection{Inversions of $t_{\varepsilon_1}$, $t_{-\varepsilon_1}$ and $t_{\varepsilon_2}$}

Let $t_\mu$ be as in \eqref{tmudefn} and let $\varepsilon_i = (0, \ldots, 0, 1, 0, \ldots, 0)$ where
the $1$ appears in the $i$th position.
Then
\begin{align*}
t_{\varepsilon_1} 
&= (1_1, 0_2, \ldots, 0_n) = \begin{pmatrix} 1 &2 &\cdots &n \\ n+1 &2 &\cdots &n
\end{pmatrix} = \pi s_{n-1}\cdots s_1, \\
t_{-\varepsilon_1} 
&= (-1_1, 0_2, \ldots, 0_n) = \begin{pmatrix} 1 &2 &\cdots &n \\ 1-n &2 &\cdots &n
\end{pmatrix} = s_1\cdots s_{n-1}\pi^{-1}, \\
t_{\varepsilon_1} s_1
&= (0_2, 1_1, 0_3, \ldots, 0_n) = \begin{pmatrix} 1 &2 &3 &\cdots &n \\ 2 &1+n &3 &\cdots &n
\end{pmatrix} = \pi s_{n-1}\cdots s_2, \\
s_1 t_{\varepsilon_1}
&= (1_2, 0_1, 0_3, \ldots, 0_n) = \begin{pmatrix} 1 &2 &3 &\cdots &n \\ 2+n &1 &3 &\cdots &n
\end{pmatrix} = s_1 \pi s_{n-1}\cdots s_1, \\
t_{\varepsilon_2}
&= s_1t_{\varepsilon_1}s_1 
= (0_1, 1_2, 0_3, \ldots, 0_n) = \begin{pmatrix} 1 &2 &3 &\cdots &n \\ 1 &2+n &3 &\cdots &n
\end{pmatrix} = s_1 \pi s_{n-1}\cdots s_2,
\end{align*}
and
\begin{align*}
\mathrm{Inv}(t_{\varepsilon_1})
&= \{ (1,2), (1,3), \ldots, (1,n)\} \\
&= \{\alpha^\vee_1, s_1\alpha^\vee_2, \ldots, s_1\cdots s_{n-2}\alpha^\vee_{n-1}\}
=\{\varepsilon^\vee_1-\varepsilon^\vee_2, \varepsilon^\vee_1-\varepsilon^\vee_3, \ldots,
\varepsilon^\vee_1-\varepsilon^\vee_n\} \\
\mathrm{Inv}(t_{-\varepsilon_1})
&= \{ (2-n,1), (3-n,1), \ldots, (n-n,1)\} 
= \{ (n,1+n), (n-1,1+n), \ldots, (2,1+n)\} \\ 
&= \{\pi \alpha^\vee_{n-1}, \pi s_{n-1}\alpha^\vee_{n-2}, \ldots, \pi s_{n-1} \cdots s_2\alpha^\vee_1 \} \\
&=\{\varepsilon^\vee_n-(\varepsilon^\vee_1-K), \varepsilon^\vee_{n-1}-(\varepsilon^\vee_1-K), \ldots
\varepsilon^\vee_2-(\varepsilon^\vee_1-K)\} \\
\mathrm{Inv}(t_{\varepsilon_1}s_1)
&= \{ (2,3), \ldots, (2,n)\} \\
&= \{\alpha^\vee_2, s_2\alpha^\vee_3, \ldots, s_2\cdots s_{n-2}\alpha^\vee_{n-1}\}
=\{\varepsilon^\vee_2-\varepsilon^\vee_3, \varepsilon^\vee_2-\varepsilon^\vee_4, \ldots,
\varepsilon^\vee_2-\varepsilon^\vee_n\} \\
\mathrm{Inv}(s_1t_{\varepsilon_1})
&= \{ (1,2), (1,3), \ldots, (1,n), (1-n,2)\} 
= \{ (1,2), (1,3), \ldots, (1,n), (1,2+n)\} \\
&= \{\alpha^\vee_1, s_1\alpha^\vee_2, \ldots, s_1\cdots s_{n-2}\alpha^\vee_{n-1},
s_1\cdots s_{n-2}s_{n-1}\pi^{-1}\alpha^\vee_1\} \\
&=\{\varepsilon^\vee_1-\varepsilon^\vee_2, \varepsilon^\vee_1-\varepsilon^\vee_3, \ldots,
\varepsilon^\vee_1-\varepsilon^\vee_n, (\varepsilon^\vee_1+K) - \varepsilon^\vee_2 \} \\
\mathrm{Inv}(t_{\varepsilon_2})
&= \{ ((2,3), \ldots, (2,n), (2-n,1)\} 
= \{ ((2,3), \ldots, (2,n), (2,1+n)\} \\
&= \{\alpha^\vee_2, s_2\alpha^\vee_3, \ldots, s_2\cdots s_{n-2}\alpha^\vee_{n-1},
s_2\cdots s_{n-2}s_{n-1}\pi^{-1}\alpha^\vee_1\} \\
&=\{\varepsilon^\vee_2-\varepsilon^\vee_3, \varepsilon^\vee_2-\varepsilon^\vee_4, \ldots,
\varepsilon^\vee_2-\varepsilon^\vee_n,(\varepsilon^\vee_2 + K)- \varepsilon^\vee_1 \},
\end{align*}
where we have used
\begin{align*}
s_1\cdots s_{n-1}\pi^{-1}\alpha^\vee_1
&= s_1\cdots s_{n-1}\pi^{-1}(\varepsilon^\vee_1-\varepsilon^\vee_2)
= s_1\cdots s_{n-1}((\varepsilon^\vee_n+K)-\varepsilon^\vee_1)
= (\varepsilon^\vee_1+K)-\varepsilon^\vee_2,\quad\hbox{and} \\
s_2\cdots s_{n-1}\pi^{-1}\alpha^\vee_1
&= s_2\cdots s_{n-1}((\varepsilon^\vee_n+K)-\varepsilon^\vee_1)
= (\varepsilon^\vee_2+K)-\varepsilon^\vee_1.
\end{align*}

\subsubsection{The elements $u_\mu$ and $v_\mu$ for $\mu = (0,4,5,1,4)$}
Let $u_\mu$, $v_\mu$, $z_\mu$ and $t_\mu$ be as in Section \ref{uvzmudefns}.
If $\mu = (0,4,5,1,4)$ then
$\lambda = (5,4,4,1,0)$,
and
$$z_\mu = s_2s_4s_1s_2s_3s_4
\quad\hbox{since}\quad
(5,4,4,1,0) \stackrel{s_1s_2s_3s_4}\to 
(0,5,4,4,1) \stackrel{s_4}\to 
(0,5,4,1,4)\stackrel{s_2}\to 
(0,4,5,1,4),
$$
$$v_\mu = s_4s_2s_3 = \begin{pmatrix} 1 &2 &3 &4 &5 \\
1 &3 &5 &2 &4 \end{pmatrix}
\qquad\hbox{with}\qquad
\begin{array}{rl}
v_\mu(1) &= 1 =1, \\
v_\mu(2) &= 3 =1+ \#\{1\}, \\
v_\mu(3) &= 5 =1+\#\{1,2\}+\#\{4\}, \\
v_\mu(4) &= 2 =1+\#\{1\}, \\
v_\mu(5) &= 4 =1+\#\{2,4\}, 
\end{array}
$$
Then $v_\mu =(0_1, 0_3, 0_5, 0_3, 0_4)$ and 
$$\mathrm{Inv}(v_\mu) = \{(2,4), (3,4), (3,5)\} 
= \{ \alpha^\vee_3, s_3\alpha^\vee_2, s_3s_2\alpha^\vee_4\}
=\{ \varepsilon^\vee_3-\varepsilon^\vee_4, \varepsilon^\vee_2-\varepsilon^\vee_4,
\varepsilon^\vee_3-\varepsilon^\vee_5\}.
$$
Then, with $n=5$,
\begin{align*}
v^{-1}_\mu &= \begin{pmatrix} 1 &2 &3 &4 &5 \\ 1 &4 &2 &5 &3 \end{pmatrix} 
= (0_1, 0_4, 0_2, 0_5, 0_3) \quad\hbox{and} \\
u_\mu &= t_\mu v^{-1}_\mu = (0_1, 4_3, 5_5, 1_2, 4_4) 
= \begin{pmatrix} 1 &2 &3 &4 &5 \\ 1 &4+n &2+4n &5+4n &3+5n\end{pmatrix}
= \begin{pmatrix} 1 &2 &3 &4 &5 \\ 1 &5 &10 &25 &28 \end{pmatrix}
\end{align*}

Then
$$\ell(t_\lambda) = \left(
\begin{array}{l}
\phantom{+}(5-4)+(5-4)+(5-1)+(5-0) \\
+(4-4)+(4-1)+(4-0) \\
+(4-1)+(4-0) \\
+(1-0)
\end{array}\right)
=26 = \ell(t_\mu) = \ell(u_\mu)+\ell(v_\mu)$$
with
$$\ell(u_\mu) = 6+7\cdot 2+3 = 23,
\quad
\ell(v_\mu) = 3, \quad 
\ell(z_\mu) = 6.
$$
The decreasing rearrangement of  $\mu = (0,4,5,1,4)$ is
$\lambda = (5,4,4,1,0)$ and
$$z_\lambda = 1, \quad
w_\lambda = s_2, \quad
v_\lambda = w_0s_2$$ 

\subsubsection{The box greedy reduced word for $u_\mu$.}

If $\mu= (0,4,5,1,4)$ then the box greedy reduced word for $u_\mu$ is
\begin{equation}
u^\square_\mu = (s_1\pi)^6(s_2s_1\pi)^7(s_3s_2s_1\pi) = 
\begin{array}{|ccccc}
\phantom{ \boxed{ \begin{matrix}  \phantom{T} \\ \end{matrix} }  }
\\
\boxed{ \begin{matrix} s_1\pi  \end{matrix} } 
&\boxed{ \begin{matrix}  s_1\pi \end{matrix} }
&\boxed{ \begin{matrix} s_2s_1\pi \end{matrix} }
&\boxed{ \begin{matrix} s_2s_1\pi  \end{matrix} }
\\
\boxed{\begin{matrix} s_1\pi \end{matrix} } 
&\boxed{ \begin{matrix}   s_1\pi  \end{matrix} }
&\boxed{ \begin{matrix} s_2s_1\pi \end{matrix} }
&\boxed{ \begin{matrix} s_2s_1\pi \end{matrix} }
&\boxed{\begin{matrix} s_3s_2s_1\pi \end{matrix} } 
\\
\boxed{ \begin{matrix}  s_1\pi  \end{matrix} } 
\\
\boxed{ \begin{matrix} s_1\pi \end{matrix} }
&\boxed{ \begin{matrix} s_2s_1\pi \end{matrix} }
&\boxed{ \begin{matrix} s_2s_1\pi  \end{matrix} }
&\boxed{ \begin{matrix} s_2s_1\pi \end{matrix} }
\end{array}
\label{boxgreedy04514}
\end{equation}
and the length of $u_\mu$ is
$$\ell(u_\mu) = 6+14+3=23,
\qquad\hbox{since}\quad
\ell(\pi) = 0\quad\hbox{and}\quad \ell(s_i)=1.$$
Using one-line notation for $n$-periodic permutations, 
the computation verifying the expression for $u^\square_\mu$ is
\begin{align*}
&(0_1,4_3,5_5,1_2,4_4) \stackrel{s_1}\to 
(4_3,0_1,5_5,1_2, 4_4) \stackrel{\pi^{-1}}\to 
\\
&(0_1,5_5,1_2, 4_4,3_3)) \stackrel{s_1}\to 
(5_5,0_1,1_2, 4_4,3_3)) \stackrel{\pi^{-1}}\to 
\\
&(0_1,1_2,4_4,3_3,4_5)) \stackrel{s_1}\to 
(1_2,0_1,4_4,3_3,4_5)) \stackrel{\pi^{-1}}\to 
\\
&(0_1,4_4,3_3,4_5,0_2)) \stackrel{s_1}\to 
(4_4,0_1,3_3,4_5,0_2)) \stackrel{\pi^{-1}}\to 
\\
&(0_1,3_3,4_5,0_2,3_4)) \stackrel{s_1}\to 
(3_3,0_1,4_5,0_2,3_4)) \stackrel{\pi^{-1}}\to 
\\
&(0_1,4_5,0_2,3_4,2_3)) \stackrel{s_1}\to 
(4_5,0_1,0_2,3_4,2_3)) \stackrel{\pi^{-1}}\to 
\\
&(0_1,0_2,3_4,2_3,3_5)) \stackrel{s_2}\to 
(0_1,3_4,0_2,2_3,3_5)) \stackrel{s_1}\to 
(3_4,0_1,0_2,2_3,3_5)) \stackrel{\pi^{-1}}\to 
\\
&(0_1,0_2,2_3,3_5,2_4)) \stackrel{s_2}\to 
(0_1,2_3,0_2,3_5,2_4)) \stackrel{s_1}\to 
(2_3,0_1,0_2,3_5,2_4)) \stackrel{\pi^{-1}}\to 
\\
&(0_1,0_2,3_5,2_4,1_3)) \stackrel{s_2}\to 
(0_1,3_5,0_2,2_4,1_3)) \stackrel{s_1}\to 
(3_5,0_1,0_2,2_4,1_3)) \stackrel{\pi^{-1}}\to 
\\
&(0_1,0_2,2_4,1_3,2_5)) \stackrel{s_2}\to 
(0_1,2_4,0_2,1_3,2_5)) \stackrel{s_1}\to 
(2_4,0_1,0_2,1_3,2_5)) \stackrel{\pi^{-1}}\to 
\\
&(0_1,0_2,1_3,2_5,1_4)) \stackrel{s_2}\to 
(0_1,1_3,0_2,2_5,1_4)) \stackrel{s_1}\to 
(1_3,0_1,0_2,2_5,1_4)) \stackrel{\pi^{-1}}\to 
\\
&(0_1,0_2,2_5,1_4,0_3)) \stackrel{s_2}\to 
(0_1,2_5,0_2,1_4,0_3)) \stackrel{s_1}\to 
(2_5,0_1,0_2,1_4,0_3)) \stackrel{\pi^{-1}}\to 
\\
&(0_1,0_2,1_4,0_3,1_5)) \stackrel{s_2}\to 
(0_1,1_4,0_2,0_3,1_5)) \stackrel{s_1}\to 
(1_4,0_1,0_2,0_3,1_5)) \stackrel{\pi^{-1}}\to 
\\
&(0_1,0_2,0_3,1_5,0_4)) \stackrel{s_3}\to 
(0_1,0_2,1_5,0_3,0_4)) \stackrel{s_2}\to 
(0_1,1_5,0_2,0_3,0_4)) \stackrel{s_1}\to 
(1_5,0_1,0_2,0_3,0_4)) \stackrel{\pi^{-1}}\to 
(0_1,0_2,0_3,0_4,0_5)) 
\end{align*}

\subsubsection{Inversions of $u_\mu$.}

If $\mu = (0,4,5,1,4)$ then the inversion set of $u_\mu$ is 
$$
\mathrm{Inv}(u_\mu) = \begin{array}{|ccccc}
\phantom{ \boxed{ \begin{matrix} \phantom{T} \\  \phantom{T} \\ \end{matrix} }  }
\\
\boxed{ \begin{matrix} \varepsilon^\vee_3-\varepsilon^\vee_1+4K \\ 
\phantom{T} \end{matrix} } 
&\boxed{ \begin{matrix} \varepsilon^\vee_3-\varepsilon^\vee_1+3K  \\ 
\phantom{T} \end{matrix} } 
&\boxed{ \begin{matrix} \varepsilon^\vee_3-\varepsilon^\vee_1+2K  \\ 
\varepsilon^\vee_3-\varepsilon^\vee_2+2K  \end{matrix} }
&\boxed{ \begin{matrix} \, \varepsilon^\vee_3-\varepsilon^\vee_1+K\  \\ 
\, \varepsilon^\vee_3-\varepsilon^\vee_2+K \ \end{matrix} }
\\
\boxed{\begin{matrix} \varepsilon^\vee_5-\varepsilon^\vee_1+5K \\
\phantom{T} \\ \phantom{T}  \end{matrix} } 
&\boxed{ \begin{matrix} \varepsilon^\vee_5-\varepsilon^\vee_1+4K \\
\phantom{T} \\ \phantom{T}  \end{matrix} }
&\boxed{ \begin{matrix} \varepsilon^\vee_5-\varepsilon^\vee_1+3K  \\
\varepsilon^\vee_5-\varepsilon^\vee_2+3K  \\
\phantom{T}  \end{matrix} }
&\boxed{ \begin{matrix} \varepsilon^\vee_5-\varepsilon^\vee_1+2K \\
\varepsilon^\vee_5-\varepsilon^\vee_2+2K \\
\phantom{T}  \end{matrix} }
&\boxed{\begin{matrix} \varepsilon^\vee_5-\varepsilon^\vee_1+K  \\
\varepsilon^\vee_5-\varepsilon^\vee_2+K \\
\varepsilon^\vee_5-\varepsilon^\vee_3+K   \end{matrix} } 
\\
\boxed{ \begin{matrix} \, \varepsilon^\vee_2-\varepsilon^\vee_1+K\  \\
\phantom{T}   \end{matrix} } 
\\
\boxed{ \begin{matrix} \varepsilon^\vee_4-\varepsilon^\vee_1+4K \\
\phantom{T}  \end{matrix} }
&\boxed{ \begin{matrix} \varepsilon^\vee_4-\varepsilon^\vee_1+3K \\
\varepsilon^\vee_4-\varepsilon^\vee_2+3K  \end{matrix} }
&\boxed{ \begin{matrix} \varepsilon^\vee_4-\varepsilon^\vee_1+2K \\
\varepsilon^\vee_4-\varepsilon^\vee_2+2K  \end{matrix} }
&\boxed{ \begin{matrix} \, \varepsilon^\vee_4-\varepsilon^\vee_1+K\  \\
\, \varepsilon^\vee_4-\varepsilon^\vee_2+K\  \end{matrix} }
\end{array}
$$
The following is an example that executes the last line of the proof of \cite[Proposition 2.2]{GR21}.
The factor of $s_1$ in the factorization $u_\mu = s_1\pi u_{(0,5,1,4,3)}$ gives the root
\begin{align*}
u_{(0,5,1,4,3)}^{-1}\pi^{-1}(\varepsilon^\vee_1-\varepsilon^\vee_2)
&= u_{(0,5,1,4,3)}^{-1}\pi^{-1}(\varepsilon^\vee_1-\varepsilon^\vee_2) 
= u_{(0,5,1,4,3)}^{-1}((\varepsilon^\vee_5+K) - \varepsilon^\vee_1) \\
= v_{(0,5,1,4,3)}&t_{(0,5,1,4,3)}^{-1}(\varepsilon^\vee_5 - \varepsilon^\vee_1 + K) 
= v_{(0,5,1,4,3)}( \varepsilon^\vee_5 + 3K - (\varepsilon^\vee_1 + 0K) + K) \\
&= \varepsilon^\vee_3 - \varepsilon^\vee_1 + 4 K,
\qquad
\hbox{since $v_{(0,5,1,4,3)}(5) = 3$.}
\end{align*}

\subsubsection{The column-greedy reduced word for $u_\mu$.}

Let $\mu = (\mu_1, \ldots, \mu_n)\in \ZZ_{\ge 0}^n$.
Let
$J=(j_1<\ldots<j_r)$ be the sequence of positions of the nonzero entries of $\mu$
and let $\nu$ be the composition defined by
$$\hbox{$\nu_j = \mu_j-1$\quad if $j\in J$}\qquad\hbox{and}\qquad
\hbox{$\nu_k=0$\quad if $k\not\in J$,}
$$
so that $\nu$ is the composition which has one fewer box than $\mu$ in each (nonempty) row.
Define the \emph{column-greedy reduced word} for the element $u_\mu$ 
inductively by setting 
\begin{equation}
u^\downarrow_\mu
= \Big(\prod_{m=1}^r s_{j_m-1} \cdots s_{m+1}s_m\Big)\pi^r u^\downarrow_\nu,
\label{cgredwd}
\end{equation}
where the product is taken in increasing order.

For example, if $\lambda = (5,4,4,1,0)$ 
then
$z_\lambda = 1$, $w_\lambda = s_2$,
$v_\lambda = w_0s_2$ and the column greedy reduced word for $u_\lambda$ is
$$
u^\downarrow_\lambda 
= \pi^4 s_1s_2s_3\pi^3 (s_2s_1s_3s_2s_4s_3\pi^3)^2s_2s_1\pi
=
\begin{array}{|ccccc}
\phantom{ \boxed{ \begin{matrix} \phantom{T} \\   \end{matrix} }  }
\\
\boxed{ \begin{matrix} \phantom{\pi^4}  \end{matrix} } 
&\boxed{ \begin{matrix}  s_1 \end{matrix} }
&\boxed{ \begin{matrix} s_2s_1 \end{matrix} }
&\boxed{ \begin{matrix} s_2s_1 \end{matrix} }
&\boxed{ \begin{matrix} s_2s_1 \end{matrix} }
\\
\boxed{ \begin{matrix} \phantom{\pi^4}  \end{matrix} } 
&\boxed{ \begin{matrix}   s_2 \end{matrix} }
&\boxed{ \begin{matrix} s_3s_2 \end{matrix} }
&\boxed{ \begin{matrix} s_3s_2 \end{matrix} }
\\
\boxed{ \begin{matrix} \phantom{\pi^4}  \end{matrix} } 
&\boxed{ \begin{matrix} s_3 \end{matrix} }
&\boxed{ \begin{matrix} s_4s_3 \end{matrix} }
&\boxed{ \begin{matrix} s_4s_3 \end{matrix} }
\\
\boxed{ \begin{matrix} \phantom{\pi^4}  \end{matrix} } 
\\
\begin{matrix} \pi^4  \end{matrix} 
&\begin{matrix} \pi^3  \end{matrix} 
&\begin{matrix} \pi^3  \end{matrix} 
&\begin{matrix} \pi^3  \end{matrix} 
&\begin{matrix} \pi  \end{matrix} 
\end{array}
$$
The computation verifying the expression for $u^\downarrow_\lambda$ is
\begin{align*}
(5,4,4,1,0) \stackrel{\pi^{-4}}\to 
\\
(0,4,3,3,0) \stackrel{s_1s_2s_3}\to 
(4,3,3,0,0) \stackrel{\pi^{-3}}\to
\\
(0,0,3,2,2) \stackrel{s_2s_1s_3s_2s_4s_3}\to 
(3,2,2,0,0) \stackrel{\pi^{-3}}\to
\\
(0,0,2,1,1) \stackrel{s_2s_1s_3s_2s_4s_3}\to 
(2,1,1,0,0) \stackrel{\pi^{-3}}\to
\\
(0,0,2,0,0) \stackrel{s_2s_1}\to 
(1,0,0,0,0) \stackrel{\pi^{-1}}\to
&(0,0,0,0,0) 
\end{align*}

If $\mu = (0,4,5,1,4)$ then the column greedy reduced word for $u_\mu$ is 
$$u^\downarrow_\mu 
= s_1s_2s_3s_4 \pi^4 \cdot s_1s_2s_4s_3 \pi^3 \cdot s_2s_1s_3s_2s_4s_3 \pi^3
\cdot s_2s_1s_3s_2s_4s_3 \pi^3 \cdot s_3s_2s_1\pi.
$$
This follows from \eqref{boxgreedy04514} by using that $\pi s_i \pi^{-1} = s_{i+1}$.

\section{The step-by-step and box-by-box recursions}

\subsubsection{Examples of the step-by-step recursion}
\noindent
Examples illustrating \cite[Proposition 4.1(a)]{GR21} 
are
$$
E^{(156234)}_{(1,0,0,1,0,0)} = x_1 E^{(562341)}_{(0,0,1,0,0,0)},
\qquad
E^{(516234)}_{(1,0,0,1,0,0)} = x_5 E^{(162345)}_{(0,0,1,0,0,0)},
\qquad
E^{(651234)}_{(1,0,0,1,0,0)} = x_6 E^{(512346)}_{(0,0,1,0,0,0)}.
$$
An example illustrating \cite[Proposition 4.1(b)]{GR21} 
with $zs_i<z$ is
\begin{align*}
E^{(561234)}_{(0,0,1,1,0,0)} 
&= E^{(516234)}_{(0,1,0,1,0,0)}
+ \Big( \frac{1-t}{1-qt^{5-2}}\Big) qt^{5-2} t^{-3} E^{(561234)}_{(0,1,0,1,0,0)} \\
&= E^{(516234)}_{(0,1,0,1,0,0)}
+ \Big( \frac{1-t}{1-qt^{5-2}}\Big) qE^{(561234)}_{(0,1,0,1,0,0)},
\end{align*}
with $\mu = (0,0,1,1,0,0)$ and $z = (561234)$,
$$zv^{-1}_\mu = (563412),\quad
v^{-1}_\mu = (125634), \quad
zv^{-1}_{s_2\mu} = (513462), \quad
v^{-1}_{s_2\mu} = (135624),$$
and
$$-\hbox{$\frac12$}\big(
\ell(zv^{-1}_\mu) - \ell(v^{-1}_\mu) - \ell(zv^{-1}_{s_2\mu}) + \ell(v^{-1}_{s_2\mu})\big)
=-\hbox{$\frac12$}\big( 12-4-7+5) = -\hbox{$\frac12$}\cdot 6 = -3.$$
An example illustrating \cite[Proposition 4.1(b)]{GR21} 
with $zs_i>z$ is
$$
E^{(561234)}_{(0,1,0,1,0,0)}
=  E^{(651234)}_{(1,0,0,1,0,0)} 
+ \Big(\frac{1-t}{1-qt^{5-1}}\Big)  E^{(561234)}_{(1,0,0,1,0,0)}
$$
with $\mu = (0,1,0,1,0,0)$ and $z=(561234)$,
$$zv^{-1}_\mu = (513462), \quad
v^{-1}_{\mu} = (135624), \quad
zv^{-1}_{s_1\mu} = (613452), \quad
v^{-1}_{s_1\mu} = (235614), 
$$
and
$$-\hbox{$\frac12$}\big(
\ell(zv^{-1}_\mu) - \ell(v^{-1}_\mu) - \ell(zv^{-1}_{s_1\mu}) + \ell(v^{-1}_{s_1\mu})\big)
=-\hbox{$\frac12$}\big( 7-5-8+6) = 0.$$

\subsubsection{Examples of the box by box recursion} \label{boxbyboxexamples}

An example executing the box-by-box recursion is provided just after Theorem 1.1. in \cite{GR21}.

\subsubsection{An example of a $2^{j-1}$ to $j$ term compression when $j=3$}
In order to check the powers of $t$ in \cite[Lemma 4.2]{GR21} 
compute $\tau_2^\vee\tau_1^\vee E_\gamma$,
\begin{align*}
\tau^\vee_2\tau^\vee_1 E_\gamma
&=C_{-\beta^\vee_2} (T_1+ f_{-\beta^\vee_1})E_\gamma 
= C_{-\beta^\vee_2} T_1E_\gamma
+ f_{-\beta^\vee_1}C_{-\beta^\vee_2}  E_\gamma  \\
&= C_{-\beta^\vee_2}  T_1E_\gamma
+ c_{-\beta^\vee_2} f_{-\beta^\vee_1}E_\gamma 
=(T_2+ f_{-\beta^\vee_2}) T_1E_\gamma
+ c_{-\beta^\vee_2} f_{-\beta^\vee_1}E_\gamma \\
&=T_2T_1E_\gamma + f_{-\beta^\vee_2}T_1E_\gamma 
+ t^{-\frac12} f_{-\beta^\vee_2}E_\gamma \\
&=T_2T_1E_\gamma + t^{\frac12} f_{-\beta^\vee_2}
(t^{-\frac12} T_1E_\gamma  + t^{-\frac22} E_\gamma).
\end{align*}
Now replace $T_2 = T_2^{-1}+(t^{\frac12}-t^{-\frac12})$ to get
\begin{align*}
\tau^\vee_2\tau^\vee_1 E_\gamma
&=(T^{-1}_2 +(t^{\frac12}-t^{-\frac12})) T_1E_\gamma + t^{\frac12} f_{-\beta^\vee_2}
(t^{-\frac12} T_1E_\gamma  + t^{-\frac22} E_\gamma) \\
&= T^{-1}_2 T_1E_\gamma +  (t-1+t^{\frac12} f_{-\beta^\vee_2})
t^{-\frac12} T_1E_\gamma  
+ t^{\frac12} f_{-\beta^\vee_2} t^{-\frac22} E_\gamma \\
&= T^{-1}_2 T_1E_\gamma +   t^{\frac12} f_{-\beta^\vee_2}d_{-\beta^\vee_2}
t^{-\frac12} T_1E_\gamma  
+ t^{\frac12} f_{-\beta^\vee_2} t^{-\frac22} E_\gamma,
\end{align*}
and then replacing $T_1$ in the first term by $T_1 = T_1^{-1}+(t^{\frac12}-t^{-\frac12})$
\begin{align*}
\tau^\vee_2\tau^\vee_1 E_\gamma
&= T^{-1}_2 (T^{-1} _1+t^{-\frac12}(t-1))E_\gamma +   t^{\frac12} f_{-\beta^\vee_2}d_{-\beta_2}
t^{-\frac12} T_1E_\gamma  
+ t^{\frac12} f_{-\beta^\vee_2} t^{-\frac22} E_\gamma \\
&= T^{-1}_2 T^{-1} _1E_\gamma +   t^{\frac12} f_{-\beta^\vee_2}d_{-\beta_2}
t^{-\frac12} T_1E_\gamma  + t^{-\frac12}(1-t)t^{-\frac12}E_\gamma
+ t^{\frac12}f_{-\beta^\vee_2} t^{-\frac22} E_\gamma \\
&= T^{-1}_2 T^{-1} _1E_\gamma +   t^{\frac12} f_{-\beta^\vee_2}d_{-\beta_2}
t^{-\frac12} T_1E_\gamma  
+ (t-1+t^{\frac12}f_{-\beta^\vee_2}) t^{-\frac22} E_\gamma \\
&= T^{-1}_2 T^{-1} _1E_\gamma +   t^{\frac12} f_{-\beta^\vee_2}d_{-\beta_2}
t^{-\frac12} T_1E_\gamma  
+ t^{\frac12}f_{-\beta^\vee_2} d_{-\beta^\vee_2} t^{-\frac22} E_\gamma.
\end{align*}

\subsubsection{Check of the norm statistic in the step by step recursion}
This is an example which is helpful for checking the coefficients in 
\cite[Proposition 4.3]{GR21} and its proof.  Let
$$\mu = (0,0,1,1,0,0), \qquad \gamma = (1, 0, 0, 1, 0, 0),
\qquad \nu = (0, 0, 1,0,0,0)
\qquad\hbox{and}\qquad z = y = (561234).$$
Then 
$$\begin{array}{ll}
v^{-1}_\mu = (125634),\qquad\qquad &\ell(v^{-1}_\mu) = 2+2=4, \\
yv^{-1}_\mu = (563412),\qquad\qquad &\ell(yv^{-1}_\mu) = 4+4+2+2=12, \\
v^{-1}_\gamma = (235614),\qquad\qquad &\ell(v^{-1}_\mu) = 1+1+2+2=6, \\
ys_2s_1v^{-1}_\gamma = (563412) &\ell(ys_2s_1v^{-1}_\gamma) = 4+4+2+2=12, \\
ys_1v^{-1}_\gamma = (513462) &\ell(ys_1v^{-1}_\gamma) = 4+1+1+1=7, \\
yv^{-1}_\gamma = (613452) &\ell(yv^{-1}_\gamma) = 5+1+1+1=8.
\end{array}
$$
Then $j=3$ and 
\begin{align*}
E^y_\mu
&= t^{-\frac12(\ell(yv_\mu)-\ell(v^{-1}_\mu) -(3-1)}T_y\tau^\vee_2\tau^\vee_1E_\gamma 
= t^{-\frac12(12-4-2)}T_y \tau^\vee_2 \tau^\vee_1 E_\gamma, \\
E^{ys_2s_1}_\gamma
&= t^{-\frac12(\ell(ys_2s_1v^{-1}_\gamma) - \ell(v^{-1}_\gamma)}T_{ys_2s_1}E_\gamma 
= t^{-\frac12(12-6)} T_{ys_2s_1}E_\gamma = t^{-\frac62} T_yT_2^{-1}T_1^{-1}E_\gamma \\
E^{ys_1}_\gamma
&= t^{-\frac12(\ell(ys_1v^{-1}_\gamma) - \ell(v^{-1}_\gamma)}T_{ys_1}E_\gamma 
= t^{-\frac12(7-6)} T_{ys_1}E_\gamma = t^{-\frac12} T_yT_1E_\gamma \\
E^{y}_\gamma
&= t^{-\frac12(\ell(yv^{-1}_\gamma) - \ell(v^{-1}_\gamma)}T_{y}E_\gamma 
= t^{-\frac12(8-6)} T_{y}E_\gamma = t^{-\frac22} T_y E_\gamma 
\end{align*}
so that
\begin{align*}
t^{\frac62} E^y_\mu 
&= t^{\frac62} E^{ys_2s_1}_\gamma + d_{-\beta^\vee_1}f_{-\beta^\vee_1}
t^{\frac12} E^{ys_1}_\gamma + t^{-\frac12} d_{-\beta^\vee_1}f_{-\beta^\vee_1} t^{\frac22}E^y_\gamma \\
&= t^{\frac62} E^{ys_2s_1}_\gamma + \frac{1-t}{1-qt^{5-2}}qt^{5-2}
E^{ys_1}_\gamma + \frac{1-t}{1-qt^{5-2}}qt^{5-2} E^y_\gamma 
\end{align*}
giving
$$E^y_\mu 
= E^{ys_2s_1}_\gamma + \frac{1-t}{1-qt^{5-2}}q
E^{ys_1}_\gamma + \frac{1-t}{1-qt^{5-2}}q E^y_\gamma
$$
as in the second line of the example in \ref{boxbyboxexamples}.

\subsubsection{Check of the statistic for $E^z_{\varepsilon_j}$ where $z(j) = j+k$}
This is an example of
\cite[Proposition 4.3]{GR21} with 
$$\mu = \varepsilon_j, \quad \gamma = \varepsilon_1, \quad
y = s_{j+(k-1)}\cdots s_j.$$
Then
$$v_\mu = s_{n-1}\cdots s_j, \quad 
v_\gamma = s_{n-1}\cdots s_1,\quad 
v^{-1}_\mu = s_j\cdots s_{n-1}, \quad 
v^{-1}_\gamma = s_1\cdots s_{n-1}.$$
Then
$yv^{-1}_\mu = s_{j+k}\cdots s_{n-1}$ and  
$\ell(yv^{-1}_\mu) = (n-1)-(j-1)-k$ and
$$
\ell(yv^{-1}_\mu) -\ell(v^{-1}_\mu) -(j-1) = ((n-1)-(j-1)-k) - ((n-1)-(j-1)) = -k-(j-1).
$$
Then $yc^{-1}_ac_jv^{-1}_\mu = ((s_{j+(k-1)}\cdots s_j)(s_a\cdots s_{j-1})(s_j\cdots s_{n-1})$
and 
$$\ell(yc^{-1}_ac_jv^{-1}_\mu) =  ( j-1+k - (j-1)) + ( (j-1)-(a-1)) + (n-1-(j-1)) = (n-1)-(a-1)+k.$$
So
\begin{align*}
\ell(yc^{-1}_ac_jv^{-1}_\mu) - \ell(yv^{-1}_\mu) - \ell(c^{-1}_ac_j)
&= (n-1)-(a-1)+k - ((n-1)-(j-1)-k) - ((j-1)-(a-1)) \\
&= 2k.
\end{align*}
Thus
\begin{align*}
E^z_\mu = E^y_\mu 
&=  x_{y(j)} E^{yc_n}_\nu
+ \frac{(1-t)}{1-q^{\mu_j}t^{v_\mu(j)-(j-1) }}
\sum_{a=0}^{j-1} 
t^{\frac12\cdot 2k} x_{y(a) }  E^{yc^{-1}_a c_n}_\nu \\
&=  x_{y(j)} 
+ \frac{(1-t)}{1-q^{\mu_j}t^{v_\mu(j)-(j-1) }}
\sum_{a=0}^{j-1} 
t^k x_{y(a) }.
\end{align*}

\section{Type $GL_n$ DAArt, DAHA and the polynomial representation}\label{GLnDAHAsection}

\subsubsection{Example to check the eigenvalues of $Y_i$ on $E_\mu$}
The box greedy reduced words for $u_{(2,1,0)}$, $u_{(2,0,1)}$ and $u_{(1,2,0)}$ are
$$
u^\square_{(2,1,0)} = \begin{array}{|cc}
\boxed{\pi} &\boxed{s_1\pi} \\
\boxed{\pi} \\
\phantom{T}
\end{array}
\quad
u^\square_{(2,0,1)} = \begin{array}{|cc}
\boxed{\ \pi\ } &\boxed{s_1\pi} \\
 \\
\boxed{s_1\pi}
\end{array}
\quad
u^\square_{(1,2,0)} = \begin{array}{|cc}
\boxed{\pi} \\
\boxed{\pi} &\boxed{s_2s_1\pi} \\
\phantom{T}
\end{array}
$$
Using $u_\mu= t_\mu v^{-1}_\mu$ to carefully compute $v^{-1}_\mu$:
\begin{align*}
u_{(2,1,0)} &= \pi^2s_1\pi = 
t_{\varepsilon_1}s_1s_2 t_{\varepsilon_1}s_1s_2s_1 t_{\varepsilon_1}s_1s_2 \\
&= t_{\varepsilon_1}t_{\varepsilon_2}s_1s_2 s_1s_2s_1 t_{\varepsilon_1}s_1s_2 \\
&= t_{\varepsilon_1}t_{\varepsilon_2}s_2 t_{\varepsilon_1}s_1s_2 \\
&= t_{2\varepsilon_1+\varepsilon_2}s_2 s_1s_2,
\quad\hbox{so}\quad v^{-1}_{(2,1,0)} = s_2s_1s_2.
\end{align*}
\begin{align*}
u_{(2,0,1)} &= \pi s_1\pi s_1\pi = 
t_{\varepsilon_1}s_1s_2 s_1 t_{\varepsilon_1}s_1s_2s_1 t_{\varepsilon_1}s_1s_2 \\
&= t_{\varepsilon_1}t_{\varepsilon_3}s_1s_2 s_1s_1 s_2s_1 t_{\varepsilon_1}s_1s_2 \\
&= t_{2\varepsilon_1+\varepsilon_3}s_1s_2,
\quad\hbox{so}\quad v^{-1}_{(2,0,1)} = s_1s_2.
\end{align*}
\begin{align*}
u_{(1,2,0)} &= \pi^2 s_2 s_1\pi = 
t_{\varepsilon_1}s_1s_2 t_{\varepsilon_1}s_1s_2s_2 s_1 t_{\varepsilon_1}s_1s_2 \\
&= t_{\varepsilon_1+2\varepsilon_2}s_1s_2 s_1 s_2 \\
&= t_{\varepsilon_1+2\varepsilon_2}s_2s_1,
\quad\hbox{so}\quad v^{-1}_{(1,2,0)} = s_2s_1.
\end{align*}
Using
$$\begin{array}{ll}
u_{(2,1,0)} = t_{(2,1,0)}s_1s_2s_1 = t_{(2,1,0)}v^{-1}_{(2,1,0)}, \qquad 
&u_{(2,0,1)} = t_{(2,0,1)}s_1s_2 = t_{(2,0,1)}v^{-1}_{(2,0,1)}, \\
u_{(1,2,0)} = t_{(1,2,0)}s_2s_1 = t_{(1,2,0)}v^{-1}_{(1,2,0)}, 
&u_{(0,2,1)} = t_{(0,2,1)}s_2 = t_{(0,2,1)}v^{-1}_{(0,2,1)}, \\
u_{(1,0,2)} = t_{(1,0,2)}s_2, = t_{(1,0,2)}v^{-1}_{(1,0,2)}, 
&u_{(0,1,2)} = t_{(0,1,2)} = t_{(0,1,2)}v^{-1}_{(0,1,2)},
\end{array}
$$
and the relations
$$Y_1 \tau_{\pi}^\vee = q^{-1}\tau_\pi^\vee Y_3, \qquad
Y_2 \tau_{\pi}^\vee = \tau_\pi^\vee Y_1, \qquad
Y_3 \tau_{\pi}^\vee = \tau_\pi^\vee Y_2,
$$
then
\begin{align*}
Y_1E_{(2,1,0)} 
&= t^{-\frac32} Y_1 \tau_\pi^\vee \tau_\pi^\vee \tau_1^\vee \tau_\pi^\vee \mathbf{1} 
= t^{-\frac32} q^{-1} \tau_\pi^\vee Y_3\tau_\pi^\vee \tau_1^\vee \tau_\pi^\vee \mathbf{1} 
= t^{-\frac32} q^{-1} \tau_\pi^\vee \tau_\pi^\vee Y_2 \tau_1^\vee \tau_\pi^\vee \mathbf{1} \\
&= t^{-\frac32} q^{-1} \tau_\pi^\vee \tau_\pi^\vee  \tau_1^\vee Y_1 \tau_\pi^\vee \mathbf{1} 
= t^{-\frac32} q^{-2} \tau_\pi^\vee \tau_\pi^\vee  \tau_1^\vee  \tau_\pi^\vee Y_3\mathbf{1} 
= q^{-2}t^{-(3-1)+\frac12(3-1)} E_{(2,1,0)} ,\\
Y_2E_{(2,1,0)} 
&= t^{-\frac32} Y_2 \tau_\pi^\vee \tau_\pi^\vee \tau_1^\vee \tau_\pi^\vee \mathbf{1} 
= t^{-\frac32} \tau_\pi^\vee Y_1\tau_\pi^\vee \tau_1^\vee \tau_\pi^\vee \mathbf{1} 
= t^{-\frac32} q^{-1} \tau_\pi^\vee \tau_\pi^\vee Y_3 \tau_1^\vee \tau_\pi^\vee \mathbf{1} \\
&= t^{-\frac32} q^{-1} \tau_\pi^\vee \tau_\pi^\vee  \tau_1^\vee Y_3 \tau_\pi^\vee \mathbf{1} 
= t^{-\frac32} q^{-1} \tau_\pi^\vee \tau_\pi^\vee  \tau_1^\vee  \tau_\pi^\vee Y_2\mathbf{1} 
= q^{-1}t^{-(2-1)+\frac12(3-1)} E_{(2,1,0)} ,\\
Y_3E_{(2,1,0)} 
&= t^{-\frac32} Y_3 \tau_\pi^\vee \tau_\pi^\vee \tau_1^\vee \tau_\pi^\vee \mathbf{1} 
= t^{-\frac32}  \tau_\pi^\vee Y_2\tau_\pi^\vee \tau_1^\vee \tau_\pi^\vee \mathbf{1} 
= t^{-\frac32}  q^{-1} \tau_\pi^\vee \tau_\pi^\vee Y_1 \tau_1^\vee \tau_\pi^\vee \mathbf{1} \\
&= t^{-\frac32}  q^{-1} \tau_\pi^\vee \tau_\pi^\vee  \tau_1^\vee Y_2 \tau_\pi^\vee \mathbf{1} 
= t^{-\frac32} q^{-1} \tau_\pi^\vee \tau_\pi^\vee  \tau_1^\vee  \tau_\pi^\vee Y_1\mathbf{1} 
= t^{-(1-1)+\frac12(3-1)} E_{(2,1,0)}.
\end{align*}
Then
\begin{align*}
Y_1E_{(1,2,0)} 
&= t^{\frac12}Y_1 \tau_1^\vee E_{(2,1,0)}
= t^{\frac12} \tau_1^\vee Y_2 E_{(2,1,0)} = q^{-1} t^{-(2-1)+\frac12(3-1)}E_{(1,2,0)}, \\
Y_2E_{(1,2,0)} 
&= t^{\frac12} Y_2 \tau_1^\vee E_{(2,1,0)}
= t^{\frac12} \tau_1^\vee Y_1 E_{(2,1,0)} = q^{-2} t^{-(3-1)+\frac12(3-1)}E_{(1,2,0)}, \\
Y_3E_{(1,2,0)} 
&= t^{\frac12} Y_3 \tau_1^\vee E_{(2,1,0)}
= t^{\frac12} \tau_1^\vee Y_3 E_{(2,1,0)} =  q^{-0}t^{-(1-1)+\frac12(3-1)}E_{(1,2,0)},
\end{align*}
and
$v_{(1,2,0)}(1) = s_1s_2(1) = s_1(1) = 2$,
$v_{(1,2,0)}(2) = s_1s_2(2) = s_1(3) = 3$ and
$v_{(1,2,0)}(3) = s_1s_2(3) = s_1(2) = 1$.

\subsubsection{The elements $X^{\omega_r}$.}

For $i\in \{1, \ldots, n\}$ let $\omega_i= \varepsilon_1+\cdots +\varepsilon_i$.  
Then
$$X^{\omega_i} = X^{\varepsilon_1+\cdots+\varepsilon_i} = (g^\vee)^i T_{w_i}^{-1},
\qquad\hbox{where}\quad
w_i = \begin{pmatrix} 1 &\cdots &i &{i+1}&\cdots &n \\
i+1 &\cdots &n &1 &\cdots &i
\end{pmatrix}
$$
In $W$, the element $t_{\omega_i} = \pi^i w_i$.
There are two favorite choices of reduced word for $w_i$,
which are 
\begin{align*}
w_i &= (s_i\cdots s_{n-1})(s_{i-1}\cdots s_{n-2}) \cdots (s_1\cdots s_{n-i}) 
= (s_i\cdots s_1)(s_{i+1}\cdots s_2)\cdots (s_{n-1}\cdots s_{n-i})
\end{align*}
For example, if $n=6$ then
\begin{align*}
w_1 &= s_5s_4s_3s_2s_1, \\
w_2 &= (s_4s_3s_2s_1)(s_5s_4s_3s_2) = (s_4s_5)(s_3s_4)(s_2s_3)(s_1s_2)  \\
w_3 &= (s_3s_2s_1)(s_4s_3s_2)(s_5s_4s_3) = (s_3s_4s_5)(s_2s_3s_4)(s_1s_2s_3)  \\
w_4 &= (s_2s_1)(s_3s_2)(s_4s_3)(s_5s_4) = (s_2s_3s_4s_5)(s_1s_2s_3s_4)  \\
w_5 &= s_1s_2s_3s_4s_5 \\
w_6 &= 1,
\end{align*}
and
\begin{align*}
X^{\omega_1} &= g^\vee T_5^{-1}T_4^{-1}T_3^{-1}T_2^{-1}T_1^{-1}, \\
X^{\omega_2} &=(g^\vee)^2 (T_4^{-1}T_3^{-1}T_2^{-1}T_1^{-1})
(T_5^{-1}T_4^{-1}T_3^{-1}T_2^{-1}) \\
&= (g^\vee)^2 (T_4^{-1}T_5^{-1})(T_3^{-1}T_4^{-1})(T_2^{-1}T_3^{-1}) (T_1^{-1}T_2^{-1}) \\
X^{\omega_3} 
&=(g^\vee)^3 (T_3^{-1}T_2^{-1}T_1^{-1})(T_4^{-1}T_3^{-1}T_2^{-1})
(T_5^{-1}T_4^{-1}T_3^{-1}) \\
&= (g^\vee)^3 (T_3^{-1}T_4^{-1}T_5^{-1})(T_2^{-1}T_3^{-1}T_4^{-1})(T_1^{-1}T_2^{-1}T_3^{-1}) \\
X^{\omega_4} 
&=(g^\vee)^4 (T_2^{-1}T_1^{-1})(T_3^{-1}T_2^{-1}) (T_4^{-1}T_3^{-1})(T_5^{-1}T_4^{-1}) \\
&= (g^\vee)^4 (T_2^{-1}T_3^{-1}T_4^{-1}T_5^{-1})
(T_1^{-1}T_2^{-1}T_3^{-1}T_4^{-1}) \\
X^{\omega_5} 
&=(g^\vee)^5  T_1^{-1}T_2^{-1}T_3^{-1}T_4^{-1}T_5^{-1} \\
X^{\omega_6} &= (g^\vee)^6.
\end{align*}

\subsubsection{Type $GL_2$}\label{GL2examples}

For type $GL_2$,
$X_1 = g^\vee T_1^{-1}$ and $X_2 = T_1X_1T_1 = T_1g^\vee$ and
$$X_1X_2 = (g^\vee)^2, \quad
X_1^{k+1} T_1 = (g^\vee T_1^{-1})^kg^\vee,
\quad
(T_1g^\vee)^k = X_2^k.$$
The box greedy reduced words for the first few cases are
$$
u^\square_{(1,0)} = \begin{array}{|c}
\boxed{\pi} \\
\ 
\end{array}
\qquad\quad
u^\square_{(0,1)} = \begin{array}{|c}
 \\
\boxed{s_1\pi} 
\end{array}
$$

\smallskip
$$
u^\square_{(2,0)} = \begin{array}{|cc}
\boxed{\pi} &\boxed{s_1\pi} \\
\ 
\end{array}
\qquad\quad
u^\square_{(1,1)} = \begin{array}{|c}
\boxed{\pi} \\
\boxed{\pi} 
\end{array}
\qquad\quad
u^\square_{(0,2)} = \begin{array}{|cc}
\ \\
\boxed{s_1\pi} &\boxed{s_1\pi} 
\end{array}
$$

\medskip
$$
u^\square_{(3,0)} = \begin{array}{|ccc}
\boxed{\pi} &\boxed{s_1\pi} &\boxed{s_1\pi} \\
\ 
\end{array}
$$
In this case 
the construction of $E_\mu$ as $E_\mu = t^{\frac12\ell(v^{-1}_\mu)} \tau^\vee_{u_\mu} \mathbf{1}$
in \cite[Proposition 5.7]{GR21} is
$$E_{(k+h,k)} = t^{-\frac12} (\tau^\vee_\pi)^{2k} (\tau^\vee_\pi \tau_1^\vee)^{h-1} 
\tau^\vee_\pi \mathbf{1}
\quad\hbox{and}\quad
E_{(k,k+h)} =  (\tau^\vee_\pi)^{2k} (\tau_1^\vee \tau^\vee_\pi)^h  \mathbf{1},
\qquad\hbox{with $\tau^\vee_\pi = g^\vee$.}
$$
Let $h\in \ZZ_{>0}$.  The nonattacking fillings and words for $E_{(h,0)}$ and $E_{(0,h)}$ are
$$
\begin{matrix}
\begin{array}{c|ccccc}
1 &1 &i_1 &\cdots &i_{h-1} &i_h \\
2
\end{array}
\\
\\
x_1x_{i_2}\cdots x_{i_h}
\end{matrix}
\qquad\hbox{and}\qquad
\begin{matrix}
\begin{array}{c|cccc}
1 \\
2  &i_1 &\cdots &i_{h-1} &i_h \\
\end{array}
\\
\\
x_{i_1}\cdots x_{i_h}
\end{matrix}
\qquad\quad\hbox{with $i_1, \ldots, i_h\in \{1,2\}$.}
$$
%
%


\section{Additional examples}

\subsubsection{Formulas for $E_\mu$ when $n=2$.}

\begin{align*}
E_{(0,0)} &=1, \\
E_{(1,0)} &=x_1, \\
E_{(0,1)} &=x_2+\Big(\frac{1-t}{1-qt}\Big)x_1, \\
E_{(1,1)} &=x_1x_2, \\
E_{(2,0)} &= x_1^2 + \Big(\frac{1-t}{1-qt}\Big)q x_1x_2, \\
E_{(0,2)} &= x_2^2 + \Big(\frac{1-t}{1-q^2 t}\Big) x_1^2  
+ \Big( \Big(\frac{1-t}{1-qt}\Big)
+ \Big(\frac{1-t}{1-q^2 t}\Big)\Big(\frac{1-t}{1-qt}\Big)q \Big) x_1x_2, \\
E_{(3,0)} &= x_1^3 + \Big(\frac{1-t}{1-q^2t}\Big)q^2 x_1x_2^2
+ \Big( \Big(\frac{1-t}{1-qt}\Big)q
+ \Big(\frac{1-t}{1-q^2 t}\Big)\Big(\frac{1-t}{1-qt}\Big)q^2 \Big) x_1^2x_2.
\end{align*}
Then \cite[(6.2.7) and (6.28)]{Mac03} provides the general formula as follows.
Let
$$(x;q)_\infty = (1-x)(1-xq)(1-xq^2) \cdots,
\quad
(x;q)_r = \frac{(x;q)_\infty} { (q^rx;q)_\infty},
\quad\hbox{and}\quad
\genfrac[]{0pt}{0}{s}{r} = \frac{(q;q)_s}{(q;q)_r(q;q)_{s-r}}.$$
Let $k\in \ZZ_{>0}$ and let $t = q^k$.  Then
\begin{align*}
E_{(0,m)} 
&= \genfrac[]{0pt} { 0 } {k+m}{m}^{-1}\sum_{i+j=m} \genfrac[]{0pt} {0} {k+i-1} {i}
\genfrac[]{0pt}{0}{k+j}{j}x_1^j x_2^i \\
E_{(m+1,0)} 
&= \genfrac[]{0pt}{0}{k+m}{m}^{-1}\sum_{i+j=m} \genfrac[]{0pt} {0} {k+i-1}{i}
\genfrac[] {0pt} {0} {k+j} {j}q^i x_1^{j+1}x_2^i
\end{align*}
Since $t=q^k$, it appears that $t$ must be a power of $q$.  But this is not really the case
since we may rewrite these formulas using
\begin{align*}
\genfrac[]{0pt} { 0 } {k+m}{m} 
&= \frac{(q;q)_{k+m}}{(q;q)_m(q;q)_k}
=\frac{(q;q)_\infty (q^m;q)_\infty (q^k;q)_\infty }{(q^{k+m};q)_\infty (q;q)_\infty (q;q)_\infty}
=\frac{ (q^m;q)_\infty (t;q)_\infty }{(tq^{m};q)_\infty (q;q)_\infty }
=\frac{ (t;q)_m }{ (q;q)_m }
\end{align*}
and
\begin{align*}
\genfrac[]{0pt} {0} {k+i-1} {i}
\genfrac[]{0pt}{0}{k+j}{j}
& =\frac{ (q^i;q)_\infty (tq^{-1};q)_\infty }{(tq^{i-1};q)_\infty (q;q)_\infty }
\frac{ (q^j;q)_\infty (t;q)_\infty }{(tq^{j};q)_\infty (q;q)_\infty }
 =\frac{ (q^i;q)_\infty (q^j;q)_\infty (tq^{-1};q)_\infty (t;q)_\infty }
 {(q;q)_\infty  (q;q)_\infty (tq^{i-1};q)_\infty (tq^{j};q)_\infty }
\end{align*}


\subsubsection{Some small $E_\mu$ for $n=3$.}

\begin{align*}
E_{(0,0,0)} &=1, \\
E_{(1,0,0)} &=x_1, \\
E_{(0,1,0)} &=x_2+\Big(\frac{1-t}{1-qt^2}\Big)x_1, \\
E_{(0,0,1)} &=x_3+\Big(\frac{1-t}{1-qt}\Big)(x_2+x_1) \\
E_{(1,1,0)} &=x_1x_2, \\
E_{(1,0,1)} &= x_1x_3 + \Big(\frac{1-t}{1-qt^2}\Big) x_1x_2, \\
E_{(0,1,1)} &= x_2x_3 + \Big(\frac{1-t}{1-qt}\Big) (x_1x_3+x_1x_2), \\
E_{(2,0,0)} &=x_1^2 + \Big(\frac{1-t}{1-qt}\Big)q (x_1x_3+x_1x_2), \\
E_{(2,2,0)} &= x_1^2x_2^2 + \Big(\frac{1-t}{1-qt^2}\Big)q x_1^2x_2x_3
+ \Big(\frac{1-t}{1-qt^2}\Big)q x_1x_2^2x_3,
\end{align*}
and $E_{(2,1,0)}$, $E_{(2,0,1)}$, $E_{(1,2,0)}$, $E_{(0,2,1)}$, $E_{(1,0,2)}$, $E_{(0,1,2)}$
are given in section \ref{CX210}.  Additionally,
\begin{align*}
P_{(1,0,0)} &= m_1 = x_1+x_2+x_3, \\
P_{(2,0,0)} &= m_{1^2} + \frac{(1-q^2)(1-t)}{(1-q)(1-tq)} m_2,  \\
P_{(1,1,0)} &= m_{1^2} = x_1x_2+x_1x_3+x_2x_3,
\end{align*}
where $m_\lambda = \sum_{\mu\in S_n\lambda} x^\mu$ is the monomial symmetric function
so that $m_2 = x_1^2+x_2^2+x_3^2$.

\subsubsection{$E_\lambda$ and $P_\lambda$ when $\lambda$ is a partition with 3 boxes.}

 Letting $x^gamma = x_1^{\gamma_1}\cdots x_n{\gamma_n}$ if $\gamma = (\gamma_1, \ldots, \gamma_n)$, let
$$m_\lambda = \sum_{\gamma\in S_n\lambda} x^{\gamma},
\quad\hbox{be the monomial symmetric function (orbit sum).}
$$

\begin{prop}  Let $\varepsilon_i = (0, \ldots, 0, 1, 0, \ldots, 0)$ where the $1$ appears
in the $i$th spot.  Then
\begin{align*}
E_{3\varepsilon_1} 
&= x_1^3+ \Big(\frac{1-t}{1-q^2t}\Big)q^2 \sum_{k\in \{2, \ldots, n\}} x_1 x_k^2 
+ \Big(\frac{1-t}{1-qt}\Big) \Big(1+  \Big(\frac{1-t}{1-q^2t }\Big) q \Big) q
\sum_{k \in\{2, \ldots, n\} } x_1^2 x_k
\\
&\qquad 
+ \Big(\frac{1-t}{1-qt}\Big) \Big(\frac{1-t}{1-q^2t }\Big)(1 +q)q^2
\sum_{\{k, \ell\}\subseteq \{2, \ldots, n\}}  x_1 x_k x_\ell ,  
\\
E_{2\varepsilon_1+\varepsilon_2}
&= x_1^2 x_n + 
\Big( \frac{1-t}{1-qt^2 }\Big) q (x_1x_2x_n +\cdots +x_1x_2x_4 + x_1x_2x_3), \\
E_{\varepsilon_1+\varepsilon_2+\varepsilon_3} &= x_1x_2x_3, \\
P_{3\varepsilon_1} &= m_3+ \frac{(1-q^3)}{(1-tq^2)} \Big(\frac{1-t}{1-q}\Big) m_{21} 
+ \frac{(1-q^3)}{(1-tq^2)}\frac{(1-q^2)}{(1-tq)} \Big(\frac{1-t}{1-q}\Big)^2 m_{1^3},
\\
P_{2\varepsilon_1+\varepsilon_2} &= m_{21}
+ \Big( \frac{(1-t^2)}{(1-qt)}\frac{(1-q^2t)}{(1-qt^2)} + \frac{(1-t)}{(1-q)} \frac{(1-q^2)}{(1-qt)} \Big)
m_{1^3}, \\
P_{\varepsilon_1+\varepsilon_2+\varepsilon_3} &= m_{1^3} = e_3,
\ \hbox{where $e_r$ denotes the elementary symmetric function.}
\end{align*}
\end{prop}
\begin{proof}
From \cite[Proposition 3.5(b)]{GR21},
\begin{align*}
E_{2\varepsilon_n} 
&= x_n^2+ \Big(\frac{1-t}{1-q^2t}\Big) \sum_{k\in \{1, \ldots, n-1\}}  x_k^2 
+ \Big(\frac{1-t}{1-qt}\Big) \Big(1+  \Big(\frac{1-t}{1-q^2t }\Big) q \Big) 
\sum_{k \in\{1, \ldots, n-1\} } x_k x_n 
\\
&\qquad 
+ \Big(\frac{1-t}{1-qt}\Big) \Big(\frac{1-t}{1-q^2t }\Big)(1 +q)
\sum_{\{k, \ell\}\subseteq \{1, \ldots, n-1\}}  x_k x_\ell ,  
\end{align*}
and applying \cite[Proposition 5.8(c)]{GR21}
gives the formula for $E_{3\varepsilon_1} = E_{\pi 2\varepsilon_n}$.
Similarly, from \cite[Proposition 3.5(c)]{GR21},
\begin{align*}
E_{\varepsilon_1+\varepsilon_n }
&= x_1x_n + 
\Big( \frac{1-t}{1-qt^2 }\Big)  (x_1x_{n-1} +\cdots +x_1x_3 + x_1x_2),
\end{align*}
and applying \cite[Proposition 5.8(c)]{GR21} gives the formula for 
$E_{2\varepsilon_1+\varepsilon_2} =  
E_{\pi(\varepsilon_1+\varepsilon_n)} $ in the statement.
The formula for $E_{\varepsilon_1+\varepsilon_2+\varepsilon_3}$ follows from the first
statement of Proposition \ref{onecolumn}.

For $r\in \ZZ_{\ge0}$ and $\mu\in \ZZ_{\ge 0}^n$ define 
$$(x;q)_r = (1-x)(1-xq)(1-xq^2)\cdots (1-xq^{r-1})
\quad\hbox{and}\quad
(x;q)_\mu = (x;q)_{\mu_1}\cdots (x;q)_{\mu_n}$$
(when $r=0$ then $(x;q)_0 = 1$). 
As proved in \cite[Ch.\ VI equation (4.9) and Ch.\ VI \S 2 Ex.\ 1]{Mac},
if $r\in \ZZ_{>0}$ then
$$P_{\varepsilon_1+\cdots+\varepsilon_r} = e_r = m_{1^r}
\qquad\hbox{and}\qquad
P_{r\varepsilon_1}  = \sum_{\vert \mu\vert =r} 
\frac{(q;q)_r}{(t;q)_r}\frac{(t;q)_\mu}{(q;q)_\mu} m_\mu.$$
By \cite[Ch.\ VI (4.3) and (4.10)]{Mac}, 
the formula for $P_{2\varepsilon_1+\varepsilon_2}$ follows from the
formula for $P_{(2,1,0)}$ in 3 variables given at the end of section \ref{CX210}.
\end{proof}

\subsubsection{Macdonald polynomials $E^z_\mu$ and $P_\mu$ when $\mu$ is a single column.}


\begin{prop}  \label{onecolumn} 
Let $r\in \{1, \ldots, n\}$ and let $\omega_r = \varepsilon_1+\cdots+\varepsilon_r$.
$$E_{\varepsilon_1+\cdots+\varepsilon_r} = x_1x_2 \cdots x_i.$$
Let $W^{\omega_r}$ be the set of $z\in S_n$ such that $z$
is the minimal length element of its coset $z(S_r\times S_{n-r})$ in $S_n$.
If $z\in W^{\omega_r}$ then
$$z = \begin{pmatrix} 1 &2 &\cdots &r &r+1 &\cdots &n \\
i_1 &i_2 &\cdots &i_r &j_1 &\cdots &j_{n-r} \end{pmatrix}
\quad\hbox{with}\quad
\begin{array}{l}
\hbox{$i_1<i_2<\cdots < i_r$ and} \\
\hbox{$j_1< j_2<\cdots <j_{n-r}$}
\end{array}
$$
and
$$t^{\frac12\ell(z)} T_z E_{\omega_r} = x_{i_1}\dots x_{i_r}
\qquad\hbox{and}\qquad
P_{\omega_r} = \sum_{z\in W^{\omega_r}} t^{\frac12\ell(z)} T_z E_{\omega_r} = e_r,$$
is the $r$th elementary symmetric function.
\end{prop}
\begin{proof}
Since
$$v^{-1}_{\varepsilon_1+\cdots+\varepsilon_r} 
= \begin{pmatrix} 1 &\cdots &r &{r+1}&\cdots &n \\
r+1 &\cdots &n &1 &\cdots &r
\end{pmatrix}
\qquad\hbox{with}\quad
\ell(v^{-1}_{\varepsilon_1+\cdots+\varepsilon_r}) = (n-r)r,
$$
and $u_{\varepsilon_1+\cdots+\varepsilon_r} = \pi^r$ then 
$$E_{\varepsilon_1+\cdots+\varepsilon_r} 
= t^{-\frac12(n-r)r} (\tau_\pi^\vee)^r\mathbf{1} 
= t^{-\frac12(n-r)r} (g^\vee)^r \mathbf{1}
= t^{-\frac12(n-r)r}  X_1\cdots X_r  T_{v^{-1}_{\varepsilon_1+\cdots+\varepsilon_r}} \mathbf{1},
= x_1\cdots x_r.$$
A reduced word for $z$ is $z = (s_{i_1-1}\cdots s_1)(s_{i_2-1}\cdots s_2)\cdots (s_{i_r-1}\cdots s_r)$.
Then
\begin{align*}
&t^{\frac12\ell(z)}T_zE_{\omega_r}
=( (t^{\frac12}T_{i_1-1})\cdots (t^{\frac12}T_1))\cdot
( (t^{\frac12}T_{i_2-1})\cdots (t^{\frac12}T_2))\cdots
( (t^{\frac12}T_{i_r-1})\cdots (t^{\frac12}T_r))(x_1x_2\cdots x_r) \\
&=( (t^{\frac12}T_{i_1-1})\cdots (t^{\frac12}T_1))\cdot
( (t^{\frac12}T_{i_2-1})\cdots (t^{\frac12}T_2))\cdots
( (t^{\frac12}T_{i_{r-1}-1})\cdots (t^{\frac12}T_{r-1}))(x_1x_2\cdots x_{r-1}x_{i_r}) \\
&=( (t^{\frac12}T_{i_1-1})\cdots (t^{\frac12}T_1))\cdot
( (t^{\frac12}T_{i_2-1})\cdots (t^{\frac12}T_2))\cdots
( (t^{\frac12}T_{i_{r-2}-1})\cdots (t^{\frac12}T_{r-2}))(x_1x_2\cdots x_{r-2} x_{i_{r-1}} x_{i_r}) \\
&= \cdots = x_{i_1}x_{i_2}\cdots x_{i_r}.
\end{align*}
The last equality then follows from \eqref{Pdefn}.
\end{proof}

\subsubsection{$E^z_\mu$ for a single box}

\begin{prop}  Let $j\in \{1, \ldots, n\}$ and let $z\in S_n$. 
Then
$$E^z_{\varepsilon_j} 
= c_jx_{z(j)}
+\cdots+c_2x_{z(2)}+c_1x_{z(1)}
$$
where
$$
c_a = \begin{cases}
\displaystyle{ \Big(\frac{1-t}{1-qt^{n-j+1}}\Big) q t^{C(a)},}  &\hbox{if $z(j)<z(a)$,} \\
\displaystyle{ \Big(\frac{1-t}{1-qt^{n-j+1}}\Big)  t^{C(a)},} &\hbox{if $z(j)>z(a)$,} \\
1, &\hbox{if $z(j)=z(a)$.}
\end{cases}
$$
with
$$C(a) = \begin{cases}
\{ k\in \{j+1, \ldots, n\}\ |\ \hbox{$z(k)<z(j)<z(a)$ or $z(j)<z(a)<z(k)$}\}, &\hbox{if $z(j)<z(a)$,} \\ 
\{ k\in \{j+1, \ldots, n\}\ |\ z(j)>z(k)>z(a)\}, &\hbox{if $z(j)>z(a)$,} 
\end{cases}
$$
\end{prop}
\begin{proof}
The proof is by induction on $\ell(z)$.  If $z=1$ then $T_z=1$ and the formula is the same as
given in \cite[Proposition 3.5(a)]{GR21} for $E_{\varepsilon_j}.$
Let $r\in \{1, \ldots, n-1\}$ such that $s_rz>z$.
Recall
\begin{equation}
t^{\frac12}T_r(x_\ell)
= \begin{cases}
x_{r+1}, &\hbox{if $\ell = r$,} \\
tx_r + (t-1)x_{r+1}, &\hbox{if $\ell = r+1$,}  \\
tx_\ell, &\hbox{otherwise.}
\end{cases}
\end{equation}
\begin{equation}
t^{-\frac12}T_r(x_\ell)
= \begin{cases}
t^{-1} x_{r+1}, &\hbox{if $r = \ell$,} \\
x_r + (1 - t^{-1} )x_{r+1}, &\hbox{if $\ell = r+1$,}  \\
x_\ell, &\hbox{otherwise.}
\end{cases}
\end{equation}
Write
$$t^{-\frac12(\ell(zv^{-1}_{\varepsilon_j}) - \ell(v^{-1}_{\varepsilon_j})}
E^z_{\varepsilon_j} 
= \sum_{i=1}^n c^z_i x_{z(i)}.$$
Then
\begin{align*}
t^{\frac12}T_r(c^z_a x_r + c^z_b x_{r+1})
&= c^z_a x_{r+1} + c^z_b(tx_r+(t-1)x_{r+1}) 
= tc^z_b x_r + (c^z_b(t-1)+c^z_a)x_{r+1},
\\
&\hbox{giving $c^{s_rz}_a = tc^z_b$ and $c^{s_rz}_b = c^z_b(t-1)+c^z_a$.}
\end{align*}
\begin{align*}
t^{-\frac12}T_r(c^z_a x_r + c^z_b x_{r+1})
&= t^{-1}c^z_a x_{r+1} + c^z_b(x_r+(1-t^{-1})x_{r+1}) 
= c^z_b x_r + (c^z_b(1-t^{-1})+t^{-1}c^z_a)x_{r+1}
\\
&\hbox{giving $c^{s_rz}_a = c^z_b$ and $c^{s_rz}_b = c^z_b(1-t^{-1})+t^{-1}c^z_a$.}
\end{align*}
Let 
$$\hbox{ $a = z^{-1}(r)$ and $b = z^{-1}(r+1)$
\quad so that\quad
$b = (s_rz)^{-1}(r)$ and $a = (s_rz)^{-1}(r+1)$.}
$$
Assume $s_rz>z$ so that $a<b$.
$$\begin{array}{llllllll}
\hbox{(lll)} &z(j)<r &a<j &b<j &c^z_a=c^z_b &\hbox{multiply by $t^{-\frac12}T_r$} 
&c^{s_rz}_a = c^{s_z}_b = c^z_a \\
\hbox{(llg)} &z(j)<r &a<j &b>j &c^z_b=0 &\hbox{multiply by $t^{-\frac12}T_r$} 
&c^{s_rz}_a = 0,\ c^{s_rz}_b = t^{-1}c^z_a \\
\hbox{(lgg)} &z(j)<r &a>j &b>j &c^z_a=c^z_b=0 &\hbox{multiply by $t^{-\frac12}T_r$} 
&c^{s_rz}_a = c^{s_rz}_b = 0 \\
\hbox{(ele)} &z(j)=r &a=j &b>j &c^z_a=1,\ c^z_b=0 &\hbox{multiply by $t^{\frac12}T_r$} 
&c^{s_rz}_a=0,\ c^{s_rz}_b = 1, \\
\hbox{(flf)} &z(j)=r+1 &a<j &b=j &c^z_b=1 &\hbox{multiply by $t^{-\frac12}T_r$}  \\
\hbox{(gll)} &z(j)>r+1 &a<j &b<j &c^z_a=c^z_b &\hbox{multiply by $t^{-\frac12}T_r$}  
&c^{s_rz}_a = c^{s_z}_b = c^z_a \\
\hbox{(glg)} &z(j)>r+1 &a<j &b>j &c^z_b=0 &\hbox{multiply by $t^{-\frac12}T_r$}  
&c^{s_rz}_a = 0,\ c^{s_rz}_b = t^{-1}c^z_a \\
\hbox{(ggg)} &z(j)>r+1 &a>j &b>j &c^z_a=c^z_b=0 &\hbox{multiply by $t^{-\frac12}T_r$}  
&c^{s_rz}_a = c^{s_rz}_b = 0 
\end{array}
$$
Now we need to show that the statistics $C(a)$ provide the same recursions.

In the case
(flf), $r+1=z(j)>z(a)=r$ with $C(a) = 0$ and
$r=(s_rz)(j)< (s_rz)(a)=r+1$ and $C(a)=n-j$.  So
$$c^z_j = 1,\ \ c^z_a = \Big(\frac{1-t}{1-qt^{n-j+1}}\Big) t^0
\quad\hbox{and}\qquad
c^{s_rz}_j = 1,\ \ c^{s_rz}_a = \Big(\frac{1-t}{1-qt^{n-j+1}}\Big)qt^{n-j}
$$
since
\begin{align*}
c^{s_rz}_a &= (1-t^{-1}) + t^{-1} \Big(\frac{1-t}{1-qt^{n-j+1}}\Big) t^0 \\
&=\Big(\frac{1-t}{1-qt^{n-j+1}}\Big) ( -t^{-1}(1-qt^{n-j+1})+t^{-1})
=\Big(\frac{1-t}{1-qt^{n-j+1}}\Big) qt^{n-j}.
\end{align*}
\end{proof}

Some examples are
\begin{align*}
(t^{\frac12} T_{i+(k-1)}) & \cdots (t^{\frac12} T_i) E_{\varepsilon_i}  
=x_{i+k} + \frac{(1-t)}{(1-qt^{n-(i-1)})} t^k (x_{i-1}+\cdots + x_1), \\
(t^{-\frac12} T_{i-k})& \cdots (t^{-\frac12} T_{i-1}) E_{\varepsilon_i} \\
&=  x_{i-k} + \frac{(1-t)}   {(1-qt^{n-(i-1)})} 
\Big( q t^{n-i}(x_i+x_{i-1}+\cdots+x_{i-(k-1)})  + (x_{i-(k+1)}+\cdots + x_1)\Big), \\
\end{align*}

\subsubsection{The nonattacking fillings for $E_{\varepsilon_i}$.}

The box greedy reduced word for $u_{\varepsilon_i}$ is
$$
u^\square_{\varepsilon_i} = 
\begin{array}{c|c}
\boxed{\phantom{T}} \\
\vdots \\
\boxed{\phantom{T}} &\boxed{s_{i-1}\cdots s_1\pi} \\
\vdots \\
\boxed{\phantom{T}}
\end{array}
\qquad\hbox{with $i$ non-attacking fillings,}\qquad
\begin{matrix}
\begin{array}{c|c}
1 \\
\vdots \\
i &i   \\
\vdots \\
n
\end{array}
\qquad
&\begin{array}{c|c}
1 \\
\vdots \\
i&k  \\
\vdots \\
n
\end{array}
\\
&i>k
\\
&\big( \frac{1-t}{1-qt^{n-(i-1)}} \big)
\end{matrix}
$$

\subsubsection{The nonattacking fillings for $E^z_{\varepsilon_i}$.}
If $z(i) = i+k$ then the $i$ non-attacking fillings are
$$
\begin{matrix}
\begin{array}{c|c}
z(1) \\
\vdots \\
i+k &i+k   \\
\vdots \\
z(n)
\end{array}
\qquad
&\begin{array}{c|c}
z(1) \\
\vdots \\
i+k &j  \\
\vdots \\
z(n)
\end{array}
\\
&i>j \ge 1
\\
t^{-k}
&\big( \frac{1-t}{1-qt^{n-(i-1)}} \big)
\end{matrix}
$$
If $z(i) = i-k$ then the $i$ non-attacking fillings are
$$
\begin{matrix}
\begin{array}{c|c}
z(1) \\
\vdots \\
i-k  &j  \\
\vdots \\
z(n)
\end{array}
&\begin{array}{c|c}
z(1) \\
\vdots \\
i-k &i-k   \\
\vdots \\
z(n)
\end{array}
\qquad
&\begin{array}{c|c}
z(1) \\
\vdots \\
i-k &j  \\
\vdots \\
z(n)
\end{array}
\\
i\ge j>i-k & &i-k>j\ge 1
\\
\big( \frac{(1-t) qt^{n-i} }{1-qt^{n-(i-1)}} \big)
&1
&\big( \frac{1-t}{1-qt^{n-(i-1)}} \big)
\end{matrix}
$$

\subsubsection{The nonattacking fillings for $E_{2\varepsilon_i}$}
The box greedy reduced word for $u_{2\varepsilon_i}$ is
$$
u^\square_{2\varepsilon_i} = (s_{i-1}\cdots s_1\pi)(s_{n-1}\cdots s_1\pi) = \quad
\begin{array}{c|cc}
\boxed{\phantom{T}} \\
\vdots \\
\boxed{\phantom{T}} &\boxed{s_{i-1}\cdots s_1\pi}  &\boxed{s_{n-1}\cdots s_1\pi} \\
\vdots \\
\boxed{\phantom{T}}
\end{array}
$$
The case $E_{2\varepsilon_i}$ has $i\cdot n$ nonattacking fillngs and $2^{n+i-2}$ alcove walks.
There are no covid triples for any of the nonattacking fillings so that
$t^{covid(T)} = t^0=1$, and $q^{maj(T)}=q^1=q$ exactly when $T(i,1)<T(i,2)$.
$$
\begin{matrix}
\begin{array}{c|cc}
1 \\
\vdots \\
i &i &i  \\
\vdots \\
n
\end{array}
\qquad
&\begin{array}{c|cc}
1 \\
\vdots \\
i &k &k  \\
\vdots \\
n
\end{array}
\qquad
&\begin{array}{c|cc}
1 \\
\vdots \\
i &i &\ell  \\
\vdots \\
n
\end{array}
\qquad
&\begin{array}{c|cc}
1 \\
\vdots \\
i &i &k  \\
\vdots \\
n
\end{array}
\qquad
&\begin{array}{c|cc}
1 \\
\vdots \\
i &k  &i  \\
\vdots \\
n
\end{array}
\\
&k<i
&\ell>i
&k<i
&k<i
\\
&\big( \frac{1-t}{1-q^2t^{n-(i-1)}} \big)
&\Big(\frac{1-t}{1-qt}\Big)
q
&\Big(\frac{1-t}{1-qt}\Big)
&
\Big(\frac{1-t}{1-q^2t^{n-(i-1)}} \Big)
\Big(\frac{1-t}{1-qt}\Big)
q
\end{matrix}
$$

$$
\begin{matrix}
\begin{array}{c|cc}
1 \\
\vdots \\
i &k &\ell  \\
\vdots \\
n
\end{array}
\quad
&\begin{array}{c|cc}
1 \\
\vdots \\
i &\ell &k  \\
\vdots \\
n
\end{array}
\qquad
&\begin{array}{c|cc}
1 \\
\vdots \\
i &k &\ell  \\
\vdots \\
n
\end{array}
\\
k<i,\ \ell>i  \quad
&\{ k, \ell\} \subseteq \{1, \ldots, i-1\} \quad 
& \{ k, \ell\} \subseteq \{1, \ldots, i-1\}
\\
\Big(\frac{1-t}{1-q^2t^{n-(i-1)}} \Big)
\Big(\frac{1-t}{1-qt}\Big)
q
&\big(\frac{1-t}{1-q^2t^{n-(i-1)}}\big)
\big(\frac{1-t}{1-qt}\big)
&\big(\frac{1-t}{1-q^2t^{n-(i-1)}}\big)
\big( \frac{1-t}{1-qt} \big)
q
\end{matrix}
$$

\subsubsection{The nonattacking fillings for $E_{\varepsilon_{j_1}+\varepsilon_{j_2}}$.}
Let $j_1, j_2\in \{1, \ldots, n\}$ with $j_1<j_2$.
The box greedy reduced word for $u_{\varepsilon_{j_1}+\varepsilon_{j_2}}$ is
$$u^\square_{\varepsilon_{j_1}+\varepsilon_{j_2}} = 
\begin{array}{c|c}
\boxed{\phantom{T}} \\
\vdots \\
\boxed{\phantom{T}} &\boxed{s_{j_1-1}\cdots s_1\pi} \\
\vdots \\
\boxed{\phantom{T}} &\boxed{s_{j_2-2}\cdots s_1\pi} \\
\vdots \\
\boxed{\phantom{T}}
\end{array}
$$
$E_{\varepsilon_{j_1}+\varepsilon_{j_2}}$ has $j_1(j_2-1)$ nonattacking fillings and $2^{j_1-1}2^{j_2-2}$ alcove walks.
$$
\begin{matrix}
\begin{array}{c|c}
1 \\
\vdots \\
j_1  &j_1 \\
\vdots \\
j_2 &j_2 \\
\vdots \\
n
\end{array}
\qquad
&
\begin{array}{c|c}
1 \\
\vdots \\
j_1  &k \\
\vdots \\
j_2 &j_2 \\
\vdots \\
n
\end{array}
\qquad
&
\begin{array}{c|c}
1 \\
\vdots \\
j_1  &j_1 \\
\vdots \\
j_2 &\ell \\
\vdots \\
n
\end{array}
\qquad
&
\begin{array}{c|c}
1 \\
\vdots \\
j_1 &k \\
\vdots \\
j_2 &j_1 \\
\vdots \\
n
\end{array}
\qquad
&
\begin{array}{c|c}
1 \\
\vdots \\
j_1 &j_1 \\
\vdots \\
j_2 &k \\
\vdots \\
n
\end{array}
\\
&1\le k\le j_1-1
\quad &j_1+1\le \ell\le j_2-1
\quad &1\le k\le j_1-1
\quad &1\le k\le j_1-1
\\
&\big(\frac{1-t}{1-qt^{n-j_1}} \big)
&\big(\frac{1-t}{1-qt^{n-(j_2-2)}} \big)
&\big(\frac{1-t}{1-qt^{n-j_1}} \big)
\big(\frac{1-t}{1-qt^{n-(j_2-2)}} \big)
&t \big(\frac{1-t}{1-qt^{n-(j_2-2)}} \big)
\end{matrix}
$$
$$
\begin{matrix}
\begin{array}{c|c}
1 \\
\vdots \\
j_1 &k \\
\vdots \\
j_2 &\ell \\
\vdots \\
n
\end{array}
\qquad
&\begin{array}{c|c}
1 \\
\vdots \\
j_1 &k \\
\vdots \\
j_2 &\ell \\
\vdots \\
n
\end{array}
\qquad
&
\begin{array}{c|c}
1 \\
\vdots \\
j_1 &\ell \\
\vdots \\
j_2 &k \\
\vdots \\
n
\end{array}
\\
\begin{matrix}
k\in \{1, \ldots, j_1-1\} \\
\ell\in \{j_1+1, \ldots, j_2-1\}
\end{matrix}
\quad &\{ k, \ell\} \subseteq \{1, \ldots, j_1-1\}
\quad &\{ k, \ell\} \subseteq \{1, \ldots, j_1-1\}
\\
\big(\frac{1-t}{1-qt^{n-(j_2-2)}} \big)
\big(\frac{1-t}{1-qt^{n-j_1}} \big)
&\big(\frac{1-t}{1-qt^{n-(j_2-2)}} \big)
\big(\frac{1-t}{1-qt^{n-j_1}} \big)
&\big(\frac{1-t}{1-qt^{n-(j_2-2)}} \big)
\big(\frac{1-t}{1-qt^{n-j_1}} \big)t?
\end{matrix}
$$

\section{Queue tableaux}

\subsubsection{An instance of compression of NAFs -- Motivation for Queue Tableaux.}\label{QTmotiv}
In \cite[Proposition 3.5(c)]{GR21}, if $j_1 = j_2-1$ then the third and fifth summands disappear to give
\begin{align*}
E_{\varepsilon_{j_2-1}+\varepsilon_{j_2} }
&= x_{j_2-1}x_{j_2} 
+ \Big( \frac{1-t}{1-qt^{n-(j_2-1)} } \Big)  \sum_{k=1}^{j_2-2} x_k x_{j_2} 
+ \Big( \frac{1-t}{1-qt^{n-(j_2-2)} }\Big) \Big(\frac{1-t}{1-qt^{n-(j_2-1)} }+t\Big) 
\sum_{k = 1}^{j_2-2} x_k x_{j_2-1} \\
&\qquad + \Big(\frac{1-t}{1-qt^{n-(j_2-2)} } \Big)
\Big(  \frac{1-t}{1-qt^{n-j_1} } \Big)(1+t)
\sum_{\{k,\ell\}\subseteq  \{1,\ldots, j_2-2\}}   x_k x_\ell \\
&= x_{j_2-1}x_{j_2} 
+ \Big( \frac{1-t}{1-qt^{n-(j_2-1)} } \Big)  \sum_{k=1}^{j_2-2} x_k x_{j_2} 
+ \Big( \frac{1-t}{\cancel{ 1-qt^{n-(j_2-2)} } }\Big) 
\Big(\frac{ \cancel{1-qt^{n-(j_2-2)} } }{1-qt^{n-(j_2-1)} }\Big) 
\sum_{k = 1}^{j_2-2} x_k x_{j_2-1} \\
&\qquad + \Big(\frac{1-t}{1-qt^{n-(j_2-2)} } \Big)
\Big(  \frac{1-t}{1-qt^{n-(j_2-1)} } \Big)(1+t)
\sum_{\{k,\ell\}\subseteq  \{1,\ldots, j_2-2\}}   x_k x_\ell 
\end{align*}
which is an example of the additional cancellation that occurs when there are 
adjacent rows of equal length and illustrates the
the difference between nonattacking fillings and queue tableaux.

\subsubsection{Queue tableaux}

Following (and slightly generalizing) \cite[Definition A.1]{CMW18},
a \emph{queue tableau} of shape $(z,\mu)$ is a nonattacking filling $T$ of $(z,\mu)$
such that 
\begin{enumerate}
\item[(QT)] If $\mu_i=\mu_{i-1}=\cdots =\mu_{i-r}$ then $T(i,j)\not\in \{T(i-1,j-1), \ldots, T(i-r,j-1)\}$.
\end{enumerate}
If the parts of $\mu$ are distinct then a queue tableau is no different than a nonattacking filling.
More generally, if $\mu_i\ne \mu_{i+1}$ for $i\in \{1, \ldots, n-1\}$ then a queue tableau is no
different than a nonattacking filling.

\subsubsection{Multiline queues}\label{MLQs}

The \emph{multiline queue} corresponding to a queue tableau $T$ is the pipe dream $P$
corresponding to $T$ under the map given in \eqref{fillingtopipedream}, namely
\begin{equation*}
P(k,j) = i \qquad\hbox{if and only if}\qquad T(i,j) = k,
\end{equation*}
The example in \cite[Figures 3 and 12]{CMW18} has
\begin{equation*}
\hbox{queue tableau}\quad T=
\begin{array}{c|cccc}
6 &6 &5 &3 \\
1 &1 &6 \\
2 &2 &2 \\
7 &7 &4 \\
8 &8 \\
3 \\
4 \\
5
\end{array}
\quad\hbox{and pipe dream}
\quad
P= \left(\begin{array}{c|ccc}
2 &2 &0 &0 \\
3 &3 &3 &0 \\
6 &0 &0 &1 \\
7 &0 &4 &0 \\
8 &0 &1 &0 \\
1 &1 &2 &0 \\
4 &4 &0 &0 \\
5 &5 &0 &0
\end{array}
\right)
\end{equation*}
The picture of this pipe dream from \cite[Figures 3]{CMW18} is
\begin{equation*}
\hbox{the multiline queue}\qquad
\vcenter{\hbox{\includegraphics[scale=0.4, angle=270]{CMWFigure3.png}}} 
\label{CMWmultilinequeue}
\end{equation*}

\subsubsection{Compression not captured by NAFs or QT}

Let 
$\mathrm{AW}_\mu = \mathrm{AW}^{\id}_\mu$,
$\mathrm{NAF}_\mu = \mathrm{NAF}^{\id}_\mu$,
and $\mathrm{QT}_\mu = \mathrm{QT}^{\id}_\mu$.
The example
$$
\#\mathrm{AW}_{(2,2,1,1,0,0)} = 16, \quad
\#\mathrm{NAF}_{(2,2,1,1,0,0)} = 9
\quad\hbox{and}\quad
\#\mathrm{QT}_{(2,2,1,1,0,0)} = 7.
$$
is provided in \cite[Figure 4]{CMW18}).  The equalities (see (see \cite[Proposition 5.8]{GR21}) \begin{align*}
E_{(2,0,1)}(x_1,x_2,x_3;q,t) &= (x_1x_2x_3)^2E_{(1,2,0)}(x_3^{-1},x_2^{-1},x_1^{-1};q,t), 
\quad\hbox{and} \\
E_{(2,2,0)}(x_1,x_2,x_3;q,t) &= q^{-1}E_{(2,0,1)}(x_3,x_1,x_2;q,t)
\end{align*}
indicate that if one provides a formula for $E_{(1,2,0)}$ then there are
formulas for  $E_{(2,0,1)}$ and $E_{(2,2,0)}$
with exactly the same number of terms.  For these cases,
$$\#\mathrm{AW}_{(1,2,0)} = 4, \quad
\#\mathrm{NAF}_{(1,2,0)} = 3,\quad
\#\mathrm{QT}_{(1,2,0)} = 3.$$
$$\#\mathrm{AW}_{(2,0,1)} = 4, \quad
\#\mathrm{NAF}_{(2,0,1)} = 4,\quad
\#\mathrm{QT}_{(2,0,1)} = 4.$$
$$\#\mathrm{AW}_{(2,2,0)} = 4, \quad
\#NAF_{(2,2,0)} = 4,\quad
\#\mathrm{QT}_{(2,2,0)} = 3.$$
Thus $\mu=(2,0,1)$ is a case where 
possible compression is not realized by either the NAFs or the QT.

\subsubsection{Comparing $\#\mathrm{NAF}$ and $\#\mathrm{QT}$ for 
$(r,0,\ldots,0)$ and $(r,\ldots, r, 0)$.}

Since $u_{(r,0,\ldots,0)} = \pi(s_{n-1}\cdots s_1\pi)^{r-1}$ 
and $u_{(r,r,\ldots, r,0)}=\pi^{n-1}(s_1\pi)^{(n-1)(r-1)}$ then
$$
\begin{array}{lll}
\#\mathrm{AW}_{(r,0,0, \ldots, 0)} = (2^{n-1})^{r-1},\qquad
&\#\mathrm{NAF}_{(r,0,0, \ldots, 0)} = n^{r-1},\qquad
&\#\mathrm{QT}_{(r,0,0, \ldots, 0)} = n^{r-1}, \\
\#\mathrm{AW}_{(r,r,\ldots, r, 0)} = (2^{n-1})^{r-1}, \qquad
&\#\mathrm{NAF}_{(r,r, \ldots, r, 0)} = (2^{n-1})^{r-1}, \qquad
&\#\mathrm{QT}_{(r,r, \ldots, r, 0)} = n^{r-1}.
\end{array}
$$
To see the last equality: In a queue tableau of
shape $(r,r,\ldots, r, 0)$, for each column after the first,
we get to choose the position of the $j\in \{1, \ldots, n\}$ that did not appear
in the column before ($n$ choices total for each column).

\end{document}